\documentclass[10pt]{article}
\usepackage[colorlinks,linkcolor=blue,citecolor=blue]{hyperref}
\usepackage{tabulary} 
\usepackage{amsthm}
\usepackage{multirow}
\usepackage{marginnote}
\usepackage{a4wide}
\usepackage{amssymb}
\usepackage{amsfonts}
\usepackage{amsmath}
\usepackage{mathrsfs}
\usepackage{tikz}
\usepackage{tikz-cd}
\usetikzlibrary{arrows,matrix}
\usetikzlibrary{positioning}
\usepackage{mdframed} % draw fram of texts
\usepackage{lipsum} % this is used for drawing frams of texts
\usepackage{extarrows} % making long equal sign (or arrows) with text under or above
\usepackage{color}
\usepackage{enumitem} 
\usepackage{tocloft} % change the style of table of contents
\usepackage{titlesec} % control the space before and after section titles

\usepackage[all]{xy} % draw diagrams

\input xy
\xyoption{arrow} \xyoption{matrix}

\date{}

\newtheorem{proposition}{Proposition}[section]
\newtheorem{theorem}[proposition]{Theorem}
\newtheorem{lemma}[proposition]{Lemma}

\newtheorem{definition}[proposition]{Definition}
\newtheorem{corollary}[proposition]{Corollary}

\def\Hom{{\rm Hom}}
\def\der{\partial }

\def\nFM0{{\nu }_{F,M_0}}
\def\nFN0{{\nu }_{F,N_0}}
\def\nGN0{{\nu }_{G,N_0}}

\def\N0{ {\bf N}_0 }

\def\t{\otimes}

\def\ra{\rightarrow}

\def\Xpm{X^{\pm }}

\def\s{\sigma}

\def\l1{{\lambda}_1}

\def\a{\alpha}
\def\a0{ {\alpha }_0}
\def\a1{ {\alpha }_1}

\def\l{\lambda}
\def\o{\omega}
\def\nFGM0{{\nu }_{F,G,M_0}}

\def\nFN0{{\nu}_{F,N_0}}

\def\sm{{\sigma}^m}

\def\sm1{{\sigma}^{-1}}

\def\smtp1{{\sigma}^{-t+1}}

\def\o{\omega }
\def\S1{S^{-1}}

\def\Xpm1{X^{\pm 1}_1}

\def\sPM1{{\sigma }^{\pm 1}}
\def\sMP1{{\sigma }^{\mp 1 }}

\def\b{\beta}
\def\d{\delta}

\def\di{{\rm d.ind}}

\def\L{\Lambda}
\def\O{\Omega}

\def\G{\Gamma}

\def\CA{{\cal A}}

\def\CD{{\cal D}}

\def\Ytm1{Y^{t-1}}
\def\Yim1{Y^{i-1}}

\def\CL{{\cal L}}
\def\CM{{\cal M}}
\def\CN{{\cal N}}

\def\CF{{\cal F}}
\def\CG{{\cal G}}

\def\supp{{\rm supp }}

\def\Aut{{\rm Aut}}
\def\bK{\overline{K}}

\def\Der{{\rm Der }}
\def\ad{{\rm ad }}
\def\dim{{\rm dim }}
\def\char{{\rm char }}

\def\ker{ {\rm ker } }

\def\gr{ {\rm gr} }

\def\D{ \Delta }
\def\SL2Z{ {\rm SL}_2({\bf Z}) }

\def\CR{ {\cal R}}

\def\th{ \theta }

\def\CL{{\cal L}}

\def\Gp1{ G^{1 , 1 } }
\def\P11{ P^{-1 , 1 } }
\def\Pp1{ P^{1 , 1 } }

\def\th{\theta}

\def\nCLsr{{}^\nu\kern-2pt {\cal L}^{\sigma , \rho  }}
\def\nP{{}^\nu \kern-2pt P}
\def\nL{{}^\nu\kern-2pt L}
\def\nLL{{}^\nu\kern-2pt \Lambda}
\def\nPsr{{}^\nu\kern-2pt P^{\sigma , \rho  }}
\def\nLsr{{}^\nu\kern-2pt L^{\sigma , \rho  }}
\def\nuCL{{}^\nu\kern-2pt  {\cal L}}
\def\nCLsr{{}^\nu\kern-2pt {\cal L}^{\sigma , \rho  }}
\def\nCL1m{{}^\nu\kern-2pt {\cal L}^{-1 , 1  }}
\def\x1nu{x^\frac{1}{\nu}}
\def\xm1nu{x^{-\frac{1}{\nu}}}

\def\CR{ {\cal R}}

\def\CN{{\cal N}}
\def\ra{\rightarrow }

\def\CB{{\cal B}}

\def\nAM0{{\nu }_{{\cal A},M_0}}
\def\nAN0{{\nu }_{{\cal A},N_0}}

\def\End{ {\rm End }}
\def\Der{ {\rm Der }}

\def\CR{ {\cal R }}

\def\ad{ {\rm ad }}

\def\gr{\mathfrak{r}}

\def\GL{{\rm GL}}
\def\SL{{\rm SL}}

\def\Hom{{\rm Hom}}

\def\di!{\frac{\der^i}{i!}}
\def\dik!{\frac{\der^k_i}{k!}}
\def\Fp{\mathbb{F}_p}

\def\gl{\mathfrak{l}}

\def\id{{\rm id}}

\def\Max{{\rm Max}}

\def\N{\mathbb{N}}

\def\0{\overline{0}}
\def\1{\overline{1}}

\def\Ln1{\L_{n,\overline{1}}}

\def\a1{a_{\overline{1}}}

\def\S{\Sigma}

\def\vn1{\overrightarrow{n-1}}

\def\im{{\rm im}}

\def\gl{{\rm gl}}
\def\sl{{\rm sl}}

\def\mD{\mathbb{D}}
\def\mL{\mathbb{L}}

\def\gf{\mathfrak{f}}

\def\mJ{\mathbb{J}}
\def\mI{\mathbb{I}}

\def\ann{{\rm ann}}

\def\K1{{\rm K}_1}

\def\hmI1{\widehat{\mI_1}}
\def\tmI1{\widetilde{\mI_1}}
\def\tmJ1{\widetilde{\mJ_1}}
\def\hB1{\widehat{B_1}}
\def\hCB1{\widehat{\CB_1}}

\def \S{\mathcal{S}}

\def\sl2{\mathfrak{sl}_2}

\def\Deg{{\rm Deg}}
\def\Ind{{\rm Ind}}

\def\sl2{\mathfrak{sl}_2}
\def\gl2{\mathfrak{gl}_2}

\def\res{{\rm res}}

\def\b1{\overline{1}}

\def\res{{\rm res}}

\def\fC{{\mathfrak{C}}}
\def\fCK{{\mathfrak{C}}(K)}

\def\fCdK{{\mathfrak{C}}_d(K)}

\def\gl{{\mathfrak{l}}}

\makeatletter
\newenvironment{proof*}[1][\proofname]{\par
  \pushQED{\qed}%
  \normalfont \partopsep=\z@skip \topsep=\z@skip
  \trivlist
  \item[\hskip\labelsep
        \itshape
    #1\@addpunct{.}]\ignorespaces
}{%
  \popQED\endtrivlist\@endpefalse
}
\makeatother

\begin{document}

\author{V. V. \  Bavula 
%(AnGaloisTh-NORMAL-Fields-NEW.tex)
}
\title{Analogue of the Galois Theory for normal fields  and B-extensions (characteristic free approach)

 %B-extensions and analogue of the Galois Theory for normal fields  (characteristic free approach)
}

\maketitle

\begin{abstract}

This paper resolves the long-standing open problem of establishing a unified Galois theory for all normal finite field extensions.

The key idea is to introduce the concept of  `B-extensions' which are the most symmetrical finite field extensions (a finite field extension $L/K$ is called a {\em B-extension}  if the endomorphism algebra $\End_K(L)$ is generated by the algebra of differential operators $\CD (L/K)$ on the $K$-algebra  $L$ and the automorphism group $G(L/K):=\Aut_{K-{\rm alg}}(L)$) and to obtain an analogue of the Galois Theory for B-extensions.  Surprisingly, the class of B-extensions coincides with the class of {\em normal } finite field extensions and hence we have a new characterization of normality via differential operators and automorphisms. As a result, an analogue of the Galois Theory is obtained for normal field extensions. In particular, all Galois field extensions and all purely inseparable field extensions are B-extensions. Our approach is a ring theoretic (characteristic free)  approach which is based on central simple algebras.
In this approach, analogues of the Galois Correspondences (for subfields and normal  subfields of $L$) are deduced from the Double Centralizer Theorem which is applied to the central simple algebra  $\End_K(L)$ and subfields of B-extensions.

  Since Galois finite field extensions are B-extensions, this approach gives a new conceptual (short) proofs of key results of the Galois Theory, see \cite{GaloisTh-RingThAp}  for details. It also reveals that the `maximal symmetry' (of field extensions) is the essence of the classical Galois Theory and  the Galois Theory for normal field extensions. 

 In the Galois theory,  the two Galois Correspondences are given in terms of the fields $L^H$ of $H$-{\em invariants} where $H$ is a subgroup of the Galois group $G(L/K)$. But for normal field extensions,  analogues of the two Galois Correspondences are given in terms of the fields  $L^H\cap L^{\CG}$ of $H$-invariants and $\CG$-{\em invariants/constants} where $\CG$ is a dominant Lie algebra and $(\CG , H)$ is a dom-group.

%*** $\CD (L/K)_+$ is the left ideal of $\CD (L/K)$ that contains precisely differential operators with zero  `constant term', i.e. that annihilate the identity element 1 of the field $L$. ***

As a particular case, we have a Galois Theory for purely inseparable finite  field extensions $L/K$ where the  Lie algebra of proper differential operators $\CD(L/K)_+$ plays the role of the Galois group $G(L/K)$ in the Galois case (which is the identity group in this case) and $H$-invariants are replaced by $\CG$-invariants/constants where $\CG\subseteq \CD (L/K)_+$ is a dominant Lie algebra.

%As a particular case, we have a `differential' Galois Theory for purely inseparable finite  field extensions $L/K$ where the algebra  of differential operators $\CD(L/K)$ plays the role of the Galois group $G(L/K)$ in the Galois case (which is the identity group in this case) and $G(L/K)$-invariants are replaced by $\CD (L/K)_+$-invariants/constants. 

The classical results on  `Galois Theory' of Jacobson (1937, 1944) for purely inseparable field extensions of exponent one and   
its generalizations for modular extensions by Sweedler (1968), and Gerstenhaber and Zaromp (1970) follow from the above as particular cases (since in these cases the rings of differential operators are uniquely determined by the derivations and higher derivations, respectively).

Explicit descriptions of the subalgebras $\CD (L/K)$, $L\rtimes G(L/K)$ and $\CD (L/K)\rtimes G(L/K)$ of $E(L/K)$ are given and their properties are studied.  These results constitute the technical core of the paper on which the analogue of Galois theory for normal fields is built, and they are interesting in their own right.

$\noindent$



{ \small \tableofcontents}

\end{abstract}

%%%%%%%%%%%%%%%%%% SECTION 1 %%%%%%%%%%%%%%%%%%%%%

\section{Introduction} \label{INTR} %\marginpar{INTR}

In this paper, module means a {\em left} module.  The following
notation will remain fixed throughout the paper (if it is not
stated otherwise): 
 
 \begin{itemize}
 
% \item $\N=\{0,1,  \ldots\}$ is the set of natural numbers, $\N_+:=\{1,2,   \ldots\}$, $p$ is a prime number and $\N_{<p}:=\{0,1, \ldots , p-1\}$, 

\item $K$ is a  field and $\bK$ is its algebraic closure,  

\item If  the field $K$ has  characteristic $p>0$,   the map  $\gf :\bK\ra \bK$, $\l \mapsto \l^p$ is called  the Frobenius monomorphism, 

\item $L/K$ is a finite field extension,
 
\item  $G:=G(L/K)$ is the automorphism group of the $K$-algebra $L$. If $L/K$ is a Galois field extension then  $G(L/K)$ is called the Galois group of $L/K$, 

\item $L^{G(L/K)}:=L^G:=\{ l\in L\, | \, \s (l)=l$ for all $\s\in G(L/K)\}$ is the field of $G(L/K)$-invariants,

\item  $\CF (L/K)$ and  $\CG (L/K)$ are the sets of all and Galois subfields of $L/K$, respectively,

\item    $\CG (G(L/K))$ and  $\CN(G(L/K))$ are  the sets of all and normal subgroups of $G(L/K)$, respectively, 

\item $\O_{L/K}$ is the module of K\"{a}hler differentials of the $K$-algebra $L$,

\item $\CD :=\CD_K(L):=\CD (L/K)=L\oplus \CD (L/K)_+$ is the algebra of $K$-linear differential operators on the field extension $L/K$ where 
$$\CD (L/K)_+:=\{\d \in \CD (L/K)\, | \, \d (1)=0\}$$
 is the left ideal of proper  differential operators which is a Lie algebra with commutator of elements as the Lie bracket, 
 % with  zero  constant term,

\item $L^{ \CD (L/K)_+}:=\{ l\in L\, | \, \d(l)=0$ for all $\d \in  \CD (L/K)_+\}$ is the field of $ \CD (L/K)_+$-constants,

\item $\Der_K(L)=\Der (L/K)$ is the Lie algebra of $K$-derivations,

\item  $\D (L):= \D (L/K):=L\langle \Der_K(L)\rangle\subseteq \CD (L/K)$ is the derivation algebra of  $L/K$,

\item The field $L_{dif}:=(L/K)_{dif}:=L_{\CD (L/K)}:=L_\CD:=\{ l\in L\, | \, \d l = l\d$ for all $\d \in \CD (L/K)\}$  is the centralizer in $L$ of the algebra of differential operators on $L/K$,

\item $L^{G(L/K)}_{dif}:=L^{G(L/K)}_{\CD (L/K)}:=L^G_{dif}:= L^G\cap L_{dif}$, 

\item $E:=E(L/K):=\End (L/K):=\End_K (L)\simeq M_n(K)$ is the endomorphism algebra of the $K$-vector space $L$ and  $M_n(K)$ is the algebra of $n\times n$ matrices over $K$ where $n=[L:K]:=\dim_K(L)$ is the degree of the field extension $L/K$,

\item For  subsets $S$ and $T$ of $E(L/K)$, $C_{E(L/K)}(S):=\{ c\in E(L/K)\, | \, cs=sc$ for all $s\in S\}$ is the centralizer of $S$ in $E(L/K)$ and  $C_{T}(S):=\{ t\in T\, | \, ts=st$ for all $s\in S\}$ is the centralizer of $S$ in $T$, 

\item For a field $ M\in \CF (L/K)$, let
\begin{eqnarray*}
\CD (L/K)^M &:=& \{ \d \in \CD (L/K)\, | \, \d m=m\d\;\; {\rm for\; all}\;\; m\in M\}=C_{\CD (L/K)}(M),\\
 G(L/K)^M&:=& \{ g\in G(L/K)\, | \, g(m)=m\;\; {\rm for\; all}\;\; m\in M \}=C_{G(L/K)}(M),
\end{eqnarray*}

\item $L^{pi}$, $L^{sep}$, $L^{nor}$  and  $L^{gal}$ are the largest pure inseparable/separable/normal/ Galois subfields of the field extension $L/K$, respectively, 

\item $L^{pi}_{dif}:=L^{pi}_{\CD (L/K)}:= L^{pi}\cap L_{dif}$,

\item  $(L/K)_{nor}=L_{nor}$ is the least  (w.r.t. $\subseteq$)  co-normal subfield of a finite field extension $L/K$,

\item $\CA (E(L/K),L)$ is the set of all $K$-subalgebras of $E(L/K)$ that contain the field $L$,

\item For an algebra $A\in  \CA (E(L/K),L)$,  let
\begin{eqnarray*}
L^{A\cap \CD (L/K)_+}  &:=& \{ l \in L\, | \, \d (l)=0\;\; {\rm for\; all}\;\; \d\in A\cap \CD (L/K)_+\},\\
&\stackrel{{\rm Cor.}\, \ref{a1Jun25}}{=}&\{ l \in L\, | \, \d l=l\d\;\; {\rm for\; all}\;\; \d\in A\cap \CD (L/K)_+\},\\
 L^{A\cap G(L/K)}&:=& \{  l \in L\, | \, g(l)=l\;\; {\rm for\; all}\;\; g\in A\cap G(L/K) \},
\end{eqnarray*} 

\item $\CA (E,L, G):=\{ A\in  \CA (E,L)\, | \, gAg^{-1}\subseteq A$ for all $g\in G\}$,

%\item $K$ is a field of characteristic $p>0$,  $\bK$ is its algebraic closure and $\gf :\bK\ra \bK$, $\l \mapsto \l^p$ is the Frobenius monomorphism, 

%\item $K_{perf}:=\bigcap_{i\geq 0}\gf^{i}(K)\subseteq K\subseteq K^{perf}:=\bigcup_{i\geq 0}\gf^{-i}(K)\subseteq \bK$ where  $K_{perf}$ is the largest perfect subfield of $K$ and $K^{perf}$ is the perfect closure of $K$ in $\bK$ (the smallest  perfect overfield of $K$ in $\bK$), 

\item $\CL (L/K)$ is the  set of balanced Lie subalgebras $\CG$ of the Lie algebra $\CD (L/K)_+$, i.e.  $C_L(\CG) = L^\CG$,

\item $\mL (L,K)$ is the set of dominant Lie algebras,  

\item $\mL G(L/K):=\{ (\CG, H)\in \mL (L/K)\times \CG (G(L/K))\, | \, h\CG h^{-1}=\CG$    for all elements $h\in H\}$,

\item $\fCK$ is  the class of   central simple finite dimensional  $K$-algebras and  $\fCdK$ is  the class of  central simple finite dimensional  division $K$-algebras,

%\item

\end{itemize}

   The following two results are the main results of the classical Galois Theory: {\em Let a finite field extension $L/K$ be a Galois field extension. Then:}
  \begin{enumerate}

\item {\bf (The Galois Correspondence for subfields)} 
{\em The map
$$
\CF (L/K)\ra \CG (G(L/K)), \;\; M\mapsto  G(L/M)
$$ 
is a bijection with inverse} $
H\mapsto L^H.$

\item {\bf (The  Galois Correspondence for Galois subfields)} 
{\em The map
$$
\CG (L/K)\ra \CN (G(L/K)), \;\; \G\mapsto  G(L/\G)
$$ 
is a bijection with inverse} $
\G\mapsto L^{\G}.$

\end{enumerate}
 
 The relationship between derivations and purely inseparable field extensions of exponent 1 was studied by Baer \cite{Baer-1927}. In \cite{Jacobson-1937, Jacobson-1944}, Jacobson  introduced an analogue  of the Galois theory for {\em purely inseparable} field  extensions  $L/K$ of exponent 1 (meaning that the $p$'th power of every element of $L$ is in $K$), where the Galois groups (which are trivial) are replaced by restricted Lie algebras of derivations.  
 Jacobson's theory establishes a correspondence between intermediate fields of a purely inseparable extension of exponent 1 and restricted Lie subalgebras of derivations.

 A purely inseparable extension is called a {\bf modular extension} if it is a tensor product of simple extensions. In particular, every extension of exponent 1 is modular, but there are non-modular extensions of exponent 2 as shown by Weisfeld \cite{Weisfeld-1965}. In more detail, let us   consider subfields  $K=\Fp (x^p,y^p,z^{p^2})$ and $L=K (z,xz+y)$ of the field $\Fp (x,y,z)$ of rational functions in 3 variables.  Weisfeld \cite{Weisfeld-1965}  proved that the field extension $L/K$ is purely inseparable field extension of exponent 2 which is not modular.
  Sweedler \cite{Sweedler-1968} and Gerstenhaber  and  Zaromp \cite{Gerstenhaber-Zaromp-1970} gave an extension of the Galois correspondence to {\em modular} purely inseparable extensions, where derivations are replaced by higher derivations. A survey of Galois theory for inseparable field extensions is given in the paper of  Deveney  and Mordeson \cite{Deveney-Mordeson-1996}.
\\
  
 Given a $K$-algebra $A$, a group $\CG$ and a group homomorphism $\phi: \CG\ra \Aut_K(A)$, $g\mapsto \phi_g$, then a direct sum of $A$-modules,
$\CA := \bigoplus_{g\in \CG}Ag$,  has a  $K$-algebra structure that is given by the rule: For all elements $a,b\in A$ and $g,h\in \CG$, 
$$ ag\cdot bh=a\phi_g(b)gh.$$
The algebra $\CA$ is called a {\em skew group algebra}. \\
 
{\bf Definition of B-extensions in arbitrary characteristic.} For a finite field extension $L/K$, the algebra $E(L/K)$ contains the group $G(L/K)$ and the algebra $\CD (L/K)$ of differential operators on the $K$-algebra $L$. A $K$-subalgebra of $E(L/K)$ which generated by the group $G(L/K)$ and the algebra $\CD (L/K)$ is the skew group algebra 
$$
\CD (L/K)\rtimes G(L/K)=\bigoplus_{g\in G(L/K)}\CD (L/K)g\subseteq E(L/K),\;\; {\rm (Theorem \; \ref{AB17Apr25}.(1))}.
$$
 \begin{definition}
 A finite field extension $L/K$ is called a {\bf B-extension} if 
 $$E(L/K)=\CD (L/K)\rtimes G(L/K).$$
 \end{definition}
% `B' stands for `Bi'=`2', i.e. both fundamental objects that are attached to the field extension $L/K$, the algebra of differential operators $\CD (L/K)$ and the automorphism group $G(L/K)$, play a key role in the definition. 
 
The notation `B' abbreviates ‘Bi’ = ‘2’: both fundamental structures associated with the field extension $L/K$ - the algebra of differential operators $\CD (L/K)$ and the automorphism group $G(L/K)$ - are essential to the definition.

 By the definition,  a field  extension $L/K$ is a B-extension iff the algebra $\CD (L/K)\rtimes G(L/K)$ is as large as possible. So, B-extensions are the most symmetrical field extensions. 
 \begin{definition}
 A finite field extension $L/K$ is called a {\bf G-extension} if 
 $$E(L/K)=L\rtimes G(L/K).$$
 A finite field extension $L/K$ is called a {\bf D-extension} if 
 $$E(L/K)=\CD(L/K).$$
 \end{definition}
Above, `G' and `D' stand for `the Galois group' and `the algebra of differential operators', respectively. By the definition, G- and D-extensions are B-extensions. There are many B-extensions that are neither G- nor D-extensions. They are precisely normal field extensions that are neither Galois nor purely inseparable field extensions (Corollary \ref{VVB-d4May25}).

\begin{itemize}

\item  {\bf Proposition \ref{A8May25}.} {\em A finite field extension  is   a G-extension iff it  is a Galois field extension.}

\item {\bf Corollary \ref{VVB-a4May25}.}
 {\em A finite field extension is a D-extension iff it is a purely inseparable field extension.}

\item {\bf Proposition \ref{VB-aC24Mar25}.}  {\em A finite field extension $L/K$  is separable iff} $\CD (L/K)=L$.

\item {\em So, for finite separable field extensions, the Galois extensions are the only B-extensions (and this case was treated separately in \cite{GaloisTh-RingThAp} from the point of view of G-extensions). Therefore, in this paper we consider the (remaining) case of finite inseparable field extensions (necessarily of characteristic $p>0$).  }
\end{itemize}

 So, intuitively, purely inseparable field extensions in prime characteristic are direct analogue of Galois field extensions  in the separable situation where the algebras  $E(L/K)=\CD(L/K)$ and $E(L/K)=L\rtimes G(L/K)$ are `swapped', i.e. the skew group algebra $L\rtimes G(L/K)$ is replaced by the algebras of differential operators $\CD(L/K)$ (automorphisms are replaced by differential operators). The general situation for B-extensions (in prime characteristic), is a mixture of two cases: the Galois and the purely inseparable ones (Theorem \ref{VVB-30Apr25}).\\
 
{\bf Analogue of the Galois Theory for purely inseparable finite  field extensions.} Let $\CA ( \CD (L/K), L, sim)$ be the set of simple $K$-subalgebras of $\CD (L/K)$ that contains the field $L$. Notice that $\CD (L/K)=L\oplus \CD (L/K)_+$ where 
$$\CD (L/K)_+:=\{\d \in \CD (L/K)\, | \, \d (1)=0\}$$
 is the {\bf left ideal of proper  differential operators} on $L/K$  which is a Lie algebra with commutator of elements as the Lie bracket. The left $\CD (L/K)$-module $L$ is isomorphic to the 
 factor module $\CD (L/K)/\CD (L/K)_+$.
  For a subset $S\subseteq \CD (L/K)$, let 
  $$
  C_L(S) := \{l\in L\, |\, ls=sl\; {\rm  for\; all}\;s\in S \}\;\; {\rm and}\;\; 
  L^S := \{l\in L\, |\, s(l)=0\; {\rm  for\; all }\; s\in S\}.
  $$
   Clearly, $C_L(S) \in \CF (L/K)$. Lemma \ref{bD24Mar25}.(2) shows that $ C_L(S)\subseteq L^S$.

\begin{definition} 

A Lie subalgebra $\CG$ of the Lie algebra $\CD (L/K)_+$ is called {\bf balanced} if  
$C_L(\CG) = L^\CG$. The set of balanced Lie subalgebras  of $\CD (L/K)_+$ is denoted by $\CL (L/K)$.  
A Lie algebra $\CG \in \CL (L/K)$ is called a {\bf dominant Lie algebra} if $\CG'\subseteq \CG$ 
for all Lie algebras $\CG' \in \CL (L/K)$  such that $L^{\CG'}=L^\CG$. Let $\mL (L/K)$ be the set  of dominant Lie algebras and 
$$
\mL G(L/K):=\{ (\CG, H)\in \mL (L/K)\times \CG (G(L/K))\, | \, h\CG h^{-1}=\CG\;\; {\rm  for\; all\; elements}\;\; h\in H\}.
$$  An element $(\CG, H)\in \mL G(L/K)$ is called a {\bf dom-group} or a {\bf dominant pair} and the set $\mL G(L/K)$ is called the {\bf dom-group set} or the 
{\bf set of dominant pairs} of the field extension $L/K$.
\end{definition}
 A dominant Lie algebra and a dom-group are two of the key concepts  in the analogue of the Galois correspondence for {\em normal} finite  field extensions (Theorem \ref{1Jun25}). The concept of  a dominant Lie algebra is one  of the key concepts  in the analogue of the Galois correspondence for {\em purely inseparable}   finite  field extensions (Theorem \ref{D24Mar25}).

Theorem \ref{D24Mar25} is an analogue of the Galois Theory for purely inseparable finite field extensions that gives a bijection between subfields and the algebras  of   differential operators (Theorem \ref{D24Mar25}.(1,2)) and between subfields and the dominant Lie algebras  (Theorem \ref{D24Mar25}.(3,4)).\\

$\bullet$ {\bf Theorem \ref{D24Mar25}.} 
{\em  Let $K$ be a field of  characteristic $p$, 
 $L/K$ be a  purely inseparable finite  field extension. Then:}
 \begin{enumerate}

\item $\CA ( \CD (L/K), L)\stackrel{{\rm Pr.}\, \ref{aD24Mar25}}{=}\CA ( \CD (L/K), L, sim)=\{ \CD(L/M)\, | \, M\in \CF (L/K) \}$.

\item {\sc (Analogue of the Galois correspondence  for subfields of a purely inseparable finite  field extension)}

 {\em The map 
$$\CF (L/K)\ra \CA ( \CD (L/K), L)\stackrel{{\rm Pr.}\, \ref{aD24Mar25}}{=}\CA (\CD (L/K), L, sim), \;\; M\mapsto  C_{\CD (L/K)}(M)=\CD (L/M)$$  is a bijection with inverse} $A\mapsto C_{\CD (L/K)}(A)$.  %$\Big( \CD'=\CD (L/M)\mapsto L^{\CD (L/M)_+}.\Big)$ 

\item $\mL (L/K)=\{ \CG (M)=\CD(L/M)_+\, | \, M\in \CF (L/K)\}$.

\item  {\sc (Analogue of the Galois correspondence  for subfields of a purely inseparable finite  field extension)}

{\em The map}
$$
\CF (L/K)\ra  \mL (L/K), \;\; M\mapsto \CG (M)=\CD(L/M)_+
$$
{\em is a bijection with inverse} $\CG\mapsto L^\CG=C_L(\CG)$.

\end{enumerate}

{\bf Analogue of the Galois Theory for normal finite field extensions.} Notice that for a  normal finite field extension $L/K=L^{pi}\t L^{gal}$  and a field $M\in \CF (L/K)$, the following  maps are  isomorphisms of Galois groups:
\begin{eqnarray*}
 G(L/K)&\ra & G(L/L^{pi}), \;\;\;\; g\mapsto g,\\
 G(L/L^{pi})&\ra & G(L^{gal}/K), \;\;g\mapsto g|_{L^{gal}}: L^{gal}\ra L^{gal},\\
  G(L/ML^{pi})&\ra & G(L/M), \;\;\;\;\;\; g\mapsto g.
\end{eqnarray*}
Therefore, we may identify the Galois groups above: $ G(L/K)=G(L/L^{pi})=G(L^{gal}/K)$ and  $G(L/M)=G(L/ML^{pi})$. For normal field extensions, Theorem \ref{1Jun25} gives an order reversing  Galois-type correspondence for their subfields. \\

$\bullet$ {\bf Theorem \ref{1Jun25}.} 
{\em  Let $L/K$ be a normal finite field extension of characteristic $p$. Then:}

\begin{enumerate}

\item $ \CA (E(L/K),L)=\{ C_{E(L/K)}(M)=\CD (L/M)\rtimes G(L/M)=\CD (L/K)^M\rtimes G(L/K)^M\, | \, M\in \CF (L/K)\}$. 

\item  {\sc (Analogue of the  Galois correspondence for subfields of a normal  field extension)}

 {\em The map}
$$
\CF (L/K)\ra \CA (E(L/K),L), \;\; M\mapsto  C_{E(L/K)}(M)=\CD (L/M)\rtimes G(L/M)=\CD (L/K)^M\rtimes G(L/K)^M
$$ 
{\em is a bijection with inverse}
\begin{eqnarray*}
A&\mapsto& C_{E(L/K)}(A)= L^{A\cap \CD (L/K)_+}\cap L^{A\cap G(L/K)}=\Big(L^{A\cap \CD (L/K)_+} \Big)^{A\cap G(L/K)}.\\
 \bigg( A&=&\CD (L/M)\rtimes G(L/M)\mapsto 
 L^{\CD (L/M)_+}\cap L^{G(L/M)}=\Big(L^{\CD (L/M)_+} \Big)^{G(L/M}.\bigg)
\end{eqnarray*}

\item $\mL G(L/K)=\Big\{ \Big(\CD (L/M)_+, G(L/M)\Big) \, | \, M\in \CF (L/K)\Big\}$ {\em and for all fields $M\in \CF (L/K)$, 
$$
\Big(\CD (L/M)_+, G(L/M) \Big)=\Big(\CD (L/ML^{gal})_+, G(L/ML^{pi}) \Big)\in \mL (L/L^{gal})\times G(L/L^{pi}).
$$
 In particular,} 
$$
\mL G(L/K)=\Big\{ (\CG , H)\in \mL (L/L^{gal})\times \CG (G(L/L^{pi}))\, | \, h\CG h^{-1}=\CG\;\; {\rm  for\; all\; elements}\;\; h\in H\}.
$$
%The restriction map $G(L/L^{pi})\ra G(L^{gal}/K)$, $g\mapsto g|_{L^{gal}}: L^{gal}\ra L^{gal}$ is a group  isomorphism.

\item  {\sc (Analogue of the  Galois correspondence for subfields of a normal  field extension)}

{\em The map}
$$
\CF (L/K)\ra \mL G(L/K),  \;\; M\mapsto  \Big(\CD (L/M)_+, G(L/M) \Big)=\Big(\CD (L/ML^{gal})_+, G(L/ML^{pi}) \Big)
$$ 
{\em is a bijection with inverse}
$ (\CG , H)\mapsto L^\CG\cap L^H=\Big( L^\CG\Big)^H$ and $h(L^\CG)=L^\CG$ for all $h\in H$.
%The restriction map $G(L/L^{pi})\ra G(L^{gal}/K)$, $g\mapsto g|_{L^{gal}}: L^{gal}\ra L^{gal}$ is a group  isomorphism.

\item {\em For all fields $M\in \CF (L/K)$, $[M:K]|G(L/M)|\dim_K(\CD (L/M))=[L:K]^2$  or, equivalently,} $|G(L/M)|\dim_M(\CD (L/M))=[L:M]^2$.

\item {\em The field extension $L/L^{gal}$ is a finite purely inseparable field extension and} 
$\mL (L/L^{gal})=\{ \CD (L/N)_+\, | \, N\in \CF (L/L^{gal})\}$.

\end{enumerate}

For a finite field extension $L/K$, let $\CN (L/K)$ be the set of normal subfields $N/K$ of $L/K$ and 
 $\CN(G(L/K))$ be the set of normal subgroups of the group $ G(L/K)$. For normal field extensions, Theorem \ref{11Jun25} describes their  normal subfields  and  establishes an analogue of the Galois correspondence  for normal  subfields.\\

$\bullet$ {\bf Theorem \ref{11Jun25}.} 
{\em 
Let $L/K$ be a normal finite field extension of characteristic $p$. Then:}

\begin{enumerate}

\item {\em For each field $M\in  \CN (L/K)$, $M=M^{pi}\t M^{gal}$, where $M^{pi}=M\cap L^{pi}
$ and $M^{gal}:=M\cap L^{gal}$, and} 
$$\CN (L/K)=\CF (L^{pi}/K)\t \CN (L^{gal}/K)
:=\{ N\t \G\, | \, N\in \CF (L^{pi}/K),\,  \G\in \CN (L^{gal}/K)\}.
$$
\item $\CA (E(L/K),L, G(L/K))=\Big\{\CD (L^{pi}/N)\t \Big(L^{gal}\rtimes G(L^{gal}/\G) \Big)= E(L^{pi}/N)\t  E(L^{gal}/M^{gal})\, | \,$ $ N\in \CF (L^{pi}/K),\,  \G\in \CN  (L^{gal}/K)  \Big\}$.

\item {\sc (Analogue of the  Galois correspondence for normal  subfields of a normal  field extension)}

{\em The map}
\begin{eqnarray*}
\CN (L/K)&\ra & \CA (E(L/K),L, G(L/K)), \;\; M\mapsto  C_{E(L/K)}(M)=C_{E(L^{pi}/K)}(M^{pi})\t C_{E(L^{gal}/K)}(M^{gal})\\
&=&E(L^{pi}/M^{pi})\t E(L^{gal}/M^{gal})=\CD (L^{pi}/M^{pi})\t \Big(L^{gal}\rtimes G(L^{gal}/M^{gal}) \Big)
\end{eqnarray*}
{\em  is a bijection with inverse}
\begin{eqnarray*} A&\mapsto& C_{E(L/K)}(A)=\Big( L^{pi}\Big)^{A\cap \CD (L^{pi}/K)_+}\t \Big( L^{gal}\Big)^{A\cap G(L^{gal}/K)}.\\
 \bigg(A&=&\CD (L^{pi}/N)\t \Big(L^{gal}\rtimes G(L^{gal}/\G) \Big) \mapsto   \Big( L^{pi}\Big)^{ \CD (L^{pi}/N)_+}\t \Big( L^{gal}\Big)^{ G(L^{gal}/\G)}.\bigg)
\end{eqnarray*}

\item {\sc (Analogue of the  Galois correspondence for normal  subfields of a normal  field extension)} 

{\em The map}
$$
\CN (L/K)\ra \mL (L^{pi}/K)\times \CN (G(L^{gal}/K)), \;\; M=M^{pi}\t M^{gal}\mapsto \Big(\CD (L^{pi}/M^{pi})_+, G(L^{gal}/M^{gal})\Big)
$$
{\em  is a bijection with inverse} $(\CG, H)\mapsto \Big(L^{pi} \Big)^\CG\t \Big( L^{gal}\Big)^H$.

\end{enumerate}

{\bf Characterizations of the field $L_{dif}$ and the vector space $L^{\CD (L/K)_+}$.}  Theorem \ref{23May25} is one of the key facts in the analogue of the Galois Theory for normal finite field extensions: the equality $L_{dif}=L^{\CD (L/K)_+}$ is the guiding principle behind the definition of a balanced Lie algebra.
 The equalities   $L^{sep}=L_{dif}$ and $L^{sep}=L^{\CD (L/K)_+}$ provide two complementary characterizations of the field $L^{sep}$: one via commutation, the other via the solution set of the differential-operator system.

\begin{itemize}
\item  {\bf Theorem \ref{23May25}.} 
{\em Let $L/K$ be a finite field extension. Then $L_{dif}=L^{sep}=L^{\CD (L/K)_+}$.}
\end{itemize}

{\bf Characterization of B-extensions.}
  Theorem \ref{VVB-30Apr25} is a characterization of  B-extensions. In particular, it shows that the class of B-extensions coincide with the class of normal finite field extensions. \\

 $\bullet$ {\bf Theorem \ref{VVB-30Apr25}. }
{\em Let $L/K$ be a finite field extension  of prime characteristic. Then the  following statements are equivalent:}
\begin{enumerate}

\item $L/K$ {\em is a B-extension.}

\item $L=L^{G(L/K)}\t L_{dif}$.

%\item $L=L^{pi}\t L_{dif}$.

\item $L^{G(L/K)}_{dif}=K$.

%\item $L^{pi}\cap  L_{dif}=\{ e\}$.

\item $L=L^{G(L/K)}\t L^{gal}$ and  $L^{G(L/K)}/K$ {\em is a normal the field  extension.}

\item $L^{G(L/K)}=L^{pi}$. 

%\item ***neverno $L=L^{pi}\t L_{dif}$.***

%\item ***neverno   $L^{pi}_{dif}=K$.  ***

\item $L=L^{pi}\t L^{gal}$.

\item  $L/K$ {\em is normal.}

\end{enumerate} 
{\em If the equivalent conditions hold then:}

\begin{enumerate}

%\item[(i)] $L^{G(L/K)}=L^{pi}$ and $L_{dif}= L^{gal}$.

\item[(a)] $L^{G(L/K)}=L^{pi}$,  $\CD (L^{G(L/K)}/K)=E(L^{G(L/K)}/K)$ {\em and } $\Big(L^{G(L/K)}/K\Big)_{dif}=K$. 

\item[(b)]    $L_{dif}=L^{sep}= L^{gal}$,  $L_{dif}\rtimes G(L_{dif}/K)=E(L_{dif}/K)$, $(L_{dif})^{G(L_{dif}/K)}=K$ {\em and}  
\begin{eqnarray*}
E(L/K)&=&\CD (L/K)\rtimes G(L/K)=\CD (L^{G(L/K)}/K) \t \Big(L_{dif}\rtimes G(L_{dif}/K) \Big)\\
&=&\CD (L^{G(L/K)}/K) \t \Big(L^{gal}\rtimes G(L^{gal}/K) \Big)=\CD (L^{pi}/K) \t \Big(L_{dif}\rtimes G(L_{dif}/K) \Big)\\
&=&\CD (L^{pi}/K) \t \Big(L^{gal}\rtimes G(L^{gal}/K) \Big).
\end{eqnarray*}

\end{enumerate}

{\bf The subalgebra $\CD (L/K)\rtimes G(L/K)$ of the endomorphism algebra $E(L/K)$.}  For a finite field extension $L/K$ of characteristic $p$, the map
$$
G(L/K)\ra \Aut_K(\CD (L/K)), \;\; g\mapsto \o_g: \d \mapsto g(\d):=g\d g^{-1}
$$

is a group {\em monomorphism} (since $L\subseteq \CD (L/K)$). Theorem \ref{AB17Apr25} shows that a subalgebra of $E(L/K)$ which is generated by  the automorphism group $G(L/K)$ and the algebra of differential operators $\CD (L/K)$ is the skew group algebra $\CD (L/K)\rtimes G(L/K)$. \\

$\bullet$  {\bf Theorem \ref{AB17Apr25}.} 
{\em Let $L/K$ be a finite field extension of   characteristic $p$. Then:}

\begin{enumerate}

\item {\em The subalgebra of $E(L/K)$ which is generated by the automorphism group $G(L/K)$ and the algebra of differential operators $\CD (L/K)$ is the skew group algebra 
$$\CD (L/K)\rtimes G(L/K)=\bigoplus_{g\in G(L/K)}\CD (L/K)g$$
where the multiplication is given by the rule: For all elements $\d, \d'\in \CD  (L/K)$ and $g,g'\in G(L/K)$,
$$ \d g\cdot \d'g'=\d g\d'g^{-1} \cdot gg'$$ 
 where $g\d'g^{-1}\in \CD (L/K)$, by (\ref{GCDmon}).

\item $\CD (L/K)\rtimes G(L/K)=E\Big(L/L^{G(L/K)}_{dif}\Big)\in \fC \Big(L^{G(L/K)}_{dif}\Big)$ where} $L^{G(L/K)}_{dif}:= L^{G(L/K)}\cap L_{dif}$.

\item $C_{E(L/K)}(\CD (L/K)\rtimes G(L/K))=L^{G(L/K)}_{dif}$ and $C_{E(L/K)}\Big(L^{G(L/K)}_{dif}\Big)=\CD (L/K)\rtimes G(L/K)$.

\end{enumerate}

 {\bf The algebra of    differential operators  $\CD (L/K)$ on a finite field extension $L/K$ in prime characteristic.}  Theorem \ref{20Apr25} is an explicit  description of the algebra $\CD (L/K)$ of    differential operators on a finite field extension $L/K$ in prime characteristic.\\

 $\bullet$ {\bf Theorem \ref{20Apr25}.}
 {\em 
Let $L/K$ be a finite field extension of characteristic $p$. Then:}
\begin{enumerate}

\item $\CD (L/K)=\CD (L/L_{dif})=\CD\Big(L^{G(L/K)}\t_{L^{G(L/K)}_{dif}} L_{dif}/ L^{G(L/K)}_{dif}\Big)=
L_{dif}\t_{L^{G(L/K)}_{dif}}\CD\Big(L^{G(L/K)}/ L^{G(L/K)}_{dif}\Big)$.
 {\em In particular, each differential operator in $\CD\Big(L^{G(L/K)}/ L^{G(L/K)}_{dif}\Big)$ can be uniquely lifted to a differential operator in $\CD (L/K)$.

\item If, in addition, $ L^{G(L/K)}_{dif}=K$ $(\Leftrightarrow L=L^{G(L/K)}\t L_{dif}$, by Proposition \ref{AC17Apr25}.(6)),  i.e. $L/K$ is a B-extension  then} $\CD (L/K)=\CD (L^{G(L/K)}\t L_{dif}/K)=L_{dif}\t \CD (L^{G(L/K)}/K)$. 

\end{enumerate}

 {\bf The algebra $L\rtimes G(L/K)$.}  Proposition \ref{A20Apr25} is an explicit  description of the subalgebra  of $E(L/K)$ which is generated by the field $L$ and the automorphism group $G(L/K)$. This subalgebra is a skew group algebra $L\rtimes G(L/K)$.\\

 $\bullet$ {\bf Proposition  \ref{A20Apr25}.}
{\em Let $L/K$ be a  finite field extension of characteristic $p$. Then }
$$L\rtimes G(L/K)=L^{G(L/K)}\t_{L^{G(L/K)}_{dif}}\Big( L_{dif}\rtimes G \Big(L_{dif}/L^{G(L/K)}_{dif}\Big)\Big).$$

{\bf The field extension $L/L^{G(L/K)}_{dif}$ and its  properties.} Theorem  \ref{9Jun25} shows that the  field extension $L/L^{G(L/K)}_{dif}$  is a B-extension.\\

$\bullet$ {\bf Theorem \ref{9Jun25}.}  {\em Let $L/K$ be a  finite field extension of characteristic $p$. Then:}

\begin{enumerate}

\item {\em The field extension $L/L^{G(L/K)}_{dif}$   is a B-extension, hence normal (by Theorem \ref{VVB-30Apr25}). }

\item $\CD (L/K)\rtimes G(L/K)=\CD \Big(L/L^{G(L/K)}_{dif}\Big)\rtimes G\Big(L/L^{G(L/K)}_{dif}\Big)=E\Big(L/L^{G(L/K)}_{dif}\Big)$. 

\item $\CD (L/K)=\CD \Big(L/L^{G(L/K)}_{dif}\Big)$.

\item $ G(L/K) =G\Big(L/L^{G(L/K)}_{dif}\Big)$.

%\item $\Big( L/L^{G(L/K)}_{dif}\Big)_{dif}=L^{G(L/K)}_{dif}$.
\end{enumerate}

Corollary \ref{a9Jun25} clarifies the structure of the field extension $L/L^{G(L/K)}_{dif}$. \\

 $\bullet$ {\bf Corollary \ref{a9Jun25}.} 
{\em Let $L/K$ be a  finite field extension of characteristic $p$. Then:}

\begin{enumerate}

\item $L=\Big( L/L^{G(L/K)}_{dif}\Big)^{pi}\t_{L^{G(L/K)}_{dif}}\Big( L/L^{G(L/K)}_{dif}\Big)^{gal}=L^{G(L/K)}\t_{L^{G(L/K)}_{dif}}   L_{dif}$.

\item $\Big( L/L^{G(L/K)}_{dif}\Big)^{pi}=L^{G\Big(L/L^{G(L/K)}_{dif}\Big)}= L^{G(L/K)}$.

\item  $\Big( L/L^{G(L/K)}_{dif}\Big)^{gal}=\Big( L/L^{G(L/K)}_{dif}\Big)^{sep}=\Big( L/L^{G(L/K)}_{dif}\Big)_{dif}=L_{dif}$.

\end{enumerate}

{\bf The least co-normal field extension $(L/K)_{nor}$.}
\begin{definition}
For a finite field extension $L/K$ of arbitrary characteristic, a field extension $M/K$ in $L/K$ is called a {\bf co-normal subfield} if the field extension $L/M$ is normal. The {\bf least co-normal subfield} of $L/K$ (w.r.t. $\subseteq$) is denoted by $(L/K)_{nor}=L_{nor}$. 
\end{definition}

 Theorem  \ref{A9Jun25} shows that the  field extension $L^{G(L/K)}_{dif}/K$  is the least co-normal field extension of $L/K$ when $\char (K)=p$. \\

  $\bullet$ {\bf Theorem  \ref{A9Jun25}.}
{\em Let $L/K$ be a  finite field extension of characteristic $p$. Then}  
$$(L/K)_{nor}=L^{G(L/K)}_{dif}/K.$$

In characteristic zero, Proposition  \ref{B9Jun25} shows that the  field extension $L^{G(L/K)}/K$  is the least co-normal  field extension of $L/K$.\\

 $\bullet$  {\bf Proposition \ref{B9Jun25}.}{\em 
Let $L/K$ be a  finite field extension of characteristic zero. Then} 
$$(L/K)_{nor}=L^{G(L/K)}/K.$$

In \cite{AnGaloisTh-ArbFields}, 
 an analogue of the Galois Theory for arbitrary finite field extensions is presented. The approach is based on this paper.

The paper is organized as follows:

 In Section \ref{AELKL}, we collect results on the central simple algebras that are used in the paper. The key results of Section \ref{AELKL} are  Theorem 
 \ref{A24Mar25} and  Theorem  \ref{B24Mar25}.
 For a central simple algebra $A\in \fCK$ and its strongly maximal subfield $L\in \Max_s(A)$, Theorem 
 \ref{A24Mar25}  describes subalgebras of $A$  that contain the field $L$. For a finite field extension $L/K$,  Theorem \ref{B24Mar25} describes subalgebras of the endomorphism algebra $E(L/K)$ that contain the field $L$.

In Section \ref{ALG-DIF-OPS}, Lemma \ref{aC24Mar25} and Theorem \ref{AC24Mar25} give  explicit descriptions of the algebras of differential operators $\CD (L/K)$ for an arbitrary finite field extension $L/K$ in characteristic zero and  in prime characteristic, respectively. Proposition \ref{VB-aC24Mar25} is a criterion for the algebra of differential operators $\CD (L/K)$ being a commutative algebra or equivalently $\CD (L/K)=L$. 
Theorem \ref{23May25} shows that  $L_{dif}=L^{sep}=L^{\CD (L/K)_+}$. Corollary \ref{b16Apr25} is  a criterion for $\CD (L/K)=\D (L/K)$.

  In Section \ref{GAL-PURE-INS}, for a purely inseparable   finite field extension $L/K$, Theorem \ref{C24Mar25} provides an explicit description of the algebra of differential operators  $\CD (L/K)$.   Theorem \ref{D24Mar25} is an analogue of the Galois Theory for purely inseparable field extensions that establishes a bijection between subfields and  algebras  of   differential operators and  between subfields and the dominant Lie algebras.
  
 In Section \ref{3-ALGS},   Proposition  \ref{A14May25} provides sufficient conditions for two subfields  being linearly disjoint (i.e. their compositum  is isomorphic to their tensor product). Proposition  \ref{A14May25} is a powerful result since it implies the known fact that {\em a finite field extension $L/K$ is normal iff}
  $$
 L=L^{pi}\t L^{gal}\;\; {\rm (Corollary\;  \ref{a25May25})}.
 $$
  Theorem \ref{B11May25} is a criterion for $\CD (L/K)=E(L/K)$. Theorem \ref{BC24Mar25} describes properties of the subalgebra $L\rtimes G(L/K)$ of $E(L/K)$.  Proposition \ref{A8May25} is a criterion for a finite field extension being a Galois field extensions which is given via the algebra $L\rtimes G(L/K)$. Corollary \ref{a24Mar25} provides explicit descriptions of the sets 
  $$\CA (L\rtimes G(L/K), L)\;\; {\rm  and }\;\;\CA (L\rtimes G(L/K), L, G(L/K)).
  $$
   Theorem \ref{AB17Apr25} shows that the subalgebra of $E(L/K)$ which is generated by  the automorphism group $G(L/K)$ and the algebra of differential operators $\CD (L/K)$ is the skew group algebra $\CD (L/K)\rtimes G(L/K)$. Every finite field extension $L/K$ is a B-extension over the field $L^{G(L/K)}_{dif}$, i.e. $L/L^{G(L/K)}_{dif}$ is a B-extension  (Theorem \ref{AB17Apr25}.(2)). Corollary \ref{c17Apr25} is a criterion for the finite field extension $L/K$ being a B-extension.

 Proposition \ref{AC17Apr25} shows that,
 the field extension $L_{dif}/L^{G(L/K)}_{dif}$ is a Galois field extension with Galois group isomorphic to $G(L/K)=G(L/L^{G(L/K)})$ and 
 $$
 L=L^{G(L/K)}\t_{L^{G(L/K)}_{dif}}L_{dif}.
 $$
  Proposition \ref{AC17Apr25}.(5) presents natural classes of field extensions with identity automorphism groups. Namely, $G(L/L_{dif})=\{ e\}$ and $G\Big(L^{G(L/K)}/L^{G(L/K)}_{dif} \Big)=\{ e\}$.
  Proposition \ref{VB-C17Apr25} is a strengthening of  Proposition \ref{AC17Apr25} where the finite field extension $L/K$ is normal. 
 
 Theorem \ref{20Apr25} is an explicit  description of the algebra $\CD (L/K)$ of    differential operators on a finite field extension $L/K$ in prime characteristic. Corollary \ref{a20Apr25} is an explicit  description of the algebra $\CD (L/K)$ of    differential operators on a {\em normal} finite field extension $L/K$ in prime characteristic.

  Theorem  \ref{bA20Apr25} shows that the algebra $ \CD (L/K)\rtimes G(L/K)$ is a tensor product of  the subalgebras 
  $$\CD \Big(L^{G(L/K)}/L^{G(L/K)}_{dif}\Big)\;\; {\rm  and} \;\; L_{dif}\rtimes  G\Big(L_{dif}/L^{G(L/K)}_{dif}\Big)
  $$
   over the field $L^{G(L/K)}_{dif}$.
   If, in addition, the field extension $L/K$ is normal, Corollary \ref{cA20Apr25} shows that the algebra $ \CD (L/K)\rtimes G(L/K)$ is a tensor product of  the subalgebras $\CD (L^{pi}/K)$ and $ L^{sep}\rtimes  G (L^{sep}/K)$.
 
Suppose that $L_i/K$ $(i=1, \ldots , n)$ are normal finite field extensions such that their tensor product $L=\bigotimes_{i=1}^nL_i$ is a field.  Proposition \ref{A7Jun25} shows that 
$$\CD (L/K)\rtimes G(L/K)\simeq \bigotimes_{i=1}^n\CD (L_i/K)\rtimes G(L_i/K)
,\;\;  \CD(L)=\bigotimes_{i=1}^n\CD (L_i/K)$$  
  and $G(L/K)\simeq \prod_{i=1}^n G(L_i/K)$.

In Section \ref{B-EXT=NORMAL}, the aim is to give a proof of Theorem \ref{VVB-30Apr25} which is a characterization of B-extensions. In particular, it shows that the classes of B-extensions and  normal finite field extensions coincide. 
Corollary \ref{VVB-a4May25} is a characterization of B-extensions $L/K$ with $\CD (L/K)=E(L/K)$.
  Corollary \ref{VVB-c4May25} is a characterization of B-extensions $L/K$ with $L\rtimes G(L/K)=E(L/K)$.  
  
In Section \ref{AN-GT-NORMAL}, the aim is to prove  Theorem \ref{1Jun25} and  Theorem \ref{11Jun25}.  
 For normal field extensions, Theorem \ref{1Jun25} gives an order reversing  Galois-type correspondence for their subfields. Theorem \ref{11Jun25} describes normal subfields of normal field extensions and  establishes an analogue of the Galois correspondence  for normal subfields.

In Section \ref{LLGKL}, we apply the results of the previous sections to  field extensions  $L/L^{G(L/K)}_{dif}$ where $L/K$ are finite field extension of characteristic $p$. We show that the  field extensions  $L/L^{G(L/K)}_{dif}$ are B-extensions, hence normal (Theorem \ref{9Jun25}). Corollary \ref{a9Jun25} is about  the structure of the field extension $L/L^{G(L/K)}_{dif}$. In particular, it shows that   
$$\Big( L/L^{G(L/K)}_{dif}\Big)^{pi}= L^{G(L/K)}
\;\; {\rm  and }\;\;\Big( L/L^{G(L/K)}_{dif}\Big)^{gal}=\Big( L/L^{G(L/K)}_{dif}\Big)^{sep}=\Big( L/L^{G(L/K)}_{dif}\Big)_{dif}=L_{dif}.
$$
In prime characteristic, we prove that  $(L/K)_{nor}=L^{G(L/K)}_{dif}$ (Theorem \ref{A9Jun25}) and in characteristic zero, $(L/K)_{nor}=L^{G(L/K)}$  (Proposition \ref{B9Jun25}). 
%Theorem \ref{A9Jun25} and Proposition \ref{B9Jun25} explicitly describe the least co-normal subfields are  in prime and  zero characteristics, respectively.

 %%%%%%%%%%%%%%%%%% SECTION 2 %%%%%%%%%%%%%%%%%%%%%

\section{Description of subalgebras of the endomorphism algebra $E(L/K)$ that contain the field $L$} \label{AELKL} %\marginpar{AELKL}
 
 At the beginning of the section, we collect results on the central simple algebras that are used in the paper. The key results of the section are  Theorem 
 \ref{A24Mar25} and  Theorem  \ref{B24Mar25}.
 For a central simple algebra $A\in \fCK$ and its strongly maximal subfield $L\in \Max_s(A)$, Theorem 
 \ref{A24Mar25}  describes subalgebras of $A$  that contain the field $L$. For a finite field extension $L/K$,  Theorem \ref{B24Mar25} describes subalgebras of the endomorphism algebra $E(L/K)$ that contain the field $L$.\\ 
 
 {\bf Central simple algebras.}    Let us recall some basic facts and definitions on central simple algebras, see the book \cite{Pierce-AssAlg} for details. For a field $K$, we denote by $\fCK$  the class of   central simple finite dimensional  $K$-algebras. Let  $\fCdK$ the class of   central simple finite dimensional  division $K$-algebras. Clearly,  $\fCdK\subseteq \fCK $. A 
$K$-algebra belongs to the class $\fCK$ iff $A\simeq M_n(D)$ for some $D\in \fCdK$.

 The dimension of a division $K$-algebra $D$ over a field $K$ is denoted by $\dim_K(D)=[D:K]$.  For each algebra $\CA\in \fCK$, 
$$
\dim_K(\CA )=n^2\;\; {\rm for\; some\; natural\; number}\;\; n\geq 1
$$
 which is called the {\bf degree} of $\CA $ and is denoted by $\Deg (\CA )$. The algebra $\CA$ is isomorphic to the matrix algebra $M_l(D)$ over a division algebra $D\in \fCK$ which is unique (up to a $K$-isomorphism). Clearly, $\Deg (\CA ) = l\Deg (D)$. The degree $\Deg (D)$ is called the {\bf index} of $\CA$ and is denoted by $\Ind (\CA )$. If $L$ is a  subfield of $\CA$ that contains the field $K$ then 
 $$
 [L:K]\leq \Deg (\CA ).
 $$
  The field $L$ is called a {\bf strictly maximal} subfield if $[L:K]= \Deg (\CA )$. The sets of maximal and strictly maximal subfields of $A\in \fCK$ are denoted by $\Max (A)$ and $\Max_s(A)$, respectively. In general, the containment $\Max_s(A)\subseteq \Max (A)$ is strict but for all central division algebras $D\in \fCdK$,  the equality 
  $$
  \Max_s(D)= \Max (D)
  $$ 
  holds.

\begin{theorem}\label{Noether-SkolemThm}%\marginpar{Noether-SkolemThm}
{\bf (The Noether-Skolem Theorem)} Let $\CA\in \fCK$,  $A$ and $B$ be isomorphic,   simple subalgebras of $\CA$. Then $B=uAu^{-1}$ for some unit $u\in \CA^\times$.
\end{theorem} 
 
 For a subset $S$ of $A$, the set  
$C_A(S)=\{ a\in A\, | \, as=sa\}$ 
 is called the {\bf centralizer} of $S$ in $A$. The centralizer $C_A(S)$ is a subalgebra of $A$. 
 The algebra $C_AC_A(S)=C_A(C_A(S))$ is called the {\em double centralizer} of $S$.
 Notice that $S\subseteq C_AC_A(S)$. Clearly, 
$C_AC_AC_A(S)=C_A(S)$.
 
  Lemma \ref{DCT-1} shows that the centralizer is a natural endomorphism algebra.

\begin{lemma}\label{DCT-1}%\marginpar{DCT-1}
Let $B$ be a subalgebra of an algebra $A$. Then the maps 
\begin{eqnarray*}
C_A(B)& \ra & E({}_BA_A), \;\; c\mapsto (l_c=c\cdot : A\ra A, \;\; a\mapsto ca),\\
C_A(B)&\ra & E({}_AA_B), \;\; c\mapsto (r_c=\cdot c: A\ra A, \;\; a\mapsto ac),
\end{eqnarray*}
are  algebra isomorphisms.

\end{lemma}

\begin{proof} The map $A\ra E(A_A)$, $c\mapsto l_c$ is an algebra isomorphism. Since $ E({}_BA_A)\subseteq  E(A_A)$, the map $l_c$ belong to  $ E({}_BA_A)$ iff $c\in C_A(B)$.

 Similarly, the map $A\ra E(_AA)$, $c\mapsto r_c$ is an algebra isomorphism. Since $ E({}_AA_B)\subseteq  E({}_AA)$, the map $r_c$ belong to  $ E({}_AA_B)$ iff $c\in C_A(B)$.
\end{proof}

The dimension of a division $K$-algebra $D$ over a field $K$ is denoted by $\dim_K(D)=[D:K]$.

\begin{theorem}\label{DCThm}%\marginpar{DCThm}
Let $A\in \fCK$ and $B$ be a simple subalgebra of $A$.
\begin{enumerate}
\item {\bf (Double Centralizer Theorem)} $C_AC_A(B)=B$.
\item The algebra $C_A(B)$ is a simple algebra. 
\item $\dim_K(B)\, \dim_K(C_A(B))=\dim_K(A)$.
\item Suppose that $B\in \fCK$. Then $C_A(B)\in \fCK$ and $A=B\t C_A(B)$.
\end{enumerate}
\end{theorem}
 
\begin{definition}
Let $A\in \fCK$ and $B$ be a simple subalgebra of $A$. Then the pair $(B,C_A(B))$ (of necessarily  simple)  $K$-subagebras of $A$ is called {\bf a double centralizer pair}, a {\bf dcp},  for short.
\end{definition}

If $(B,C_A(B))$ is a dcp then so is $(C_A(B),B)$, and vice versa.

\begin{theorem}\label{DCThm-Mod}%\marginpar{DCThm-Mod}
{\bf (Double Centralizer Theorem for semi-simple modules)} Let $A$ be a simple artinian ring and $M$ be a finitely generated  semi-simple $A$-module. Then $E(M_{E({}_AM)})=A$.
\end{theorem}

\begin{proof} We can assume that $A=M_n(D)$ is  matrix ring over a division ring $D$. Then $U=AE_{11}=D^n$ is a simple $A$-module and  $M\simeq U^m$ for  some natural number $m\geq 1$. Clearly, $M= M_{n,m}(D)$ where $M_{n,m}(D)$ is the set of $n\times m$ matrices with coefficients in the division ring $D$. Then   $E({}_AM)=M_m(D)$,  and so $E(M_{M_m(D)})= E(M_{n,m}(D)_{M_m(D)})  =M_n(D)=A$.
\end{proof}

 {\bf Description of subalgebras of the endomorphism algebra $E(L/K)$ that contain the field $L$.} 
 Let $A$ be an algebra. For an $A$-module  $M$, the set $\ann_A(M):=\{ a\in A\, | \, aM=0\}$ is called the {\em annihilator} of $M$. The annihilator $\ann_A(M)$ is an ideal of the algebra $A$. A module  is called a {\em faithful} module if its annihilator is equal to zero.

\begin{lemma}\label{a5May25}%\marginpar{a5May25}
A finite dimensional algebra $A$ is a simple algebra iff there is a faithful simple $A$-module. 
\end{lemma}

\begin{proof} $(\Rightarrow)$ All nonzero modules of a simple algebra are faithful.

$(\Leftarrow)$ Let $\gr$ be the radical of the finite dimensional  algebra $A$.  
 Suppose that $M$ is a faithful simple $A$-module. Then  $\gr M=0$ (by the simplicity of $M$ and  definition of the radical). This  implies that $\gr =0$, by the faithfulness of the $A$-module $M$. So, the algebra $A$ is a semisimple finite dimensional algebra, i.e. it is  a finite  direct product of simple finite dimensional algebras. For a semisimple  but not simple finite dimensional algebra, every simple module is unfaithful. Therefore, the algebra $A$ is a simple algebra. 
\end{proof}

 For an algebra $A$ and its subalgebra $B$, let $\CA (A, B)$ (resp., $\CA (A, B, sim)$) be the sets of all (resp., simple) subalgebras of $A$ that contain $B$. 

\begin{proposition}\label{A24Mar25}%\marginpar{A24Mar25}
Let $K$ be a field of arbitrary characteristic,   $A\in \fCK$ and $L$ is a strongly maximal subfield of $A$. Then:

\begin{enumerate}

\item   $\CA ( E, L)=\CA ( E, L, sim)$.

\item The map 
$$\CF (L/K)\ra \CA ( E, L)=\CA (A, L, sim), \;\; M\mapsto C_A(M)$$  is a bijection with inverse $B\mapsto C_A(B)$. 

\item For all fields $M\in \CF (L/K)$, $[M:K]\dim_K(C_A(M))=\dim_K(A)$.
\end{enumerate}
\end{proposition}

\begin{proof} 1. We have to show that each  algebra $B\in \CA ( E, L)$ is a simple algebra. The inclusion $L\subseteq B$ implies the inclusion 
$$
C_E(B)\subseteq C_E (L)=L.
$$
 The last equality follows from the fact that $L\in \Max_s (E)$ (statement 1). Clearly, the field $L$ is a faithful simple $E$-module. Since $L\subseteq B$, the field $L$ is also a faithful simple $B$-module. By Lemma \ref{a5May25}, the finite dimensional algebra $B$ is a simple algebra.

2. Let us show that the map is a well-defined map. Since the field $M/K\in \CF (L/K)$ is a simple algebra, its centralizer $C_A(M)$ is also a simple subalgebra of $A$, by Theorem \ref{DCThm}.(2). Clearly, $L\subseteq C_A(M)$  (since $M\subseteq L$), and so $C_A(M)\in \CA:= \CA (A, L, sim)$. 

The field $L$ is a strongly maximal subfield of $A$. Hence, $C_A(L)=L$. For every algebra $B\in \CA$, $ L\subseteq B$. Hence, $C_A(B)\subseteq C_A(L)=L$, and so 
$$
C_A(B)\in \CF (L/K).
$$ 
By the Double Centralizer Theorem (Theorem \ref{DCThm}.(1)), $C_A(C_A(M))=M$   for all $M\in \CF (L/K)$ and $C_A(C_A(B))=B$ for all $B\in \CA$, and the result follows. 

3. Since  fields in $\CF (L/K)$ are simple  $K$-algebras, statement 3 follows from  Theorem \ref{DCThm}.(3).
\end{proof}

Theorem \ref{B24Mar25} is a  description of subalgebras of the endomorphism algebra $E(L/K)$ that contain the field $L$. This is one of the key results of the paper.

\begin{theorem}\label{B24Mar25}%\marginpar{B24Mar25}
Let $K$ be a field of arbitrary characteristic, 
 $L/K$ be a finite field extension of degree $n:=[L:K]<\infty$ and  $E:=E(L/K )\simeq M_n(K)$. Then:
 \begin{enumerate}
 
\item $L\in \Max_s (E)$.

\item $\CA ( E, L)=\CA ( E, L, sim)=\{ E(L/M)\, | \, M\in \CF (L/K) \}$.

\item  The map 
$$\CF (L/K)\ra \CA ( E, L)=\CA (E, L, sim), \;\; M\mapsto C_E(M)=E(L/M)$$  is a bijection with inverse $B\mapsto C_E(B)$.

\item For all fields $M\in \CF (L/K)$, $[M:K]\dim_K(C_E(M))=[L:K]^2$.

\end{enumerate}
\end{theorem}  

\begin{proof} 1. Since $\Deg (E)=\Deg (M_n(K))=n=[L:K]$, we have that $L\in \Max_s (E)$. 

2. (i)  $\CA ( E, L)=\CA ( E, L, sim)$: Since $E(L/K)\in \fC (K)$ and $L\in \Max_s (E)$, the statement (i) follows from Theorem \ref{A24Mar25}.(1).

(ii)  $\CA ( E, L, sim)=\{ C_E(M)=E(L/M)\, | \, M\in \CF (L/K) \}$: For all $M\in \CF (L/K)$, 
$C_E(M)=E(L/M)$. Now,  
$$
\CA ( E, L, sim)=\{ C_E(M)=E(L/M)\, | \, M\in \CF (L/K) \},
$$ 
by Proposition \ref{A24Mar25}.

3. Statement 3 follows from statement 2 and Proposition \ref{A24Mar25}.  

4. By Proposition \ref{A24Mar25}.(3), $[M:K]\dim_K(C_E(M))=\dim_K(E)=[L:K]^2$.
\end{proof}

\begin{corollary}\label{aB24Mar25}%\marginpar{aB24Mar25}
Let $K$ be a field of arbitrary characteristic, 
 $L/K$ be a finite field extension,  $B\in \CA ( E(L/K), L)$ and $M:=C_E(B)\in \CF (L/K)$ (Theorem  \ref{B24Mar25}.(3)). Then:
 \begin{enumerate}
 
 \item  $B=E(L/M)\simeq M_{[L:M]}(M)\in \fC (M)$.
 
\item The field $L$ is the only (up to isomorphism)  simple $B$-module.

\item $\End_B(L)=M$.

\end{enumerate}
\end{corollary}

\begin{proof} 1.  By Theorem  \ref{B24Mar25}.(3), $B=E(L/M)\simeq M_{l:M]}(M)\in \fC (M)$.

2 and 3. Statements 2  and 3 follow from statement 1. 
\end{proof}

{\bf Embeddings of endomorphism algebras into endomorphism algebras in the case of  fields.} 
Let $K\subseteq M\subseteq L$ be   finite field  extensions of the  field $K$. Recall that $L$ is a subfield of the endomorphism $K$-algebra $E(L/K)$ via a monomorphism 
$$L\ra E(L/K), \;\; l\mapsto l\cdot : L\ra L , \l \mapsto l\l.$$
Let $e=(e_1, \ldots , e_m)$ be  an $M$-basis of the vector space $L$ over the field $M$, i.e. $L=\bigoplus_{i=1}^mMe_i$. Then the simple algebra $E(M/K)$ can be seen as a subalgebra of $E(L/K)$ via the monomorphism that {\em depends} on the choice of the basis $e$:
%\marginpar{emonEMK}
\begin{equation}\label{emonEMK}
\th_e: E(M/K)\ra E(L/K), \;\;\th_e(\phi) \mapsto \phi:  L\ra L, \;\; \sum_{i=1}^m m_ie_i\mapsto \sum_{i=1}^m \phi (m_i)e_i.
\end{equation}
Since the algebra $E(M/K)$ is a simple algebra, the images of the algebra $E(M/K)$ for different choices of the basis $e$ are conjugate in $E(L/K)$ (by the Norther-Skolem Theorem). That is 
 for two $M$-bases $e$ and $e'$ of $L$, there is a unit $u=u(e,e')\in E(L/K)^\times$ such that 
 $$ 
 \th_{e'}(\phi) = u\th_{e}(\phi)u^{-1}\;\; {\rm for \; all}\;\; \phi \in E(M/K).$$
The embedding $\th_e$ is a diagonal embedding. If we choose a $K$-basis $f=(f_1 \ldots , f_l)$ for the $K$-vector space $M$ and consider a $K$-basis of $L$  which is the  product 
of bases $f$ and $e$, $fe:=(f_1e_1 \ldots , f_le_1, \cdots , f_1e_m \ldots , f_le_m)$, then the matrix $[\th_e(\phi)]_{fe}$ of the $K$-linear map $\th_e(\phi):L\ra L$ in the basis $fe$ is the diagonal matrix
%\marginpar{emonEMK-1}
\begin{equation}\label{emonEMK-1}
[\th_e(\phi)]_{fe}=
\begin{pmatrix}
[\phi]_f &0 &0 &\cdots  & 0\\ 0 & [\phi]_f& 0&\cdots &0\\  & & \cdots & &\\
0&0 & \cdots &0& [\phi]_f
\end{pmatrix}
\end{equation}
where on the diagonal is the matrix $[\phi]_f$ of the $K$-linear map $\phi : M\ra M$ in the $K$-basis $f$.

\begin{proposition}\label{A5May25}%\marginpar{A5May25}
Let $K\subseteq M\subseteq L$ be   finite field  extensions of the  field $K$ and $e=(e_1, \ldots , e_m)$ be  an $M$-basis of the vector space $L$.
Then a  subalgebra of $E(L/K)$ which is generated by its  two simple subalgebras,  $\th_e(E(M/K))$ and $E(L/M)$, is equal to $E(L/K)$. 

\end{proposition} 

\begin{proof} Let $A$ be the subalgebra of $E(L/K)$ which is generated by the  subalgebras  $\th_e(E(M/K))$ and $E(L/M)$. Since $L\subseteq A$, the algebra $A$ is a simple algebra (Theorem \ref{B24Mar25}.(2)) and (where $E=E(L/K)$)
$$
 C_E(A)=C_E(E(L/M))\cap C_E\Big(\th_e(E(M/K))\Big) =M\cap C_E\Big(\th_e(E(M/K))\Big) =K.
 $$
The algebra $A$ is a simple subalgebra of $E(L/K)$. 
 By the Double Centralizer Theorem, 
 $$A=C_EC_E(A)=C_E(K)=E(L/K),$$
 as required.
\end{proof}

Clearly, the field $L=M\t N$ is a tensor product of two vector spaces $M$ and   $N:=\bigoplus_{i=1}^mKe_i$. Therefore, the central simple $K$-algebra  
$$E(L/K)=E(M\t N/K)=E(M/K)\t E(N/K)$$
is a tensor product of two central simple $K$-algebras $E(M/K)$ and $E(N/K)$. The equality above depends on the choice of the $M$-basis $e$ for $L$ since the vector space $N$ does.

\begin{lemma}\label{aA5May25}%\marginpar{aA5May25}
Let $K\subseteq M\subseteq L$ be   finite field  extensions of the  field $K$, $e=(e_1, \ldots , e_m)$ be  an $M$-basis of the vector space $L$  and $N=\bigoplus_{i=1}^mKe_i$.
Then $E(L/K)=E(M/K)\t E(N/K)$, $C_{E(L/K)}(E(M/K))=E(N/K)$, $C_{E(L/K)}(E(N/K))=E(M/K)$  and 
 $E(N/K)\subseteq E(L/M)$. 
\end{lemma} 

\begin{proof} The equality $E(L/K)=E(M/K)\t E(N/K)$  was proven above. It follows that 
$$
C_{E(L/K)}(E(M/K))=E(N/K)\;\; {\rm  and }\;\; C_{E(L/K)}(E(N/K))=E(M/K).
$$ 
Now, the containment $M\subseteq E(M/K)$ implies that 
 $E(N/K)=C_{E(L/K)}(E(M/K))\subseteq C_{E(L/K)}(M)\subseteq E(L/M)$. 
\end{proof}

%%%%%%%%%%%%%%%%%%%   Section 3    %%%%%%%%%%%

\section{Algebras of differential operators on finite field extensions}\label{ALG-DIF-OPS}%\marginpar{ALG-DIF-OPS}

Lemma \ref{aC24Mar25} and Theorem \ref{AC24Mar25} give  explicit descriptions of the algebras of differential operators $\CD (L/K)$ for an arbitrary finite field extension $L/K$ in characteristic zero and  in prime characteristic, respectively. Proposition \ref{VB-aC24Mar25} is a criterion for the algebra of differential operators $\CD (L/K)$ being a commutative algebra or equivalently $\CD (L/K)=L$. 
Theorem \ref{23May25} shows that  $L_{dif}=L^{sep}=L^{\CD (L/K)_+}$. Corollary \ref{b16Apr25} is  a criterion for $\CD (L/K)=\D (L/K)$.
\\

 Let  $\CD (L)=\CD_K(L)=\CD (L/K)$ be the {\bf algebra of $K$-differential operations} on the $K$-algebra $L$ and the $K$-subalgebra 
 $$
 \D (L)=\D_K (L)=\D (L/K):=L\langle \Der_K(L)\rangle
 $$ 
 of $\CD (L)$   is called the {\bf derivation algebra} of  $L$. For a subset $\Phi$ of $K$-linear maps on $L$, let $L^{\Phi}:=\bigcap_{\phi \in \Phi}\ker (\phi)$. If $\Phi\subseteq \Der_K(L)$ then $L^{\Phi}$ is a $K$-subalgebra of $L$, the {\bf algebra of  $\Phi$-invariants/constants}. Let $\End_K(L)$ be the $K$-endomorphism algebra of the $K$-vector space $L$, i.e.  the set of  of all $K$-linear maps from $L$ to $L$.

\begin{lemma}\label{a20May25}%\marginpar{a20May25}
Let $K\subseteq M\subseteq L$ be   finite field  extensions of the  field $K$. Then $\CD (L/M)\subseteq \CD (L/K)$ and $\CD (L/M)=E(L/M)\cap \CD (L/K)
=C_{\CD (L/K)}(M)$.
\end{lemma} 

\begin{proof} Since the algebra $E(L/M)$ is a subalgebra of $E(L/K)$ such  that $\ad_l(E(L/M))\subseteq E(L/M)$ for all elements $l\in L$, we have that 
$$
\CD (L/M)=E(L/M)\cap \CD (L/K)
=C_{\CD (L/K)}(M)\subseteq \CD (L/K).
$$
\end{proof}

 \begin{lemma}\label{aC24Mar25}%\marginpar{aC24Mar25}
Let $K$ be a field of arbitrary  characteristic, 
 $L/K$ be a finite separable field extension.  Then $\CD (L/K)=L$.
\end{lemma}

\begin{proof} By the assumption, the  field $L/K$ is  a finite separable field extension. Therefore, $\Der (L/K)=0$, and so $\CD (L/K)=L$, by the definition of the algebra of differential operators. 
\end{proof}

\begin{definition}
For a finite field extension $L/K$, the set  
%\marginpar{LDL+}
\begin{equation}\label{LDL+}
L_{dif}:=(L/K)_{dif}:=C_L(\CD (L/K))=\{ l\in L\, | \, l\d=\d l\;\; {\rm for\; all}\;\; \d \in \CD (L/K)\}
\end{equation}
is a subfield of the field $L$ that contains $K$. The field $L_{dif}/K$ is called  the {\bf dif-subfield} of the field extension $L/K$.
\end{definition}

{\bf Description of the algebra of differential operators $\CD (L/K)$ for an arbitrary finite field extension $L/K$.} Lemma \ref{aC24Mar25} and Theorem \ref{AC24Mar25} give  explicit descriptions of the algebras of differential operators $\CD (L/K)$ for an arbitrary finite field extension $L/K$ in characteristic zero and  in prime characteristic, respectively.

Proposition \ref{A21May25} is an explicit description of $\Der (L/K)$ for a finite field extension $L/K$ where $\char (K)=p$. Let $L^{sep}L^p$ be a  subfield of $L$    generated by the subfields $L^{sep}$ and $L^p:=\{ l^p\, | \, l\in L\}$. 

\begin{proposition}\label{A21May25}%\marginpar{A21May25}

Let $L/K$ be a finite field extension and $\char (K)=p$. Then: 
\begin{enumerate}

\item $L\simeq \bigotimes_{i=1}^nL^{sep}L^p[x_i]/(x_i^p-\l_i)$ is a field  where $\l_1, \ldots , \l_n\in L^{sep}L^p$ and $\t = \t_{L^{sep}L^p}$.

\item $\O_{L/K}=\O_{L/L^{sep}L^p}=\bigoplus_{i=1}^nLdx_i$.

\item $[\O_{L/K}:L]=n$ ($\Leftrightarrow [\O_{L/K}:K]=np^n)$.

\item $\Der (L/K)=\Der (L/L^{sep}L^p)=\bigoplus_{i=1}^nL\der_i$ where $\der_i=\der_{x_i}=\frac{\der}{\der x_i}$.

\item $[\Der (L/K):L]=n$ ($\Leftrightarrow [\Der (L/K):K]=np^n)$.

\item $L^{\Der (L/K)}=L^{\Der (L/L^{sep}L^p)}=L^{sep}L^p$.

\item $E(L/L^{sep}L^p)=\CD (L/L^{sep}L^p)=\D (L/L^{sep}L^p)=\bigoplus_{\alpha \in \N^n_{<p}}L\der^\alpha$ where 
$\alpha = (\alpha_1, \ldots , \alpha_n)$ and $\der^\alpha:=\prod_{i=1}^n\der_i^{\alpha_i}$.
\end{enumerate}
\end{proposition}

\begin{proof} 2. Let $d: L\ra \O_{L/K}$, $l\mapsto dl$ be the canonical $K$-derivation map. Then $d(L^{sep}L^p)=0$, and so $ \O_{L/K}=\O_{L/L^{sep}L^p}$. For all elements $l\in L$, $l^p\in L^p$. Therefore
$\O_{L/K}=\O_{L/L^{sep}L^p}=\bigoplus_{i=1}^nLdx_i$ for some elements $x_1, \ldots , x_n\in L$.

4. The field extension $L/L^{sep}L^p$ is a purely inseparable field extension of exponent 1, i.e. $l^p\in L^{sep}L^p$ for all elements $l\in L$. Recall  that $\Der (L/K)=\Hom_L(\O_{L/K}, L)$.
Now, statement 4 follows from statement 2.

1. Statement 1 follows from statement 4. 

3 and 5. Statements 3 and 5 are obvious.

6. Statement 6 follows from statements 1 and 4. 

7. By statement 4, $\D (L/L^{sep}L^p)=\bigoplus_{\alpha \in \N^n_{<p}}L\der^\alpha\subseteq E(L/L^{sep}L^p)$. Since $\D (L/L^{sep}L^p)\subseteq \CD (L/L^{sep}L^p)\subseteq  E(L/L^{sep}L^p)$ and 
$$ 
\dim_{L^{sep}L^p}\Big(\D (L/L^{sep}L^p)\Big)=p^{2n}=[L:L^{sep}L^p]^2=\dim_{L^{sep}L^p}\Big(E (L/L^{sep}L^p)\Big)
$$
we must have that   $\D (L/L^{sep}L^p)=\CD (L/L^{sep}L^p)=E(L/L^{sep}L^p)$.
\end{proof}

 Notice that $L\subseteq \CD (L/K)$. Proposition \ref{VB-aC24Mar25} is a criterion for  $ \CD (L/K)=L$.

\begin{proposition}\label{VB-aC24Mar25}%\marginpar{VB-aC24Mar25}
Let  $L/K$ be a finite  field extension over the field  $K$  of arbitrary  characteristic.  Then the following statements are equivalent:

\begin{enumerate}

\item $\CD (L/K)=L$.

\item $\Der(L/K)=0$.

\item  $\O_{L/K}=0$.

\item  $L/K$ is a separable field extension  ($\Leftrightarrow$ either $\char (K)=0$  or  $\char (K)=p$ and $L=L^{sep}L^p$). 
 
 \item $L_{dif}=L$.
 
 \item The algebra $\CD (L/K)$ is a commutative algebra.

\end{enumerate}
\end{proposition} 

\begin{proof} $(1\Leftrightarrow 2)$ The equivalence follows from the definition of the algebra of differential operators.

$(2\Leftrightarrow 3)$ The module of K\"{a}hler  differentials $\O_{L/K}$ is a left vector space over the field $L$ and $\Der (L/K)=\Hom_L(\O_{L/K},L)$. 
Therefore, $\Der (L/K)=0$ iff $\O_{L/K}=0$.
 
 $(2\Leftrightarrow 4)$ Suppose that  $\char (K)=0$. Then 
 the field extension $L/K$ is a separable field  extension, and so $\Der (L/K)=0$. 
 
 Suppose that  $\char (K)=p$. Then, by Proposition \ref{A21May25},  
 $\Der (L/K)=0$ iff $L=L^{sep}L^p$ iff $L=L^{sep}L^{p^n}$  for all $n\geq 1$ 
 $$
 (L=L^{sep}(L^{sep}L^p)^p=L^{sep}L^{p^2}=\cdots = L^{sep}L^{p^n}=\cdots)
 $$
  iff $L=L^{sep}$ (since $L/L^{sep}$ is a finite field extension we have that $L^{p^n}\subseteq L^{sep}$ for some $n\geq 1$).

 $(1\Rightarrow 5)$  If $\CD (L/K)=L$ then 
 $$L_{dif}=C_{E(L/K)}(\CD (L/K))=C_{E(L/K)}(L)=E(L/L)=L.$$

  $(5\Rightarrow 2)$ If $L_{dif}=L$ then $\Der(L/K)=0$ since for all derivations $\d \in \Der (L/K)$ and elements $l\in L=L_{dif}$, 
  $$\d (l)=[\d, l]=0.$$
  
$(1\Rightarrow 6)$ The implication is obvious.

 $(6\Rightarrow 2)$ Suppose that  the algebra $\CD (L/K)$ is a commutative algebra. Since $L\subseteq \CD (L/K)$, we must have $\Der (L/K)=0$.
  \end{proof}

\begin{theorem}\label{AC24Mar25}%\marginpar{AC24Mar25}
Let $K$ be a field of  characteristic $p$, 
 $L/K$ be a finite field extension. Then:
 \begin{enumerate}
 
 \item $\CD (L/K)=E(L/L_{dif})\simeq M_{[L:L_{dif}]}(L_{dif})\in \fC (L_{dif})$,  $\Deg (\CD (L/K))=[L:L_{dif}]$ and $\dim_K(\CD (L/K))=[L:L_{dif}]^2[L_{dif}:K]$.
 
 \item $L\in \Max_s (\CD (L/K))$.

 \item $Z(\CD (L/K))=L_{dif}$.
 
 \item $L$ is the only (up to isomorphism) simple 
 $\CD (L/K)$-module.
 
 \item $\End_{\CD (L/K)}(L)=L_{dif}$.
 
 \item $C_{E(L/K)}(\CD (L/K))=L_{dif}$ and $C_{E(L/K)}(L_{dif})=\CD (L/K)$.

\end{enumerate}
\end{theorem}

\begin{proof} 1. Clearly, $L\subseteq \CD (L/K)\subseteq E(L/K)\simeq M_n(K)\in \fCK$ and $L\in \Max_s( E(L/K))$. 

Let us show that the algebra $ \CD (L/K)$ is a {\em simple} algebra. Let $I$ be a nonzero ideal of the algebra $\CD (L/K)$ and $a\in I\backslash \{ 0\}$. If $a\in L^\times$ then $I=\CD (L/K)$. Suppose that $a\in \CD (L/K)\backslash L$.  It follows from the definition of the algebra of differential operators $\CD (L/K)$ that there are elements $a_1, \ldots , a_s\in L$ such that
 $$ l=\ad_{a_1}\cdots \ad_{a_s}(a)\in L^\times .$$
 Clearly, $l\in I$, and so $I=\CD (L/K)$. Therefore, the algebra $\CD (L/K)$ is a simple algebra.   Since $\CD(L/K)\in \CA (E(L/K), L, sim)$, 
 $$
 L_{dif}=C_{E(L/K)}(\CD (L/K))\in \CF (L/K)\;\; {\rm  and}\;\;\CD (L/K)=C_{E(L/K)}(L_{dif})=E(L/L_{dif}),$$
  by Theorem  \ref{B24Mar25}.(3). It is obvious that 
$E(L/L_{dif})\simeq M_{[L:L_{dif}]}(L_{dif})\in \fC (L_{dif})$,  $\Deg (\CD (L/L_{dif}))=[L:L_{dif}]$ and $\dim_K(\CD (L/K))=[L:L_{dif}]^2[L_{dif}:K]$.

2.  By statement 1, $[L:L_{dif}]=\Deg(E(L/L_{dif}))=\Deg(\CD (L/L_{dif}))$, and so  $L\in \Max_s (\CD (L/K))$. 

3. Statement 3 follows from statement 1.

4. The field $L$ is a simple $\CD (L/K)$-module and  $\CD (L/K)\in \fC (L_{dif})$. Therefore, the field $L$ is  the only (up to isomorphism) simple 
 $\CD (L/K)$-module.

5. Statement 5 follows from the fact that 
$\CD (L/K)=E(L/L_{dif})\simeq M_{[L:L_{dif}]}(L_{dif})$ (statement 1).

6. By statement 1, $\CD (L/K)=E(L/L_{dif})$, and so  $$
C_{E(L/K)}(\CD (L/K))= C_{E(L/K)}(E(L/L_{dif}))=L_{dif}\;\; {\rm and}\;\; C_{E(L/K)}(L_{dif})=E(L/L_{dif})=\CD (L/K).
$$
\end{proof}

Corollary \ref{a16Apr25} shows that $\CD (L/K)=\CD (L/L_{dif})$ and $(L/L_{dif})_{dif}=L_{dif}$.

\begin{corollary}\label{a16Apr25}%%\marginpar{a16Apr25}
Let $K$ be a field of  characteristic $p$ and  
 $L/K$ be a finite field extension. Then  $\CD (L/K)=\CD (L/L_{dif})$ and $(L/L_{dif})_{dif}=L_{dif}$. Furthermore,  

\begin{enumerate}
 
 \item $\CD (L/K)=\CD (L/L_{dif})=E(L/L_{dif})\simeq M_{[L:L_{dif}]}(L_{dif})\in \fC (L_{dif})$,  $\Deg (\CD (L/L_{dif}))=[L:L_{dif}]$ and $\dim_K(\CD (L/L_{dif}))=[L:L_{dif}]^2[L_{dif}:K]$.
 
 \item $L\in \Max_s (\CD (L/L_{dif}))$.

 \item $Z(\CD (L/L_{dif}))=L_{dif}$.
 
 \item $L$ is the only (up to isomorphism) simple 
 $\CD (L/L_{dif})$-module.
 
 \item $\End_{\CD (L/L_{dif})}(L)=L_{dif}$.

\end{enumerate}
\end{corollary}

\begin{proof}  By Theorem \ref{AC24Mar25}.(1), $\CD (L/K)=E(L/L_{dif})\subseteq E(L/K)$. Therefore, 
$$
E(L/L_{dif})=\CD (L/K)\subseteq \CD (L/L_{dif})\subseteq E(L/L_{dif}),
$$
and so $\CD (L/K)=\CD (L/L_{dif})=E(L/L_{dif})\simeq M_{[L:L_{dif}]}(L_{dif})\in \fC (L_{dif})$. Now, 
$$
L_{dif}\subseteq (L/L_{dif})_{dif}=C_{E(L/L_{dif})}(\CD (L/L_{dif}))\subseteq C_{E(L/K)}(\CD (L/K))
=L_{dif},$$
and  so $(L/L_{dif})_{dif}=L_{dif}$. Now, statements 1--5 follow from Theorem \ref{AC24Mar25}.
\end{proof}

{\bf The equality   $L_{dif}=L^{sep}=L^{\CD (L/K)_+}$.} Let $L/K$ be a finite field extension. Then the     field $L$ is a left $E(L/K)$-module which is also a left $\CD (L/K)$-module  since $\CD (L/K)\subseteq E(L/K)$. The annihilator of the element $1\in L$,   
$$
\CD (L/K)_+:=\ann_{\CD (L/K)}(1):=\{ \d \in \CD (L/K)\, | \, \d (1)=0\}, 
$$ 
is a left ideal of the algebra  of $K$-linear  differential operators $\CD (L/K)$, 

%\marginpar{CDLK+}
\begin{equation}\label{CDLK+}
\CD (L)=L\oplus \CD (L/K)_+\;\; {\rm  and}\;\; {}_{\CD (L/K)}L=\CD (L/K)/\CD (L/K)_+.
\end{equation}
So,  every differential operator $\d \in \CD (L/K)$ is a unique sum
%\marginpar{CDLK+1}
\begin{equation}\label{CDLK+1}
 \d = \d(1)+\d_+ \;\; {\rm where} \;\; \d (1) \in L\;\; {\rm and}\;\; \d_+:= \d -\d (1)\in \CD (L/K)_+.
\end{equation}
Elements of the set $\CD (L/K)_+$ are called {\bf differential operators without constant term}.

Any associative algebra $A$ is a Lie algebra $(A, [\cdot, \cdot])$ where the Lie bracket $[\cdot, \cdot]$ is the commutator: For all elements $a,b\in A$, $[a,b]:=ab-ba$. By the definition, the left ideal $\CD (L/K)_+$  is a Lie subalgebra of the Lie algebra $(\CD (L/K), [\cdot, \cdot])$.

\begin{definition}
\begin{eqnarray*}
 L^{\Der (L/K)}&:=&\{ l\in L\, | \, \d (l)=0\;\; {\rm for\; all}\;\; \d \in \Der (L/K)\}, \\
L^{\CD (L/K)_+}&:=&\{ l\in L\, | \, \d (l)=0\;\; {\rm for\; all}\;\; \d \in \CD (L/K)_+\}.
\end{eqnarray*}
 The elements of the sets $L^{\Der (L/K)}$ and $L^{\CD (L/K)_+}$ are called $\Der (L/K)${\bf -constants}  and $\CD (L/K)_+${\bf -constants}, respectively.
 \end{definition}
 
  \begin{proposition}\label{A27Apr25}%\marginpar{A27Apr25}
Let $L/K$ be a   finite field extension of   characteristic $p$. Then:

\begin{enumerate}

\item  $L_{dif}=\{ l\in L\, | \, \d (l)=0$ and $ l\d =(\d l)_+$ for all $\d \in \CD (L/K)\}$.

\item $L^{sep}\subseteq L_{dif}\subseteq L^{\CD (L/K)_+}\subseteq  L^{\Der (L/K)}=L^{sep}L^p$.

\end{enumerate}
\end{proposition}

\begin{proof} 1. By the definition, $L_{dif}=\{ l\in L\, | \, \d l=l\d$ for all $\d \in \CD (L/K)_+\}$.
 It follows from the equality $ l\d=\d l=\d (l)+(\d l)_+$ that 
 $$
 \d (l)=l\d-(\d l)_+\in L\cap \CD (L/K)_+=\{ 0\}.
 $$
  Therefore, $l\in L_{dif}$ iff $\d (l)=0$ and $ l\d =(\d l)_+$ for all $\d \in \CD (L/K)_+$.

2. By Proposition \ref{A21May25}.(6), $L^{\Der (L/K)}=L^{sep}L^p$.
By statement 1, $L_{dif}\subseteq  L^{\Der (L/K)_+}$. Since $\Der(L/K)\subseteq \CD (L/K)_+$, we have the inclusion  
$$L^{\CD (L/K)_+}\subseteq  L^{\Der (L/K)}.
$$
 It remains to show that $L^{sep}\subseteq L_{dif}$.
 Let $l\in L^{sep}$ and $f(x)=\sum_{i=0}^d\l_ix^i\in K[x]$ be the  minimal polynomial of the element $l$. The algebra $\CD (L/K)=\bigcup_{i\geq 0} \CD (L/K)_i$ admits the degree filtration. We have to show that 
 $$[\d_i,l]=0\;\; {\rm  for\; all\;  elements}\;\; \d_i\in \CD_i:=\CD (L/K)_i\;\; {\rm  and}\;\;i\geq 0.
 $$
We use induction on $i\geq 0$.  The initial case, $i=0$, is obvious as $\CD_0=L$. Suppose that $i>0$ and  $[\d_j,l]=0$ for all elements $\d_j\in \CD _j$ and $j=0, \ldots , i-1$.  Let $\d\in \CD_i$. Then $[\d, l]\in \CD_{i-1}$ and so the elements $l$ and $[\d  ,l]$ commute. Since $l\in L^{sep}$, $f'(l)\neq 0$. It follows from the equalities 
$$ 0=[\d,0]=[\d,f(l)]=f'(l)[\d,l]$$
that $[\d,l]=0$, as required.
\end{proof}

\begin{theorem}\label{23May25}%\marginpar{23May25}
Let $L/K$ be a finite field extension. Then $L_{dif}=L^{sep}=L^{\CD (L/K)_+}$.
\end{theorem}

\begin{proof} (i)  $L_{dif}=L^{sep}$: By Theorem \ref{AC24Mar25}, $E(L/L_{dif})=\CD (L/K)$. By Proposition \ref{A27Apr25}.(2), $ L^{sep}\subseteq L_{dif}$. Therefore, $\CD (L/K)=\CD (L/L^{sep})$. If $L=L^{sep}$ then 
$$
E(L/L_{dif})=\CD (L/K)=\CD (L/L^{sep})=\CD (L/L)=L,$$
and so $L_{dif}=L=L^{sep}$.

If $L\neq L^{sep}$ then the field $K$ has prime characteristic, say $p$, and the finite field extension $L/L^{sep}$ is purely inseparable. Then, by  Theorem \ref{C24Mar25}.(1), $\CD (L/L^{sep})=E(L/L^{sep})$.
Therefore, 
$$
E(L/L_{dif})=\CD (L/K)=\CD (L/L^{sep})=E(L/L^{sep}
),$$ and so $L_{dif}=L^{sep}$.

(ii) $L^{sep}=L^{\CD (L/K)_+}$: Recall that $\CD (L/K)=L\oplus \CD (L/K)_+$, ${}_{\CD (L/K)}L=\CD (L/K)/\CD (L/K)_+$ and $\CD (L/K)=E(L/L^{sep})$, see above. Now,
\begin{eqnarray*}
L^{\CD (L/K)_+}&=&\End_{\CD (L/K)}\Big(\CD (L/K)/\CD (L/K)_+\Big)=\End_{\CD (L/K)}(L)\\
&=& \End_{E(L/L^{sep})}(L)=L^{sep}.
\end{eqnarray*}
\end{proof}
Theorem \ref{23May25} provides a characterization of the subfield $L^{sep}$ of separable elements of the field $L$ as the set of $\CD (L/K)_+$-constants.\\

{\bf Criterion for $\CD (L/K)= \D (L/K)$.} Corollary \ref{b16Apr25} provides a criterion for $\CD (L/K)=\D (L/K)$.

\begin{corollary}\label{b16Apr25}%\marginpar{b16Apr25}
Let $K$ be a field of  characteristic $p$ and  
 $L/K$ be a finite field extension. Then the following statements are equivalent:  

\begin{enumerate}
 
 \item $\CD (L/K)=\D (L/K)$.
 
 \item $L^p\subseteq L^{sep}$.

 \item The field extension $L/L^{sep}$ is a purely inseparable field extension of exponent 1.

 \item $L=\bigotimes_{i=1}^n L^{sep}[x_i]/(x_i^p-\l_i)$ for some $\l_i\in L^{sep}$ where $\t = \t_{L^{sep}}$.

\end{enumerate}
\end{corollary}

\begin{proof}  $(1\Leftrightarrow 2)$ Notice that the finite field extensions $L/L^{sep}$ and $L/L^{sep}L^p$ are purely inseparable  and 
\begin{eqnarray*}
 \CD (L/K)&\stackrel{{\rm Thm}\, \ref{AC24Mar25}.(1)}{=} &E (L/L_{dif})\stackrel{{\rm Thm}\, \ref{23May25}}{=} E(L/L^{sep}),\\
 \D (L/K)&= &L\langle \Der (L/K)\rangle\stackrel{{\rm Pr.}\, \ref{A21May25}.(4)}{=} L\langle \Der (L/L^{sep}L^p)\rangle\stackrel{{\rm Pr.}\, \ref{A21May25}.(7)}{=} E(L/L^{sep}L^p).
\end{eqnarray*}
Therefore, $\CD (L/K)=\D (L/K)$ iff $E(L/L^{sep})=E(L/L^{sep}L^p)$ iff $L^{sep}=L^{sep}L^p$ iff $L^p\subseteq L^{sep}$.

 $(2\Leftrightarrow 3)$ Clearly, the   purely inseparable field extension $L/L^{sep}$ is  of exponent 1 iff $L^p\subseteq L^{sep}$.

  $(3\Leftrightarrow 4)$  Since statements 1--3 are equivalent, the equivalence follows from Proposition \ref{A21May25}.
  \end{proof}

{\bf $\CD (L/K)=\bigotimes_{i=1}^n\CD (L_i/K)$ where $L=\bigotimes_{i=1}^nL_i$.}

\begin{proposition}\label{7Jun25}%\marginpar{7Jun25}
Suppose that $L_i/K$ $(i=1, \ldots , n)$ are finite field extensions such that their tensor product $L=\bigotimes_{i=1}^nL_i$ is a field. Then: 
\begin{enumerate}

\item $L^{sep}=\bigotimes_{i=1}^n L^{sep}_i$.

\item $\CD \Big(\bigotimes_{i=1}^nL_i/K\Big)=\bigotimes_{i=1}^n\CD (L_i/K)$.

\end{enumerate}
\end{proposition}

\begin{proof} 1. Clearly, $L^{sep}\supseteq \bigotimes_{i=1}^n L^{sep}_i$ and the finite field extension $L/\bigotimes_{i=1}^n L^{sep}_i=\bigotimes_{i=1}^nL_i /\bigotimes_{i=1}^n L^{sep}_i$ is pure inseparable. Therefore, $L^{sep}=\bigotimes_{i=1}^n L^{sep}_i$.

2. Notice that $E(L/K)=E\Big(\bigotimes_{i=1}^nL_i /K\Big)=
\bigotimes_{i=1}^n E(L_i /K)$ and $L^{sep}=\bigotimes_{i=1}^n L^{sep}_i$ (statement 1). Now,
\begin{eqnarray*}
\CD (L/K) &\stackrel{ {\rm Thm.}\, \ref{AC24Mar25}.(1)}{=}& E (L/L_{dif}) \stackrel{ {\rm Thm.}\, \ref{23May25}}{=} E(L/L^{sep})=C_{E(L/K)}(L^{sep})=C_{\otimes_{i=1}^n E(L_i/K)}\bigg(\bigotimes_{i=1}^n L_i^{sep}\bigg)\\
& =& \bigotimes_{i=1}^n   C_{E(L_i/K)}( L_i^{sep})
=
\bigotimes_{i=1}^nE(L/L_i^{sep})
\stackrel{ {\rm Thm.}\,\ref{23May25}}{=}\bigotimes_{i=1}^n\CD \Big(L_i/(L_i)_{dif}\Big)\\
&\stackrel{ {\rm Thm.}\, \ref{AC24Mar25}.(1)}{=}&
\bigotimes_{i=1}^n\CD (L_i/K).
\end{eqnarray*}
\end{proof}

 %%%%%%%%%%%%%%%%%% SECTION 4   %%%%%%%%%%%%%%%%%%%%%

\section{Analogue of the Galois correspondence for purely inseparable field extensions} \label{GAL-PURE-INS} %\marginpar{GAL-PURE-INS}

In this section,  $K$ is a field of characteristic $p$ and $L/K$ is a purely inseparable   finite field extension. In this section, the concepts of balanced and dominant Lie algebras are introduced. These are the key concepts in the analogue of the Galois correspondence for purely inseparable field extensions.

Theorem \ref{C24Mar25} provides an explicit description of the algebra of differential operators  $\CD (L/K)$. Theorem \ref{D24Mar25} is an analogue of the Galois Theory for purely inseparable field extensions that establishes a bijection between subfields and the algebras  of   differential operators (Theorem \ref{D24Mar25}.(1,2)) and a bijection between subfields and the dominant Lie algebras (Theorem \ref{D24Mar25}.(3,4)).\\
 
{\bf Description of the algebra of differential operators on a purely inseparable finite field extension.}   Theorem \ref{C24Mar25} provides an explicit description of the algebra of differential operators 
 $\CD (L/K)$ for a purely inseparable finite field extension $L/K$ in characteristic $p$. 

\begin{theorem}\label{C24Mar25}%\marginpar{C24Mar25}
Let $K$ be a field of  characteristic $p$, 
 $L/K$ be a  purely inseparable field extension  of degree $n:=[L:K]<\infty$. Then:
 \begin{enumerate}
 
 \item $\CD (L/K)=E(L/K)\simeq M_n(K)\in \fCK$ and $\Deg (\CD (L/K))=n$. In particular, all purely inseparable finite field extensions are B-extensions.
 
 \item $L\in \Max_s (\CD (L/K))$.

\item $Z(\CD (L/K))=K$.
 
 \item $L$ is the only (up to isomorphism) simple 
 $\CD (L/K)$-module.
 
 \item $\End_{\CD (L/K)}(L)=K$.

\end{enumerate}
\end{theorem}

\begin{proof} 1. Clearly, $\CD (L/K)\subseteq E(L/K)\simeq M_n(K)\in \fCK$ and $\Deg (M_n(K))=n$. So, it suffices to show that $\CD (L/K)\supseteq E(L/K)$. Since $L/K$ is a purely inseparable field extension of finite degree, there is a natural number $i\geq 1$ such that $\gf^i(L)\subseteq K$. For each element $\l\in L$,  
$$
 \ad_\l:E(L/K)\ra E(L/K), \;\;\phi\mapsto [\l , \phi]:=\l \phi-\phi\l
$$
 is   the inner derivation of the  algebra $E(L/K)$ determined by the element $\l$. Now,  
$$
(\ad_\l)^{p^i}=\ad_{\l^{p^i}}=0\;\; {\rm  for \;all}\;\;\l \in L.
$$
 Then, by the definition of the alebra of differential operators,  $E(L/K)\subseteq \CD (L/K)$, since $[L:K]<\infty$. Therefore, $\CD(L/K)=E(L/K)$.

2.  By statement 1, $[L:K]=n= \Deg (M_n(K))=\Deg(\CD (L/K))$, and so  $L\in \Max_s (\CD (L/K))$. 

3. Statement 3 follows from statement 1.

4. The field $L$ is a simple $\CD (L/K)$-module and  $\CD (L/K)\in \fCK$. Therefore, the field $L$ is  the only (up to isomorphism) simple 
 $\CD (L/K)$-module.

5. Statement 5 follows from the fact that 
$\CD (L/K)\simeq M_n(K)$ (statement 1).
\end{proof}

 Proposition \ref{aD24Mar25} shows that all  subalgebras of $\CD (L/K)$ that contain the field $L$ are simple algebras where $L/K$ is a  purely inseparable field extension  of finite degree.

\begin{proposition}\label{aD24Mar25}%\marginpar{aD24Mar25}
 Let $K$ be a field of  characteristic $p$, 
 $L/K$ be a  purely inseparable field extension  of finite degree. Then 
 $$
 \CA ( \CD (L/K), L)=\CA ( \CD (L/K), L, sim),
 $$
  i.e. every subalgebra of $\CD (L/K)$ that contains the field $L$ is a simple algebra.
\end{proposition}

\begin{proof} By Theorem \ref{C24Mar25}.(1), $\CD (L/K)=E(L/K)$ and the result follows from Theorem \ref{B24Mar25}.(2).

Below is a direct proof of this result.
Let $A\in \CA ( \CD (L/K), L)$.  
We have to show that the algebra $A$ is a simple  algebra. Let $I$ be a nonzero ideal of the algebra $A$ and $a\in I\backslash \{ 0\}$. If $a\in L^\times$ then $I=A$. Suppose that $a\in A\backslash L$.  It follows from the definition of the algebra of differential operators $\CD (L/K)$ that there are elements $a_1, \ldots , a_s\in L$ such that
 $$ l=\ad_{a_1}\cdots \ad_{a_s}(a)\in L^\times .$$
 Clearly, $l\in I$, and so $I=A$. Therefore, the algebra $\CD (L/K)$ is a simple algebra.  
\end{proof}

{\bf Class $\CL (L/K)$ of Lie subalgebras of the Lie algebra $\CD(L/K)_+$ and dominant Lie algebras.} Each associative algebra $A$ is a Lie algebra $(A, [\cdot, \cdot])$ where $[a,b]:=ab-ba$ and every one-sided  ideal of the algebra $A$ is a Lie subalgebra of $A$. In particular, the left ideal $\CD (L/K)_+$ of the algebra $\CD (L/K)$ is a Lie subalgebra of the Lie algebra $(\CD (L/K), [\cdot, \cdot])$. For a subset $S\subseteq \CD (L/K)$, let $C_L(S) := \{l\in L\, |\, ls=sl$ for all $s\in S \}$   and $L^S := \{l\in L\, |\, s(l)=0$ for all $s\in S\}$. Clearly, $C_L(S) \in \CF (L/K)$. Lemma \ref{bD24Mar25}.(2) shows that $ C_L(S)\subseteq L^S$.
\begin{lemma}\label{bD24Mar25}%\marginpar{bD24Mar25}
 Let $K$ be a field of  characteristic $p$, 
 $L/K$ be a  purely inseparable finite  field extension and $S\subseteq \CD (L/K)_+$. Then:
 \begin{enumerate}

\item $C_L(S)=\{l\in L\, | \, s(l)=0$ and $ls=(sl)_+ \}$. 

\item $ C_L(S)\subseteq L^S$.

\end{enumerate}
\end{lemma}

\begin{proof} 1. (i)  $C_L(S)\subseteq C':=\{l\in L\, | \, s(l)=0$ and $ls=(sl)_+ \}$: For all elements $l\in C_L(S)$ and  $s\in S$, $s(l)=sl-(sl)_+=ls-(sl)_+\in L\cap \CD(L/K)_+=0$, and the statement (i) follows.

(ii)  $C_L(S)\supseteq C'$: For all elements $l\in C'$  and  $s\in S$, $sl=s(l)+(sl)_+=0+ls=ls$, and the statement (ii) follows.

2. Statement 2 follows from statement 1. 
\end{proof}

\begin{definition} 

A Lie subalgebra $\CG$ of the Lie algebra $\CD (L/K)_+$ is called {\bf balanced} if  
$C_L(\CG) = L^\CG$. The set of balanced Lie subalgebras  of $\CD (L/K)_+$ is denoted by $\CL (L/K)$.  
A Lie algebra $\CG \in \CL (L/K)$ is called a {\bf dominant Lie algebra} if $\CG'\subseteq \CG$ 
for all Lie algebras $\CG' \in \CL (L/K)$  such that $L^{\CG'}=L^\CG$. Let $\mL (L/K)$ be the class of dominant Lie algebras.
\end{definition}
 
Notice that 
%\marginpar{CLLK}
\begin{equation}\label{CLLK}
\CL (L/K)=\coprod_{M\in \CF (L/K)}\CL (L/K,M)\;\; {\rm where}\;\; \CL (L/K,M):=\{\CG \in \CL (L/K)\, | \, M=L^\CG=C_L(\CG) \}.
\end{equation}
Lemma \ref{cD24Mar25}.(2,3) gives an explicit description of dominant Lie algebras.
\begin{lemma}\label{cD24Mar25}%\marginpar{cD24Mar25}
 Let $K$ be a field of  characteristic $p$, 
 $L/K$ be a  purely inseparable finite  field extension. Then:
 \begin{enumerate}

\item If $\CG, \CG'\in \CL (L/K,M)$ then the Lie subalgebra of $\CD (L/K)_+$ which is generated by the Lie subalgebras $\CG$ and $\CG'$ belongs to $\CL (L/K,M)$.

\item For all fields $M\in \CF (L/K)$, $\CG (M):=\bigcup_{\CG \in \CL (L/K,M)}\CG$ is the largest element (w.r.t. $\subseteq$) in $\CL (L/K,M)$. In particular, $\CG (M)\in \mL (L/K)$ is a dominant Lie algebra. 

\item For all fields $M\in \CF (L/K)$, $\CG (M)=\CD (L/M)_+$.

\end{enumerate}
\end{lemma}

\begin{proof} 1.  Let $\CG''$ be the Lie subalgebra of $\CD (L/K)_+$ which is generated by the Lie subalgebras $\CG$ and $\CG'$. 
 Since $\CG\subseteq \CG''$, we have that 
$L^\CG\supseteq L^{\CG''}$ and $C_L(\CG)\supseteq C_L(\CG'')$. The reverse inclusions follow from the fact that the Lie algebra $\CG''$ is generated by successive commutators of elements of $\CG\cup \CG'$ and as a result each element of the Lie algebra $\CG''$ is a linear  combination of noncommutative monomials in the `variables' $\CG\cup \CG'$.

2. Statement 2 follows from statement 1.

3. (i) {\em For all fields $M\in \CF (L/K)$,} $\CD (L/K)_+\in \CL (L/K,M)$:  By Theorem \ref{23May25},  $L_{dif}=L^{sep}=L^{\CD (L/K)_+}$ for all finite field extensions $L/K$. Hence, for all fields $M/K\in \CF (L/K)$, 
$$
C_L (\CD (L/K)_+)=C_L(L\oplus \CD(L/M)_+)=C_L( \CD(L/M))=(L/ M)_{dif}=(L/ M)^{\CD (L/M)_+}=(L/M)^{sep}=M,
$$
i.e. $\CD (L/K)_+\in \CL (L/K,M)$.

(ii)  {\em For all fields $M\in \CF (L/K)$,} $\CG (M)=\CD (L/K)_+$: By statement 2 and the statement (i), $\CG (M)\supseteq\CD (L/M)_+$. It remains to show that the reverse inclusion holds. By Theorem \ref{C24Mar25}.(1), $E(L/M)=\CD (L/M)=L\oplus \CD (L/M)_+$. 
Since $\CG (M)\supseteq\CD (L/M)_+$, an  associative subalgebra of $E(L/K)=\CD (L/K)=L\oplus \CD (L/K)_+$, say $A$, that is generated 
by the field $L$ and $\CG (M)$ contains the field $L$, and so $A=E(L/N)=\CD (L/N)$ for some field $N\in \CF (L/K)$, by Theorem \ref{C24Mar25}.(1). Since 
$$
N\stackrel{{\rm Th.}\, \ref{B24Mar25}.(3)}{=} C_{E(L/K)}(E(L/N))=C_{E(L/K)}(A)=L\cap C_{E(L/K)}(A)=
C_L(A)=C_L(\CG (M))=M.$$
Therefore, $C_{E(L/K)}(A)=M\stackrel{{\rm Th.}\, \ref{B24Mar25}.(3)}{=}C_{E(L/K)}(E(L/M))=C_{E(L/K)}(\CD(L/M))$. By Theorem \ref{B24Mar25}.(3), 
 $A=\CD (L/M)$, and so $\CG (M)\subseteq \CD (L/M)_+$ (since $\CG (M)\subseteq \CD (L/K)_+$). 
\end{proof}

{\bf Analogue of the Galois Theory for purely inseparable field extensions.}  Theorem \ref{D24Mar25} is an analogue of the Galois Theory for purely inseparable field extensions that establishes a bijection between subfields and the algebras  of   differential operators. 

%Applying  Theorem \ref{D24Mar25}  
%to purely inseparable field extensions with differential basis yields the results of N. Jacobson *** ref + details *** and in general situation we obtain the results of *** *** that establishes a bijection between subfields and the  higher derivations of field extensions.

\begin{theorem}\label{D24Mar25}%\marginpar{D24Mar25}
%{\sc (Analogue of the  Galois correspondence for purely inseparable  field extensions)} 

Let $K$ be a field of  characteristic $p$, 
 $L/K$ be a  purely inseparable finite  field extension. Then:
 \begin{enumerate}

\item $\CA ( \CD (L/K), L)\stackrel{{\rm Pr.}\, \ref{aD24Mar25}}{=}\CA ( \CD (L/K), L, sim)=\{ \CD(L/M)\, | \, M\in \CF (L/K) \}$.

\item {\sc (Analogue of the Galois correspondence  for subfields of a purely inseparable finite  field extension)}

 The map 
$$\CF (L/K)\ra \CA ( \CD (L/K), L)\stackrel{{\rm Pr.}\, \ref{aD24Mar25}}{=}\CA (\CD (L/K), L, sim), \;\; M\mapsto  C_{\CD (L/K)}(M)=\CD (L/M)$$  is a bijection with inverse $A\mapsto C_{\CD (L/K)}(A)$.

\item $\mL (L/K)=\{ \CG (M)=\CD(L/M)_+\, | \, M\in \CF (L/K)\}$.

\item {\sc (Analogue of the Galois correspondence  for subfields of a purely inseparable finite  field extension)}

The map
$$
\CF (L/K)\ra  \mL (L/K), \;\; M\mapsto \CG (M)=\CD(L/M)_+
$$
is a bijection with inverse $\CG\mapsto L^\CG=C_L(\CG)$.

\end{enumerate}
\end{theorem}

\begin{proof} 1 and 2. Statements 1 and 2  follow from Theorem  \ref{B24Mar25} and Theorem \ref{C24Mar25}.

3. Statement 3 follows from Lemma \ref{cD24Mar25}.(3).

4. In view of the equality $\CD (L/M)=L\oplus \CD (L/M)_+$ where $M\in \CF (L/K)$, statement 4 follows from statement 2.
\end{proof}

Recall that  the classical Galois correspondence for finite Galois field extensions is the bijection $M\mapsto G(L/M)$. So, in Theorem \ref{D24Mar25}   the algebras differential operators $\CD (L/K)$ on purely inseparable finite field extensions $L/K$ play the role of the Galois group in the classical Galois Theory.

Suppose $L/K$ is a purely inseparable finite field extension. %of finite degree. 
Then 
the map 
%\marginpar{CFL=mDL}
\begin{equation}\label{CFL=mDL}
\CF (L/K)\ra \mD (L/K):=\{ \CD (L/M)\, | \, M/K\in \CF (L/K)\}, \;\; M\mapsto \CD (L/M)
\end{equation}
is a surjection. 
Corollary \ref{B23Mar25}.(5) shows that the map (\ref{CFL=mDL}) is a bijection.

\begin{corollary}\label{B23Mar25}%\marginpar{B23Mar25}
Suppose that $K$ is a field of characteristic $p>0$,  $L/K$ is a purely inseparable finite  field extension  and $M/K\in \CF (L/K)$. Then:
 
\begin{enumerate}

\item $\CD (L/M)=E(L/M)\simeq M_{[L:M]}(M)\in \fC (M)$ and  $\dim_K(\CD (L/M))=[L:M]^2[M:K]=[L:K][L:M]$.

\item $\CD (L/M)\subseteq \CD (L/K)$ and $\CD (L/M)=C_{\CD (L/K)}(M)$.

\item $Z(\CD (L/M))=C_{\CD (L/K)}(\CD (L/M))=M$.

\item $\CD(L/C_{\CD (L/K)}(\CD (L/M)))=\CD (L/M)$.

\item The map (\ref{CFL=mDL})
is a bijection with inverse $\CD'\mapsto C_{\CD (L/K)}(\CD')$.

\end{enumerate}
\end{corollary}

\begin{proof} 1. Statement 1 is a particular case of Theorem \ref{C24Mar25}.(1) since the finite field extension $L/M$ is purely inseparable.

2. By statement 1,  $\CD (L/M)=E(L/M)\subseteq E(L/K)=\CD (L/K)$ and  
$$\CD (L/M)=E(L/M)=C_{E(L/K)}(M)=C_{\CD (L/K)}(M).$$

3. By Theorem \ref{C24Mar25}.(1), $\CD (L/K)\in \fCK$. Clearly, the field $M$ is a simple finite dimensional $K$-algebra that is contained in $\CD (L/K)$. Then, by  statement 3 and the Double Centralizer Theorem (Theorem \ref{DCThm}.(1)),
$$C_{\CD (L/K)}(\CD (L/M))= C_{\CD (L/K)}(C_{\CD (L/K)}(M)) =M.$$
4. By statement 3, $C_{\CD (L/K)}(\CD (L/M))=M$, and so 
$$ 
\CD(L/C_{\CD (L/K)}(\CD (L/M)))=\CD (L/M).
$$

5. Statement 5 follows from statements 3 and 4.
\end{proof}

 %%%%%%%%%%%%%%%%%% SECTION 5 %%%%%%%%%%%%%%%%%%%%%
 
 \section{The algebras $\CD (L/K)\rtimes G(L/K)$, $L\rtimes G(L/K)$ and  $\CD (L/K)$ } \label{3-ALGS} %\marginpar{3-ALGS}

 Proposition  \ref{A14May25} provides sufficient conditions for two subfields  being linearly disjoint (i.e. their compositum  is isomorphic to their tensor product). Proposition  \ref{A14May25} is a powerful result since it implies the known fact that {\em a finite field extension $L/K$ is normal iff} $L=L^{pi}\t L^{gal}$ (Corollary \ref{a25May25}).
  Theorem \ref{B11May25} is a criterion for $\CD (L/K)=E(L/K)$. Theorem \ref{BC24Mar25} describes properties of the subalgebra $L\rtimes G(L/K)$ of $E(L/K)$.  Proposition \ref{A8May25} is a criterion for a finite field extension being a Galois field extensions which is given via the algebra $L\rtimes G(L/K)$. Corollary \ref{a24Mar25} provides explicit descriptions of the sets 
  $$\CA (L\rtimes G(L/K), L)\;\; {\rm  and }\;\;\CA (L\rtimes G(L/K), L, G(L/K)).
  $$
   Theorem \ref{AB17Apr25} shows that the subalgebra of $E(L/K)$ which is generated by  the automorphism group $G(L/K)$ and the algebra of differential operators $\CD (L/K)$ is the skew group algebra $\CD (L/K)\rtimes G(L/K)$. Corollary \ref{c17Apr25} is a criterion for the finite field extension $L/K$ being a B-extension.

 Proposition \ref{AC17Apr25} shows that,
 the field extension $L_{dif}/L^{G(L/K)}_{dif}$ is a Galois field extension with Galois group isomorphic to $G(L/K)=G(L/L^{G(L/K)})$ and 
 $$
 L=L^{G(L/K)}\t_{L^{G(L/K)}_{dif}}L_{dif}.
 $$
  Proposition \ref{AC17Apr25}.(6) presents natural classes of field extensions with identity automorphism groups. Namely, $G(L/L_{dif})=\{ e\}$ and $G\Big(L^{G(L/K)}/L^{G(L/K)}_{dif} \Big)=\{ e\}$.
  Proposition \ref{VB-C17Apr25} is a strengthening of  Proposition \ref{AC17Apr25} where the finite field extension $L/K$ is normal. 
 
 Theorem \ref{20Apr25} is an explicit  description of the algebra $\CD (L/K)$ of    differential operators on a finite field extension $L/K$ in prime characteristic. Corollary \ref{a20Apr25} is an explicit  description of the algebra $\CD (L/K)$ of    differential operators on a {\em normal} finite field extension $L/K$ in prime characteristic.

  Theorem  \ref{bA20Apr25} shows that the algebra $ \CD (L/K)\rtimes G(L/K)$ is a tensor product of  the subalgebras 
  $$\CD \Big(L^{G(L/K)}/L^{G(L/K)}_{dif}\Big)\;\; {\rm  and} \;\; L_{dif}\rtimes  G\Big(L_{dif}/L^{G(L/K)}_{dif}\Big)
  $$
   over the field $L^{G(L/K)}_{dif}$.
   If, in addition, the field extension $L/K$ is normal, Corollary \ref{cA20Apr25} shows that the algebra $ \CD (L/K)\rtimes G(L/K)$ is a tensor product of  the subalgebras $\CD (L^{pi}/K)$ and $ L^{sep}\rtimes  G (L^{sep}/K)$.
 
Suppose that $L_i/K$ $(i=1, \ldots , n)$ are normal finite field extensions such that their tensor product $L=\bigotimes_{i=1}^nL_i$ is a field.  Proposition \ref{A7Jun25} shows that 
$$\CD (L/K)\rtimes G(L/K)\simeq \bigotimes_{i=1}^n\CD (L_i/K)\rtimes G(L_i/K)
,\;\;  \CD(L)=\bigotimes_{i=1}^n\CD (L_i/K)$$  
  and $G(L/K)\simeq \prod_{i=1}^n G(L_i/K)$. \\

 {\bf Sufficient conditions  for $MN=M\t N$.} 
 Proposition  \ref{A14May25} provides sufficient conditions for the compositum   of two subfields  being isomorphic to their tensor product. In this case, the two subfields are called {\bf linearly disjoint} over the ground field. The sufficient conditions are given via Galois groups.  Proposition \ref{A14May25} is used in the proof of Theorem \ref{VVB-30Apr25}.
 
\begin{proposition}\label{A14May25}%\marginpar{A14May25}
Let $M$ and $N$ be subfields of $\bK$ that contain the field $K$ and $MN$ be their compositum. 
\begin{enumerate}

\item Suppose that $M/K$ is  finite field extension, the field extension $N/K$ is  a Galois finite  field  extension such  that the restriction map $\res: G(MN/M)\ra G(N/K)$, $\s\mapsto \s|_{N}: N\ra N$, $n\mapsto \s(n)$ is a well-defined epimorphism. Then $MN=M\t N$. 

\item Suppose that $N/K$ is a  Galois finite field extension such that every automorphism $\s\in G(N/K)$ can be lifted to an automorphism $\s'\in G(MN/M)$. Then  $MN=M\t N$. 

\item Suppose that $M/K$ is a purely inseparable field extension and $N/K$ is a separable field extension. Then $MN=M\t N$. 

\end{enumerate}
\end{proposition}

\begin{proof} 1. We have finite  field extensions 
$K\subseteq M\subseteq (MN)^{G(MN/M)}\subseteq MN$. 
It follows from the inequalities
$$|G(MN/M)|\leq [MN:M]\leq [N:K]\;\; {\rm and}\;\; 
|G(MN/M)|\geq |G(N/K)|=[N:K]
$$
that $[MN:M]=[N:K]$. In view of the epimorphism
$M\t N\ra MN$, $m\t n\mapsto mn$, we have that  
$$ 
[M:K][N:K]=\dim_K(M\t N)\geq [MN:K]=[MN:M][M:K]=[N:K][M:K].
$$
Therefore, $[MN:K]=[N:K][M:K]=\dim_K(M\t N)$, and so
$MN=M\t N$.

2. We have to show that for every finite field extension $M'/K$ such that $M'\subseteq M$, 
$$
M'N=M'\t N.
$$
 By the assumption, every automorphism $\s\in G(N/K)$ can be lifted to an automorphism $\s'\in G(MN/M)$. Hence, $\s' (M'N)=M'N$, and so the restriction map 
 $$
 \res: G(M'N/M')\ra G(N/K), \;\; \s\mapsto \s|_{N}: N\ra N, \;\;n\mapsto \s(n)
 $$ 
 is a well-defined epimorphism. By statement 1,  $M'N=M'\t N$.

3.  We have to show that for  for every finite field extension $N'/K$ such that $N'\subseteq N$, 
$$
MN'=M\t N'.
$$
Since the field extension $N/K$ is separable, so is the field extension $N'/K$. Let $N''$ be the normal closure of the field $N'$ in $\bK$. Then the finite  field extension $N''/K$ is separable and normal, hence Galois. So, it suffices to show that 
$$
MN''=M\t N'',
$$
 Each automorphism $\s\in G(N''/K)$ can be lifted to an automorphism $\s'\in G(\bK /K)$. By the assumption,  $M/K$ is a purely inseparable field extension. Therefore, $\s (m)=m$ for all elements $m\in M$, and so  $\s' (MN'')=MN''$. 
 So, every automorphism $\s\in G(N''/K)$ can be lifted to an automorphism $\s''\in G(MN''/M)$. Then  $MN''=M\t N''$, by statement 2. 
\end{proof}

 Corollary  \ref{aA14May25} is a criterion for two subfields being linearly disjoint provided one of the subfield is Galois.

\begin{corollary}\label{aA14May25}%\marginpar{aA14May25}
Let $L/K$ be a finite field extension, $M, N\in \CF (L/K)$, $MN$ be their compositum   and the field extension $N/K$ be Galois. Then the following statements are equivalent:
\begin{enumerate}

\item $MN=M\t N$, i.e. the subfields $M$ and $L$ are linearly disjoint. 

\item $[MN:K]=[M:K][N:K]$.

\item The restriction map $\res: G(MN/M)\ra G(N/K)$, $\s\mapsto \s|_{N}: N\ra N$, $n\mapsto \s(n)$ is a well-defined  epimorphism.

\item The restriction map $\res: G(MN/M)\ra G(N/K)$, $\s\mapsto \s|_{N}: N\ra N$, $n\mapsto \s(n)$ is a well-defined   isomorphism.

\item $M\cap N=K$.

\end{enumerate}
\end{corollary}

\begin{proof} $(1\Leftrightarrow 2)$ The $K$-algebra homomorphism $\mu : M\t N\ra MN$, $m\t n\ra mn$ is an epimorphism. So, the map $\mu$ is an isomorphism iff $[MN:K]=[M\t N:K]=[M:K][N:K]$.

$(3\Leftrightarrow 4)$ Since $|G(MN/M)|\leq |G(N/K)|$, the restriction map $\res$ is an epimorphism iff it is an isomorphism.

$(3\Rightarrow 1)$ Proposition \ref{A14May25}.(1).

$(1\Rightarrow 3)$ The implication is obvious.

$(1\Rightarrow 5)$ Suppose that $M\cap N\neq K$. Then we can choose an element, say $m\in M\cap N\backslash K$, and so $0\neq m\t 1-1\t m\in \ker (\mu)$. So, the map $\mu$ is not an isomorphism, a contradiction.

$(5\Rightarrow 1)$ Suppose that $M\cap N= K$. Since the field extension $N/K$ is a Galois finite field  extension, 
$$N=K(\th )\simeq K[x]/(f)
$$ 
for some element $\th \in N$ where $f(x)\in K[x]$ is the minimal polynomial of the element $\th$ over the field $K$. Clearly, the finite field extension $MN/M=M(\th )$ is Galois (since $N/K$ is so). Suppose that   the subfields $M$ and $N$ are not linearly disjoint over $K$. Then $[MN:K]<[M\t N:K]=[M:K][N:K]$, and so the minimal polynomial $g(x)\in M[x]$ of the element $\th$ over the field $M$ is a proper divisor of the polynomial $f$ since  
$$
\deg (g)=[MN:M]< [M\t N:M]=[N:K]=\deg (f).
$$
 Since the field extension $MN/M$ is Galois,
$$g(x)=\prod_{g \in G(MN/M)}(x-g(\th))=\sum_{i=0}^l(-1)^is_i(\ldots, g(\th), \ldots) x^{l-i}$$
where $l=|G(MN/M)|$ and $s_i(x_1, \ldots , x_l)$ is the $i$'th symmetric polynomial in $l$ variables  
($\deg (s_i)=i$). Clearly, 
$$(-1)^is_i(\ldots, g(\th), \ldots)\in M\cap N\;\; {\rm for\; all}\;\; i=0, \ldots , l.$$
Since $g$ is proper divisor of the minimal polynomial $f$ of $\th$ over $K$, there is an index $i$ such that $ (-1)^is_i(\ldots, g(\th), \ldots)\in M\backslash K$ (otherwise $g=f$, a contradiction). Therefore, $M\cap N\neq K$, a contradiction.  
\end{proof}

\begin{corollary}\label{bA14May25}%\marginpar{bA14May25}

Let $L/K$ be a finite field extension and   $M, N\in \CF (L/K)$. 
\begin{enumerate}

\item If  $N/M\cap N$ is  a Galois  field extension then $M/M\cap N$ and $N/M\cap N$ are linearly disjoint.

\item If $N/K$ is  a separable  field extension,    $\widetilde{N}/K$ is  the Galois closure (in $\bK$) of the field extension $N/K$ and $M\cap \widetilde{N}= M\cap N$ then  $MN\simeq M\t_{M\cap N}N$, i.e. the fields $M$ and $N$ are linearly disjoint over $M\cap N$.

\end{enumerate}

%*****
%Let $L/K$ be a finite field extensions and  $M, N\in \CF (L/K)$ be Galois fields. Then $M$ and $N$ are linearly disjoint iff $M\cap N=K$.
%***
\end{corollary} 

\begin{proof} 1. Statement 1  follows from  Corollary \ref{aA14May25}.

2. By statement 1, $M\widetilde{N}\simeq M\t_{M\cap \widetilde{N}}\widetilde{N}=M\t_{M\cap N}\widetilde{N}$. Since $MN\subseteq M\widetilde{N}$ and $M\t_{M\cap N}N\subseteq M\t_{M\cap N}\widetilde{N}$, we have that  $MN\simeq M\t_{M\cap N}N$. 
\end{proof}

{\bf The equality $L=L^{pi}\t L^{gal}$ for a normal finite field extension $L/K$ in characteristic $p$.} 
 Let $L/K$ be a {\em normal} finite field extension  and $\char (K)=p$.  Then $L^{pi}$ and  $L^{sep}=L^{gal}$ are purely inseparable and separable finite field extensions, respectively. Therefore, $L^{pi}\cap L^{gal}=K$   and we have the following  commutative diagram of finite field extensions:
 
$$\begin{array}[c]{ccc}
       &    L      &  \\
&\stackrel{gal}{\nearrow}\;\;\;\; \;\;\;\;\;\;\;\;{\nwarrow}^{pi}&\\
L^{pi} &   & L^{gal}\\
&{\nwarrow}^{pi}\;\;\;\; \;\;\;\; \;\;\;\;\;\;\;\;\stackrel{gal}{\nearrow}  &  \\
& L^{pi}\cap L^{gal}=K &
\end{array}$$
where the field extensions $L/L^{pi}$ and $L^{gal}/K$ are Galois and the field extensions $L^{pi}/K$ and $L/L^{gal}$ are purely inseparable. 
 
   The {\em separable} extension $L^{gal}/K$ is of finite degree.   By the Primitive Element Theorem, $L^{gal}=K(a)$ for some element $a\in L^{gal}$. Then $[L^{gal}:K]=\deg (f)$ is the degree of  the (monic) minimal polynomial  $f\in K[x]$ of the element $a$ over the field $K$. Let $g\in L^{pi}[x]$ be the (monic) minimal polynomial  of the element $a$ over the field  $L^{pi}$. Since $\s (l)=l$ for all elements $\s \in G(L/K)$ and $l\in L^{pi}$, we have that 
 $$
 0=\s (g(a))=g(\s (a))\;\; {\rm for\; all}\;\; \s \in G(L/K).
 $$
Therefore, $g=f$, and so 
$$
[L^{pi}(a):L^{pi}]=\deg (g)=\deg (f)=[L^{gal}:K]=|G(L^{gal}/K)|.
$$
By the normality of the field extension $L/K$ and the diagram above, the map 
$$
\res: G(L/K)\ra G(L^{gal}/K), \;\;\s\mapsto \s|_{L^{gal}}: L^{gal}\ra L^{gal}, \;\;l\mapsto \s (l)
$$
 is a group isomorphism. Therefore, $[L:L^{pi}]=[L^{gal}:K]$,  $[L:L^{gal}]=[L^{pi}:K]$ and 
$$
[L:K]=[L:L^{pi}][L^{pi}:K]=[L^{gal}:K][L^{pi}:K]=
[L^{pi}(a):L^{pi}][L^{pi}:K]=[L^{pi}(a):K].
$$
Therefore, $L=L^{pi}(a)= L^{pi}\t K(a)= L^{pi}\t L^{gal}$.

So, for the normal finite field extension $L/K$ we have the following  commutative diagram of finite field extensions:
 
$$\begin{array}[c]{ccc}
       &    L=L^{pi}\t L^{gal}      &  \\
&\stackrel{gal}{\nearrow}\;\;\;\; \;\;\;\;\;\;\;\;{\nwarrow}^{pi}&\\
L^{pi} &   & L^{gal}\\
&{\nwarrow}^{pi}\;\;\;\; \;\;\;\; \;\;\;\;\;\;\;\;\stackrel{gal}{\nearrow}  &  \\
& L^{pi}\cap L^{gal}=K &
\end{array}$$
where the field extensions $L/L^{pi}$ and $L^{gal}/K$ are Galois and the field extensions $L^{pi}/K$ and $L/L^{gal}$ are purely inseparable.

 As a corollary,  we obtain a short proof of an important  known  result that $L=L^{pi}\t L^{gal}$.
 
 \begin{corollary}\label{a25May25}%\marginpar{a25May25}
 
 Let $L/K$ be a  finite field extension  and $\char (K)=p$. Then:
 
 \begin{enumerate}
 
 \item The  field extension $L/K$ is normal iff $L=L^{pi}\t L^{gal}$. 
 
 \item If, in addition,  $L/K$ is a  normal  field extension then $L^{sep}=L^{gal}$ and the map $\res: G(L/K)\ra G(L^{gal}/K)$, $\s\mapsto \s|_{L^{gal}}: L^{gal}\ra L^{gal}$, $l\mapsto \s(l)$ is a group isomorphism.
 
 \end{enumerate}
 \end{corollary}

\begin{proof} 1.  $(\Rightarrow )$ The implication was proven above.

%By the normality of the field extension $L/K$ and the diagram above, the map $\res: G(L/K)\ra G(L^{gal}/K)$, $\s\mapsto \s|_{L^{gal}}: L^{gal}\ra L^{gal}$, $l\mapsto \s(l)$ is a group isomorphism. Therefore, $[L:L^{pi}]=[L^{gal}:K]$,  $[L:L^{gal}]=[L^{pi}:K]$ and 
%$$[L:K]=[L:L^{pi}][L^{pi}:K]=[L^{gal}:K][L^{pi}:K]=\dim_K(L^{pi}\t L^{gal}).
%$$
%Hence, the algebra epimorphism $L^{pi}\t L^{gal}\ra L=L^{pi}L^{gal}$ is an isomorphism, and so   $L=L^{pi}\t L^{gal}$, by Proposition \ref{A14May25}.(1). Clearly, $L^{sep}=L^{gal}$. 

$(\Leftarrow )$ Suppose that $L=L^{pi}\t L^{gal}$. Since the field extensions $L^{pi}/K$ and $L^{gal}/K$ are normal, so is the field extension $L/K$.

2.  Statement 2 was proven above.
\end{proof}

%By the assumption, the finite field extension $L/K$ is normal. By statement 1,  $L=L^{pi}\t L^{gal}$. 
%Then $L^{sep}=L^{gal}$ (since the field $L^{sep}$ is $G(L/K)$-stable and $L^{gal}\subseteq L^{sep}$). 

 Lemma \ref{a11May25} is used in the proof of Theorem \ref{B11May25}.
 
\begin{lemma}\label{a11May25}%\marginpar{a11May25}
Let $L/K$ be a finite field extension  of prime characteristic, $n=[L:K]$ and $l\in L$. Then:
\begin{enumerate}

\item The inner derivation $\ad_l$ of the endomorphism algebra $E(L/K)$ is a locally nilpotent (hence, nilpotent) map iff $\ad_l^n=0$.

\item $\dim_L(E(L/K))=n$.

\end{enumerate}
\end{lemma}

\begin{proof} 2. $\dim_L(E(L/K))=\frac{\dim_K(E(L/K))}{[L:K]}=\frac{n^2}{n}=n$.

1.  For all elements $ l,l'\in L$, the elements  $\ad_l$ and $l'\cdot$ of $E(L/K))$  commute where $l': L\ra L$, $\l \mapsto \l l'$. Therefore, the $K$-vector  spaces  $\ker (\ad_l^i)$, where $i\geq 1$, are left/right $L$-modules. Since $\dim_L(E(L/K))=n<\infty$ (statement 2), the $K$-linear map $\ad_l$ is a locally nilpotent map iff $\ad_l^n=0$.
\end{proof}

In general, the containment $\CD (L/K)\subseteq E(L/K)$ is a strict containment. 
Theorem \ref{B11May25} is a criterion for $\CD (L/K)=E(L/K)$. Theorem \ref{B11May25} is used in the proof of Theorem \ref{VVB-30Apr25}.

\begin{theorem}\label{B11May25}%\marginpar{B11May25}
Let $L/K$ be a finite field extension  of prime characteristic $p$ and  $n=[L:K]$. Then $\CD (L/K)=E(L/K)$ iff $L=L^{pi}$.
\end{theorem}

\begin{proof} $\CD (L/K)=E(L/K)$ iff the maps $\ad_l\in E(L/K)$, where $l\in L$, are locally nilpotent maps iff $\ad_l^n=0$ for all $l \in L$ (Lemma \ref{a11May25}.(1)) iff
$$0=\ad_l^{p^n}=\ad_{l^{p^n}}\;\; {\rm  for\; all}\;\; l \in L$$
since $p^n\geq n$ for all $n\geq 1$ (We use induction on $n\geq 1$. The initial case  $n=1$ is obvious. Suppose that $n\geq 2$ and $p^m\geq m$ for all $m=1, \ldots , n$. Then
$$ p^{n+1}=pp^n\geq pn=(p-1)n+p\geq n+1)$$
iff $l^{p^n}\in Z(E(L/K))=K$
iff $L=L^{pi}$.
\end{proof}

Let $L/K$ be a finite field extension of the field $K$ of characteristic $p$. The set
$$
\CD (L/K,L^{pi}):=\{ \d \in \CD (L/K)\, | \, \d (L^{pi})\subseteq L^{pi}\}
$$ is a subalgebtra of $\CD (L/K)$ that cointains the field $L^{pi}$. So, for all elements $l\in L^{pi}$ and $\d \in \CD (L/K, L^{pi})$, $\ad_l(\d)\in \CD (L/K, L^{pi})$. Therefore, the map  
%\marginpar{resDLpi}
\begin{equation}\label{resDLpi}
\res : \CD (L/K, L^{pi})\ra \CD(L^{pi}/K), \;\; \d\mapsto \d |_{L^{pi}}
\end{equation}
is a $K$-algebra homomorphism. Clearly, the homomorphism $\res $ is an isomorphism iff every differential operator $\d \in \CD (L^{pi}/K)$ can be {\em uniquely} extended to  a differential operator of $\CD (L/K)$. In this case, we identify the algebras $\CD (L/K, L^{pi})$ and $ \CD(L^{pi}/K)$, and the algebra $\CD (L/K, L^{pi})= \CD(L^{pi}/K)$ is a subalgebra of the algebra $\CD (L/K)\subseteq E(L/K)$.

\begin{lemma}\label{a15Apr25}%\marginpar{a15Apr25}
Let $L/K$ be a finite field extension of the field $K$ of characteristic $p$. Then:
\begin{enumerate}

\item Each derivation $\d \in \Der(L^{pi}/K)$ admits a unique extension to a derivation $\d \in \Der(L/K)$. 

\item The derivation subalgebra $\D (L^{pi}/K)$ of the algebra $\CD (L^{pi}/K)$ is contained in the image of the restriction homomorphism $\res$, 
$\D (L^{pi}/K)\subseteq \im (\res)$, see (\ref{resDLpi}). 

\end{enumerate}
\end{lemma}

\begin{proof} 1. The field extension $L/L^{pi}$ is a separable finite field extension. Hence, $L=L^{pi}(\th)$ for some element $\th\in L$, by the  Primitive Element Theorem. Let $f(x)=\sum_{i\in \N}\l_ix^i\in L^{pi}[x]$ be the minimal polynomial of the element $\th$. Then $f(\th)=0$ implies
$$ 
0=\d (0)=\d (f(\th ))=f^\d (\th)+f'(\th)\d(\th)
$$ where $f^\d (x):=\sum_{i\in \N}\d (\l_i)x^i$ and 
$f':=\frac{df}{dx}$. Therefore,   $f'(\th )\neq 0$ (since the element $\th\in L$ is a separable element over the field $L^{pi}$, and so $f'(x)\neq 0$ and $\deg (f')<\deg (f)$) and 
$\d(\th )=-\frac{f^\d (\th)}{f'(\th )}$, and statement 1 follows.

2. Statement 2 follows from statement 1. 
\end{proof}

{\bf Centralizers $C_{E(L/K)}(M/K)$ of  subfields $M$ of $L$.}
 
 \begin{lemma}\label{a17Apr25}%\marginpar{a17Apr25}
Let $L/K$ be a finite field extension of the field $K$ of characteristic $p$ and $M/K$ be a subfield of $L/K$. Then $C_{E(L/K)}(M/K)=E(L/M)$ and $\dim_K(E(L/M))=[L:M]^2[M:K]=[L:K][L:M]$.
\end{lemma}

\begin{proof} Straightforward.
\end{proof}

 \begin{lemma}\label{b17Apr25}%\marginpar{b17Apr25}
Let $L/K$ be a finite field extension, $M/K$ and $N/K$ be  subfields of $L/K$. Then:

\begin{enumerate}

\item $C_{E(L/K)}(M)\cap C_{E(L/K)}(N)=E(L/M) \cap E(L/N)=E(L/MN)=C_{E(L/K)}(MN)$. 

\item $C_{E(L/K)}(A)=M\cap N$, where $A$ is a $K$-subalgebra of $E(L/K)$  which is generated by the subalgebras  $C_{E(L/K)}(M)$ and $ C_{E(L/K)}(N)$, and 
$C_{E(L/K)}(M\cap N)=A= E(L/M\cap N)$.

\item For all finite field extensions $F/K$, 
$F\t E(L/K)=E(F\t L/F)\subseteq E(F\t L/K)$.

\end{enumerate}
\end{lemma}

\begin{proof} 1. Let $C(\cdot):=C_{E(L/K)}(\cdot)$. 
  By Lemma \ref{a17Apr25}, $C(M/K)=E(L/M)$ and $C(N/K)=E(L/N)$. Therefore,
$$ 
C(M)\cap C(N)=E(L/M) \cap E(L/N)=E(L/MN)=C(MN/K).
$$

2. $C(A)=C\Big(K\langle C(M), C(N)\rangle\Big)=CC(M)\cap CC(N)=M\cap N$, by the Double Centralizer Theorem (Theorem \ref{DCThm}.(1)). Hence, 
$$
 E(L/M\cap N)=C(M\cap N)=CC (A)=A,
 $$
  by the Double Centralizer Theorem (Theorem \ref{DCThm}.(1)).
  
  3. Clearly, $ E(F\t L/F)\subseteq E(F\t L/K)$ and 
  $F\t E(L/K)\subseteq  E(F/K)\t E(L/K) =E(F\t L/K)$.
  Notice that $F\t E(L/K)\subseteq  E(F\t L/F)$ since for all elements $f,f'\in F$, $\phi \in E(L/K)$ and $l\in L$,
  $$ (f\t \phi) (f'\t l)=ff'\t \phi (l)=f'f\t \phi (l)=
  f'(f\t \phi) (1\t l).$$
 Now, the equality   $F\t E(L/K)=  E(F\t L/F)$ follows from the equalities of their dimensions as $F$-vector spaces: 
$$ \dim_F(F\t E(L/K))= \dim_K(E(L/K))= [L:K]^2=\dim_F(E(F\t L/F)).$$
  \end{proof}

{\bf The subalgebra $L\rtimes G(L/K)$ of the endomorphism algebra $E(L/K)$.}  Let $L/K$ be a finite field extension of arbitrary characteristic. Then a subalgebra of $E(L/K)$, which is generated by the field $L$ and the automorphism group $G(L/K)$,  is equal to the sum $\sum_{g\in G(L/K)}Lg$  since $gl=g(l)g$ for all elements $l\in L$  and $g\in G(L/K)$.

 \begin{theorem}\label{BC24Mar25}%\marginpar{BC24Mar25}
Let $L/K$ be a finite field extension of  arbitrary characteristic. Then:

\begin{enumerate}

\item  The field extension $L/L^{G(L/K)}$ is a Galois field extension with Galois group $G(L/L^{G(L/K)})=G(L/K)$.

\item $\sum_{g\in G(L/K)}Lg=\bigoplus_{g\in G(L/K)}Lg=L\rtimes G(L/K)=E(L/L^{G(L/K)})\in \fC (L^{G(L/K)})$, 
$$
\Deg (L\rtimes G(L/K))=[L:L^{G(L/K)}]=|G(L/L^{G(L/K)})|=|G(L/K)|
$$ 
and 
$\dim_K(L\rtimes G(L/K))=[L:L^{G(L/K)}]^2[L^{G(L/K)}:K]=[L:K][L:L^{G(L/K)}]$.

\item $L\in \Max_s(L\rtimes G(L/K))=\Max_s(E(L/L^{G(L/K)}))$.

\item $L$ is the only (up to isomorphism) simple 
 $L\rtimes G(L/K)$-module. 

\item $\End_{L\rtimes G(L/K)}(L)=L^{G(L/K)}$. 

\item The algebra $L\rtimes G(L/K)=\bigoplus_{g\in G(L/K)}Lg$ is a direct sum of  non-isomorphic simple $L$-bimodules. 

\end{enumerate}
\end{theorem}

\begin{proof} 1. By the Galois Theory, the field extension $L/L^G$ is a Galois field extension with Galois group $G(L/L^G)=G$.

2. Let $A:= \sum_{g\in G}Lg=\sum_{g\in G(L/L^G)}Lg$
 and $G=G(L/K)=G(L/L^G)$.
 
 (i) $A=\bigoplus_{g\in G(L/K)}Lg=L\rtimes G(L/K)$: 
 The automorphisms $g\in G$ are distinct (by the definition). Hence, they are $L$-linearly independent 
 as elements of the left $L$-vector space $E=E(L/K)$, and so  $A=\bigoplus_{g\in G(L/K)}Lg=L\rtimes G(L/K)$.

(ii)  {\em The algebra $A$ is a simple algebra that contains the field $L$}: Clearly, $L\subseteq A$. So, it remains to prove simplicity of the algebra $A$. 
 For an element $a=\sum_{g\in G}l_g g\in A$, where $l_g\in L$, the set 
$$
\supp (a):=\{ g\in G\, | \, l_g\neq 0\}
$$
is called the {\em support} of the element $a$ and $|\supp (a)|$ is the number of elements in $\supp (a)$. Let $I$ be a nonzero ideal of the algebra $A$ and
$$ m:=\min \{ |\supp (a)| \, | \, 0\neq a\in I\}.$$It suffices to show that $m=1$ since then there is a nonzero element $l_gg\in I$, and so $g=l_g^{-1}l_gg\in I$ and $I=A$ since the element $g$ is  a unit of the algebra $A$ and its inverse $g^{-1}=g^{n-1}$ belongs to the ideal $I$ where $n$ is the order of the automorphism $g\in G$ (since $G$ is  a finite group, the order of  the automorphism $g$ is also finite). 

Suppose that $m\geq 2$ and $a=l_1g_1+\cdots +l_mg_m\in I$ where $l_1, \ldots , l_m\in L\backslash \{ 0\}$ and $g_1, \ldots , g_m\in G$ are distinct automorphisms. We seek a contradiction.  Then
$
ag_1^{-1}=l_1+l_2g_2g_1^{-1}+\cdots + l_mg_mg_1^{-1}\in I$. So, we may assume that $g_1=e$, i.e.
$$a=l_1+l_2g_2+\cdots +l_mg_m\in I.$$
Since $g_2\neq e$, we can find  an element, say $l\in L$, such that $g_2(l)\neq l$. Then,
$$
b:=[a,l]=l_2(g_2(l)-l)g_2+\cdots + l_m(g_m(l)-l)g_m\in I\backslash \{ 0\}\;\; {\rm and}\;\; |\supp (b)|\leq m-1<m,
$$ a contradiction.

(iii) $A=E(L/L^G)\in \fC (L^G)$: Notice that $A\subseteq E(l/L^G)$ and 
$$
\dim_{L^G}(A)=[L:L^G]^2=\dim_{L^G}(E(L/L^G)).
$$ Therefore, $A=E(L/L^G)\in \fC (L^G)$, and so 
$$
\Deg (A)=[L:L^G]=|G(L/L^G)|=|G|
$$ 
and 
$\dim_K(A)=[L:L^G]^2[L^G:K]=[L:K][L:L^G]$.

3.  Statement 3 follows from statement 1 (since $\Deg (A)=[L:L^G]$). 

4. The field $L$ is a simple $A=E(L/L^G)$-module (since $L\subseteq A$). Therefore, the field $L$ is the only (up to isomorphism)  simple $A=E(L/L^G)$-module.

5.  Statement 5 follows from the equality $A=E(L/L^G)$.

6. For all elements $l,l'\in L$ and $g\in G$, 
$$ lgl'=lg(l')g.$$ Clearly, for each element $g\in G(L/K)$, the $L$-bimodule $Lg$ is simple. Suppose that $f: Lg\ra Lg'$ is a nonzero $L$-bimodule homomorphism. Then $f(g)=\l g'$ for some nonzero element $\l \in L$, and so 
$$g(l)\l g'=f(g(l)g)=f(gl)=f(g)l=\l g' l=\l g'(l)g'.$$
Hence, $g(l)=g'(l)$ for all elements $l\in L$. This means that $g=g'$  and statement 6 follows.
\end{proof}

 \begin{corollary}\label{aBC24Mar25}%\marginpar{aBC24Mar25}
Let $L/K$ be a finite field extension of  arbitrary characteristic and $H$ be a subgroup of $G(L/K)$. Then:

\begin{enumerate}

\item The field extension $L/L^H$ is a Galois field extension with Galois group $G(L/L^H)=H$, 
$$
L\rtimes H=L\rtimes G(L/L^H)=E(L/L^H)\in \fC (L^H),\;\; L\rtimes H\subseteq L\rtimes G(L/K), 
$$
 $\Deg (L\rtimes H)=[L:L^H]=\Big|G(L/L^H)\Big|$, 
$\dim_K(L\rtimes H)=[L:L^H]^2[L^H:K]=[L:K][L:L^H]$ and $C_{E(L/K)}(L\rtimes H)=L^H$.

\item $L\in \Max_s(L\rtimes H)=\Max_s(E(L/L^H))$.

\item $Z(L\rtimes H)=L^H$.

\item $L$ is the only (up to isomorphism) simple 
 $L\rtimes H$-module. 

\item $\End_{L\rtimes H}(L)=L^H$.

\item The algebra $L\rtimes H=\bigoplus_{h\in H}Lh$ is a direct sum of  non-isomorphic simple $L$-bimodules. 

\end{enumerate}
\end{corollary}

\begin{proof} 1. Clearly, $L\subseteq B:=L\rtimes H\subseteq E(L/L^H)\subseteq E(L/L^G)=L\rtimes G$ and 
  the field extension $L/L^H$ is a Galois field extension with Galois group $G(L/L^H)=H$. The equality 
  $B=E(L/L^H)$ follows from the inclusion $B\subseteq E(L/L^H)$ and the equalities:
  $$  \dim_{L^H}(B)= \dim_{L^H}(L\rtimes H)=[L:L^H]|H|=[L:L^H]|G(L/L^H)|=[L:L^H]^2=\dim_{L^H}(E(L/L^H).$$

2--5. Statements 2--5 follow from the equality $B=E(L/L^H)$.

6. Statement 6 follows from Theorem \ref{BC24Mar25}.(6).
\end{proof}

Proposition \ref{A8May25} is a criterion for a finite field extension being a Galois field extensions.

\begin{proposition}\label{A8May25}%\marginpar{A8May25}
A finite field extension $L/K$ is a Galois field extension iff $L\rtimes G(L/K)=E(L/K)$.
\end{proposition}

\begin{proof}
A finite field extension $L/K$ is a Galois field extension iff $G(L/K)=[L:K]$ iff 
$$
\dim_K(L\rtimes G(L/K))=[L:K]^2 \;\; (=\dim_K(E(L/K))
$$ 
iff 
 $L\rtimes G(L/K)=E(L/K)$ since $L\rtimes G(L/K)\subseteq E(L/K)$.
\end{proof}

Each unit $u$  of an algebra $A$, determines the {\bf inner automorphism} $\o_u$ of the algebra $A$ which is given  by the rule: For all elements $a\in A$, $\o_u (a):=uau^{-1}$. Elements $g$ of the group $G(L/K)$ are units of the algebra $L\rtimes G(L/K)$ and 
$\o_g(lg')=g(l)gg'g^{-1}$ for all elements $g,g'\in G(L/K)$ and $l\in L$. Therefore the map
$$
G(L/K)\ra \Aut_K(L\rtimes G(L/K)), \;\; g\mapsto \o_g
$$
is a group monomorphism. 

Let $ \CA (L\rtimes G(L/K), L)$ be the set of subalgebras of the algebra $L\rtimes G(L/K)$ that contain the field $L$. A subalgebra $A$ of $L\rtimes G(L/K)$ is called  $G(L/K){\bf-stable}$ if 
$$
\o_g (A)=A\;\; {\rm  for\; all\; elements}\;\;g\in G(L/K).
$$ 
Let $ \CA (L\rtimes G(L/K), L, G(L/K))$ be the set of $G(L/K)$-stable subalgebras of the algebra $L\rtimes G(L/K)$ that contain the field $L$. Corollary \ref{a24Mar25}  describes the sets  $\CA (L\rtimes G(L/K), L)$  and $\CA (L\rtimes G(L/K), L,  G(L/K))$.

 \begin{corollary}\label{a24Mar25}%\marginpar{a24Mar25}
Let $L/K$ be a finite field extension of  arbitrary characteristic. Then:
\begin{enumerate}

\item  $ \CA (L\rtimes G(L/K), L)=\{L\rtimes H\, | \, H$ is a  subgroup of $G(L/K) \}$.

\item  $ \CA (L\rtimes G(L/K),L,  G(L/K))=\{L\rtimes N\, | \, N$ is a normal subgroup of $G(L/K) \}$.

\end{enumerate}
\end{corollary}

\begin{proof} 1. Suppose that  $A\in  \CA (L\rtimes G(L/K), L)$. Since $L\subseteq A$, the algebra $A$ is an $L$-bimodule. By Theorem \ref{BC24Mar25}.(6), the algebra $L\rtimes G(L/K)=\bigoplus_{g\in G(L/K)}Lg$ is a direct sum of  non-isomorphic simple $L$-bimodules. Therefore, 
$$
A=\bigoplus_{g\in H}Lg
$$ 
is a direct sum of  non-isomorphic simple $L$-bimodules for some subset $H$ of $G(L/K)$ that contains the identity element $e$ of the group $G(L/K)$. Since $A$ is an algebra and $LgLh=Lgh$ for all elements $g,h\in G(L/K)$, the set $H$ is a submonoid of the {\em finite} group $G(L/K)$. Therefore, the set $H$ is a subgroup of $G(L/K)$, and statement 1 follows.

2. Suppose that  $A\in  \CA (L\rtimes G(L/K), L, G(L/K))$. By statement 1, $A=\bigoplus_{g\in N}Lg
$ for some subgroup $N$ of the group $G(L/K)$. Since 
$gLng^{-1}=Lgng^{-1}$ for all elements $n\in N$ and $g\in G(L/K)$, the  algebra $A$ belongs to the set $\CA (L\rtimes G(L/K), L, G(L/K))$ iff $gng^{-1}\in N$ for all elements $n\in N$ and $g\in G(L/K)$ iff $N$ is a normal subgroup of $G(L/K)$. 
\end{proof}

 \begin{proposition}\label{A17Apr25}%\marginpar{A17Apr25}
Let $L/K$ be a finite field extension of   characteristic $p$. Then:

\begin{enumerate}

\item  $L\rtimes G(L/K)\cap \CD (L/K)=E(L/L^{G(L/K)}L_{dif})=L$. 

\item $L^{G(L/K)} L_{dif}=L$.

\item If, in addition, the field extension $L/K$ is normal then:

\begin{enumerate}

\item $L^{G(L/K)}=L^{pi}$.

\item $L\rtimes G(L/K)\cap \CD (L/K)=E(L^{pi}L_{dif})=L$.

\item $L=L^{pi}L_{dif}=L^{pi}\t L_{dif}=L^{pi}\t L^{gal}$.

\end{enumerate}
\end{enumerate}
\end{proposition}

\begin{proof} 1. (i) $R:= L\rtimes G(L/K)\cap \CD (L/K)=L$: Clearly, $R\supseteq L$. Let us show that  $R\subseteq L$. Let $r\in L\rtimes G$. Then $r=\sum_{g\in G}l_gg$ for some elements $l_g\in L$. Then, for all elements $l\in L$,
$$ \ad_l(r)=\sum_{g\in G}(l-g(l))l_gg.$$
Therefore, $r\in \CD =\CD (L/K)$, i.e. $r\in R$,  iff
 $r\in Le=L\id_L=L$. 
 
 (ii) $R=E(L/L^G L_{dif})$: By Theorem \ref{BC24Mar25}.(2) and Theorem \ref{AC24Mar25}.(1), $L\rtimes G=E(L/L^G)$ and $\CD (L/K)=E(L/L_{dif})$. By Lemma \ref{b17Apr25}.(1), 
 $$R=E(L/L^G)\cap E(L/L_{dif})=E(L/L^GL_{dif}).$$

2. By statement 1, $L=E(L/L^GL_{dif})=C_{E (L/K)}(L^GL_{dif})$. Since $L\in \Max_s(E(L/K))$,
$$ L=C_{E (L/K)}(L)=C_{E (L/K)}C_{E (L/K)}(L^GL_{dif})=L^GL_{dif},$$
by the Double Centralizer Theorem.

3. (a)  By the assumption, the finite field extension $L/K$ is normal. Then, by Corollary \ref{a25May25}.(1), $L=L^{pi}\t L^{gal}$, and so  $L^G=L^{pi}$, by Corollary \ref{a25May25}.(2). 

%(a) Since $L/K$ is a normal finite field extension, 
% $L^G=L^{pi}$, by Proposition \ref{A12Apr25}.(2). *** add details   ***

(b) The statements (b) is  statement 1  where the field $L^G$ is replaced by the field $L^{pi}$. 
 
(c) By the assumption, the finite field extension $L/K$ is normal. Then, by Corollary \ref{a25May25}, $L=L^{pi}\t L^{gal}$ and $L^{sep}=L^{gal}$. Now, 
 the statement (c) is  statement 2 where the field $L^G$ is replaced by the field $L^{pi}$. 
\end{proof}

{\bf The subalgebra $\CD (L/K)\rtimes G(L/K)$ of the endomorphism algebra $E(L/K)$.}  For a finite field extension $L/K$ of characteristic $p$, the map
%\marginpar{GCDmon}
\begin{equation}\label{GCDmon}
G(L/K)\ra \Aut_K(\CD (L/K)), \;\; g\mapsto \o_g: \d \mapsto g(\d):=g\d g^{-1}
\end{equation}
is a group {\em monomorphism} (since $L\subseteq \CD (L/K)$). Theorem \ref{AB17Apr25} shows that the subalgebra of $E(L/K)$ which is generated by  the automorphism group $G(L/K)$ and the algebra of differential operators $\CD (L/K)$ is the skew group algebra $\CD (L/K)\rtimes G(L/K)$.

 \begin{theorem}\label{AB17Apr25}%\marginpar{AB17Apr25}
Let $L/K$ be a finite field extension of   characteristic $p$. Then:

\begin{enumerate}

\item  The subalgebra of $E(L/K)$ which is generated by the automorphism group $G(L/K)$ and the algebra of differential operators $\CD (L/K)$ is the skew group algebra 
$$\CD (L/K)\rtimes G(L/K)=\bigoplus_{g\in G(L/K)}\CD (L/K)g$$
where the multiplication is given by the rule: For all elements $\d, \d'\in \CD  (L/K)$ and $g,g'\in G(L/K)$,
$$ \d g\cdot \d'g'=\d g\d'g^{-1} \cdot gg'$$ 
 where $g\d'g^{-1}\in \CD (L/K)$, by (\ref{GCDmon}).

\item $\CD (L/K)\rtimes G(L/K)=E\Big(L/L^{G(L/K)}_{dif}\Big)\in \fC \Big(L^{G(L/K)}_{dif}\Big)$ where $L^{G(L/K)}_{dif}:= L^{G(L/K)}\cap L_{dif}$.

\item $C_{E(L/K)}(\CD (L/K)\rtimes G(L/K))=L^{G(L/K)}_{dif}$ and $C_{E(L/K)}\Big(L^{G(L/K)}_{dif}\Big)=\CD (L/K)\rtimes G(L/K)$.

\end{enumerate}
\end{theorem}

\begin{proof} 1. Let $A$ be a subalgebra of $E(L/K)$ which is generated by the algebra $\CD$ and the group $G$. Since 
$$
g\d=g\d g^{-1}\cdot g\;\; {\rm for \; all}\;\; g\in G\;\; {\rm and}\;\; \d \in \CD
$$
and $g\d g^{-1}\in \CD $ (by (\ref{GCDmon})), $A=\sum_{g\in G}\CD g$. It remains to show that the sum is a direct sum. Suppose that $r=\sum _{g\in G}\d_gg$ is a non-trivial relation in the algebra $A$ where not all elements $\d_g\in \CD$ are equal to zero. We seek a contradiction. For all elements $l\in L$,
$$
0=\ad_l(0)=\ad_l(r)=\sum_{g\in G}\Big(\ad_l(\d_g)+l-g(l)\Big)g
$$
where $l-g(l)\in L$ for all elements $g\in G$ and $l\in L$. So we can choose elements $l_1, \ldots , l_s\in L$ such that 
$$
0=\ad_{l_1}\cdots \d_{l_s}(r)=\sum_{g\in G} l_gg
$$
for some  elements $l_g\in L$ not all of them are  equal to zero. This means that  distinct automorphisms of the field extension $L/K$ are (left) $L$-linearly dependent, a contradiction. 

For all elements $\d, \d'\in \CD$ and $g,g'\in G$,
$$
\d g\d' g'=\d g\d'g^{-1}\cdot gg'\;\; {\rm where}\;\;  g\d'g^{-1}\in \CD.
$$
2 and 3. (i) $R:=\CD\rtimes G\in \CA (E(L/K),L, sim)$: Let $r=\sum_{g\in G}\d_gg$ be a nonzero element of the algebra $R$. Let us show that the ideal $(r)$ of the algebra $R$ which is generated by the element $r$ is equal to $R$. 
 This would imply that the algebra $R$ is a simple algebras. Using the same argument as in the proof of statement 1, we can choose elements $l_1, \ldots , l_s\in L$ such that 
$$
\ad_{l_1}\cdots \d_{l_s}(r)=\sum_{g\in G} l_gg\in R
$$
for some  elements $l_g\in L$ not all of them are  equal to zero. By Theorem \ref{BC24Mar25}.(2), the algebra $R=L\rtimes G(L/L^G)$ is a simple algebra. Therefore, $(r)=R$. Clearly, $L\subseteq R$, and so $R\in \CA (E(L/K),L, sim)$.

(ii) $R=E\Big(L/L^{G(L/K)}_{dif}\Big)\in \fC \Big(L^{G(L/K)}_{dif}\Big)$, $C_{E(L/K)}(R)=L^{G(L/K)}_{dif}$ {\em and} $C_{E(L/K)}\Big(L^{G(L/K)}_{dif}\Big)=R$: By the statement (i), $R\in  \CA (E(L/K),L, sim)$. By Theorem  \ref{AC24Mar25}.(6), $\CD = E(L/L_{dif})$ and $L_{dif}=C_{E(L/K)}(\CD)$. By Theorem \ref{BC24Mar25}.(2), $R=E(L/L^G)$ and $L^G=
C_{E(L/K)}(R)$. Then
$$
C_{E(L/K)}(R)=C_{E(L/K)}(\CD\rtimes G)=C_{E(L/K)}(\CD)\cap C_{E(L/K)}(G)=L_{dif}\cap L^G=L^G_{dif}.
$$
By the Double Centralizer Theorem, $R=C_{E(L/K)}\Big( L^G_{dif}\Big)=E(L/L^G_{dif})\in \fC \Big( L^G_{dif}\Big)$.
\end{proof}

In prime characteristic, Corollary \ref{VB-B17Apr25} shows that all normal finite field extensions are B-extensions.

 \begin{corollary}\label{VB-B17Apr25}%\marginpar{VB-B17Apr25}
Let $L/K$ be a normal finite field extension of   characteristic $p$. Then $\CD (L/K)\rtimes G(L/K)=\bigoplus_{g\in G(L/K)}\CD (L/K)g=E(L/K)\in \fCK$, i.e. $L/K$ is a B-extension. 
\end{corollary}

\begin{proof} 1.  By the assumption, the finite field extension $L/K$ is  normal. Then $L=L^{pi}\t L^{sep}$ and $L^{sep}=L^{gal}$ (Corollary \ref{a25May25}). Therefore, $L^{G(L/K)}=L^{pi}$. By Theorem \ref{23May25}, $L_{dif}=L^{sep}$. Therefore, 
$$ L^{G(L/K)}_{dif}=L^{G(L/K)}\cap L_{dif}=L^{pi}\cap L^{sep}=K.$$
Now,
the corollary follows from Theorem \ref{AB17Apr25}.(2), $\CD (L/K)\rtimes G(L/K)=\bigoplus_{g\in G(L/K)}\CD (L/K)g=E(L/L^{G(L/K)}_{dif})=E(L/K)$.
\end{proof}

Corollary \ref{c17Apr25} is a criterion for the finite field extension $L/K$ being a B-extension, i.e. for the  inclusion $\CD (L/K)\rtimes G(L/K)\subseteq E(L/K)$ being the equality.

  \begin{corollary}\label{c17Apr25}%\marginpar{c17Apr25}
Let $L/K$ be a  finite field extension of   characteristic $p$. Then $L/K$  is a B-extension 
iff $L^{G(L/K)}_{dif}=K$.
\end{corollary}

\begin{proof}  By Theorem \ref{AB17Apr25}.(2), 
$\CD (L/K)\rtimes G(L/K)=E\Big(L/L^{G(L/K)}_{dif}\Big)\subseteq E(L/K)$ and statement 1 follows.
 \end{proof}

 The group $G(L/K)=G(L/L^{G(L/K)})$ is a subgroup of the group of units $\GL (L/K):=E(L/K)^\times$ of the algebra $E(L/K)$. It is also a subgroup of the automorphism group $\Aut_K(\CD (L/K))$ of the algebra of differential operators $\CD (L/K)$. Therefore, the group $G(L/K)$ respects the centralizer $C_{E(L/K)}(\CD (L/K)))=L_{dif}$. That is

%\marginpar{gLdif}
\begin{equation}\label{gLdif}
g(L_{dif})=L_{dif}\;\; {\rm for\; all}\;\; g\in G(L/K).
\end{equation}
As a result, there is a group homomorphism
%\marginpar{GLdif}
\begin{equation}\label{GLdif}
\res_{L_{dif}}: G(L/K)\ra G\Big( L_{dif}/L^{G(L/K)}_{dif}\Big), \;\; g\mapsto \o_g:l\mapsto       glg^{-1}=g(l).
\end{equation}
Proposition \ref{AC17Apr25} shows that the homomorphism $\res_{L_{dif}}$ is an {\em isomorphism},
 the field extension $L_{dif}/L^{G(L/K)}_{dif}$ is a Galois field extension with Galois group isomorphic to $G(L/K)=G(L/L^{G(L/K)})$ and 
 $$
 L=L^{G(L/K)}\t_{L^{G(L/K)}_{dif}}L_{dif}.
 $$
  Proposition \ref{AC17Apr25}.(5) presents natural classes of field extensions with identity automorphism groups. Namely, $G(L/L_{dif})=\{ e\}$ and $G\Big(L^{G(L/K)}/L^{G(L/K)}_{dif} \Big)=\{ e\}$.

 \begin{proposition}\label{AC17Apr25}%\marginpar{AC17Apr25}
Let $L/K$ be a  finite field extension of   characteristic $p$. Then:

\begin{enumerate}

\item  $(L_{dif})^{G(L/K)}=L^{G(L/K)}\cap L_{dif}=L^{G(L/K)}_{dif}$.

\item The homomorphism $\res_{L_{dif}}: G(L/K) \ra G\Big( L_{dif}/L^{G(L/K)}_{dif}\Big)$  is an isomorphism.

\item The finite field extension $L_{dif}/L^{G(L/K)}_{dif}$ is a Galois field extension with Galois group 
  $G\Big(L_{dif}/L^{G(L/K)}_{dif}\Big)$ isomorphic to the group $G(L/K)=G(L/L^{G(L/K)})$ via the isomorphism $\res_{L_{dif}}$. In particular, each automorphism of the field extension $L_{dif}/L^{G(L/K)}_{dif}$ can be uniquely lifted to an automorphism  of the field extension $L/K$.

\item $[L_{dif}:L^{G(L/K)}_{dif}]=[L:L^{G(L/K)}]$.

\item  $[L^{G(L/K)}:L^{G(L/K)}_{dif}]=[L:L_{dif}]$, $G(L/L_{dif})=\{ e\}$ and $G\Big(L^{G(L/K)}/L^{G(L/K)}_{dif} \Big)=\{ e\}$.

\item $L=L^{G(L/K)}\t_{L^{G(L/K)}_{dif}}L_{dif}$ and $[L:L^{G(L/K)}_{dif}]=[L^{G(L/K)}:L^{G(L/K)}_{dif}][L_{dif}:L^{G(L/K)}_{dif}]$.

\end{enumerate}
\end{proposition}

\begin{proof} 1. Recall that $g(L_{dif})=L_{dif}$ for all automorphisms $g\in G$. Since $L_{dif}\subseteq L$, we have that 
$$(L_{dif})^G=L^G\cap L_{dif}=L^G_{dif}.$$

2 and 3. Since $L=L^GL_{dif}$ (Proposition \ref{A17Apr25}.(2)), the homomorphism $\res_{L_{dif}}$ is a monomorphism. By statement 1, $(L_{dif})^G=L^G_{dif}$.
Therefore, the field extension $L_{dif}/(L_{dif})^G=L_{dif}/L^G_{dif}$ is a Galois field extension. Since 
$$|G\Big(L_{dif}/L^G_{dif} \Big) |=|G\Big( L/(L_{dif})^G\Big)|=
[L_{dif}: L^G_{dif}]=|G(L/K)|,
$$
the monomorphism $\res_{L_{dif}}$ is an isomorphism. Therefore, the Galois group $G(L_{dif}/L^G_{dif})$ is isomorphic to the Galois group $G(L/L^G)=G(L/K)$ via the isomorphism $\res_{L_{dif}}$. 

4. Statement 4 follows from statement 3.

5. Notice that 
$$
[L:L^G_{dif}]=[L:L^G][L^G:L^G_{dif}]=[L:L_{dif}][L_{dif}:L^G_{dif}].
$$
By statement 4, $[L_{dif}:L^G_{dif}]=[L:L^G]$, and so $[L:L_{dif}]=[L^G:L^G_{dif}]$. Since the homomorphism $\res_{L_{dif}}$ is an isomorphism and $|G(L/L^G_{dif})|=|G(L/K)|<\infty$, we must have that $G(L/L_{dif})=\{ e\}$. Since the homomorphism $\res_{L_{dif}}$ is an isomorphism and $L=L^G\t_{L^G_{dif}}L_{dif}$ (statement 6), we must have that $G(L^G/L^G_{dif})=\{ e\}$.

6. The second equality implies the first. The second equality follows from statement 4: 
$$
[L:L^G_{dif}]=[L:L^G][L^G:L^G_{dif}]\stackrel{{\rm st.}\, 4}{=}[L_{dif}:L^G_{dif}]  [L^G:L^G_{dif}].
$$
\end{proof}

 Let $L/K$ be a  finite field extension of   characteristic $p$. By Proposition \ref{AC17Apr25}, we have a diagram of field extensions:
%\marginpar{diaGLK}
%\begin{displaymath}\label{diaGLK}
%\xymatrix{     &  L=L^{G(L/K)}\t_{L^{G(L/K)}_{dif}} L_{dif} &  \\
%      L^{G(L/K)}\ar[ru]^{gal}&    & L_{dif}%\ar[lu] \\  
%       &  \ar[lu]L^{G(L/K)}_{dif}=\Big( L_{dif} \Big)^{G(L/K)}\ar[ru]^{gal}&\\          & K \ar[u] &\\ }
% \end{displaymath}

%\marginpar{diaGLK}
\begin{eqnarray*}
\begin{tikzcd}\label{diaGLK}
		&  L=L^{G(L/K)}\t_{L^{G(L/K)}_{dif}} L_{dif} &  \\
		L^{G(L/K)}\arrow[ru,"gal"]&    & L_{dif}\arrow[lu] \\  
		&  \arrow[lu]L^{G(L/K)}_{dif}=\Big( L_{dif} \Big)^{G(L/K)}\arrow[ru,"gal"]&\\          & K \arrow[u] &\\
	\end{tikzcd}
\end{eqnarray*}

By Proposition \ref{AC17Apr25}.(2,6), the restriction homomorphism $\res_{L_{dif}}: G(L/K)\ra G\Big(L_{dif}/L^{G(L/K)}_{dif} \Big)$ is an isomorphism. Therefore, the inverse map 
$$\res^{-1}_{L_{dif}}:  G\Big(L_{dif}/L^{G(L/K)}_{dif} \Big)\ra G(L/K)$$
 is given by the obvious extension of every automorphism of the    field extension $L_{dif}/L^{G(L/K)}_{dif} $ to an automorphism of the field extension $L/K$ where the action on the tensor multiple $L^{G(L/K)}$ in $L=L^{G(L/K)}\t_{L^{G(L/K)}_{dif}}L_{dif}$ is trivial.

 Proposition \ref{VB-C17Apr25} is a strengthening of  Proposition \ref{AC17Apr25} where the finite field extension $L/K$ is normal.

  \begin{proposition}\label{VB-C17Apr25}%\marginpar{VB-C17Apr25}
Let $L/K$ be a  normal finite field extension of   characteristic $p$. Then:

\begin{enumerate}

\item  $(L_{dif})^{G(L/K)}=(L_{dif})^{pi}=L^{G(L/K)}\cap L_{dif}=L^{pi}\cap L_{dif}=L^{G(L/K)}_{dif}=L^{pi}_{dif}=K$.

\item The homomorphism $\res_{L_{dif}}: G(L/K) \ra G(L_{dif}/K)=G(L^{sep}/K)$  is an isomorphism.

\item The finite field extension $L_{dif}/K=L^{sep}/K$ is a Galois field extension with Galois group 
  $G(L_{dif}/K)$ isomorphic to the group $G(L/K)=G(L/L^{G(L/K)})=G(L/L^{pi})$ via the isomorphism $\res_{L_{dif}}$. In particular, each automorphism of the field extension $L_{dif}/K$ can be uniquely lifted to an automorphism  of the field extension $L/K$.

\item $[L_{dif}:K]=[L:L^{G(L/K)}]=[L:L^{pi}]$.

\item  The field extensions $L^{pi}/K$ and $L/L_{dif}$ are purely inseparable field extensions,

$[L^{pi}:K]=[L^{G(L/K)}:K]=[L:L_{dif}]$, $G(L/L_{dif})=\{ e\}$ and $G(L^{pi}/K)=\{ e\}$.

\item $L=L^{G(L/K)}\t L_{dif}=L^{pi}\t L_{dif}$ and $[L:K]=[L^{pi}:K][L_{dif}:K]$.

\end{enumerate}
\end{proposition}

\begin{proof} By the assumption,  $L/K$ be a  normal finite field extension of   characteristic $p$. Then
$L^{G(L/K)}=L^{pi}$ (Proposition  \ref{A17Apr25}.(3a)) and $L_{dif}=L^{sep}$ (Theorem \ref{23May25}).
 Now, statement 1 follows and the proposition is a particular case of Proposition \ref{AC17Apr25}.
 \end{proof}

  Let $L/K$ be a normal  finite field extension of   characteristic $p$. By Proposition \ref{VB-C17Apr25}, we have a diagram of field extensions where $L_{dif}=L^{sep}=L^{gal}$:

%\marginpar{diaLpi-VB}
\begin{eqnarray*}
\begin{tikzcd}\label{diaLpi-VB}
		&  L=L^{pi}\t  L_{dif} &  \\
		L^{pi}\arrow[ru,"gal"]&    & L_{dif}\arrow[lu,"pi"] \\
		&  \arrow[lu,"pi"]L^{pi}_{dif}=\Big( L_{dif} \Big)^{pi}=K\arrow[ru,"gal"]&\\
\end{tikzcd}
\end{eqnarray*}

 Let $n=[L:K]$, $m=[L:L^{pi}]=[L_{dif}:K]$ and  $L=[L:L_{dif}]=[L^{pi}:K]$. Then 
 
%\marginpar{n=mkd}
\begin{equation}\label{n=mkd}
n=ml.
\end{equation}

 {\bf Description of the algebra of    differential operators  $\CD (L/K)$ on a finite field extension $L/K$ in prime characteristic.} 
 By Proposition \ref{AC17Apr25}.(6), $L=L^{G(L/K)}\t_{L^{G(L/K)}_{dif}} L_{dif}$. Therefore,

%\marginpar{ELL-1}
\begin{equation}\label{ELL-1}
E\Big(L/L^{G(L/K)}_{dif}\Big)=E\Big(L^{G(L/K)}/ L^{G(L/K)}_{dif}\Big) \t_{L^{G(L/K)}_{dif}} E\Big( L_{dif}/ L^{G(L/K)}_{dif}\Big)
\end{equation}
is a tensor product of algebras.
\begin{proof}
The map 
\begin{eqnarray*}
 E\Big(L^{G(L/K)}/ L^{G(L/K)}_{dif}\Big) \t_{L^{G(L/K)}_{dif}} E\Big( L_{dif}/ L^{G(L/K)}_{dif}\Big)&\ra & E\Big(L/L^{G(L/K)}_{dif}\Big), \\
\phi\t\psi\mapsto  \phi\t\psi :L=L^{G(L/K)}\t_{L^{G(L/K)}_{dif}} L_{dif}&\ra &L=L^{G(L/K)}\t_{L^{G(L/K)}_{dif}} L_{dif}, \\
l\t m &\mapsto & \phi (l)\t \psi (m), 
\end{eqnarray*}
is a  $ L^{G(L/K)}_{dif}$-algebra monomorphism which is also an isomorphism since (where $\CL:= L^{G(L/K)}_{dif}$)
$$\dim_\CL \Big(E(L/\CL) \Big)=[L:\CL]^2=
[L^{G(L/K)}:\CL]^2\cdot[L_{dif}:\CL]^2=\dim_\CL \Big(E(L^{G(L/K)})/\CL) \Big)\cdot \dim_\CL \Big(E (L_{dif}/\CL) \Big).$$
\end{proof}
Theorem \ref{20Apr25} is an explicit  description of the algebra $\CD (L/K)$ of    differential operators on a finite field extension $L/K$ in prime characteristic.

 \begin{theorem}\label{20Apr25}%\marginpar{20Apr25}
Let $L/K$ be a finite field extension of characteristic $p$. Then:
\begin{enumerate}

\item $\CD (L/K)=\CD (L/L_{dif})=\CD\Big(L^{G(L/K)}\t_{L^{G(L/K)}_{dif}} L_{dif}/ L^{G(L/K)}_{dif}\Big)=
L_{dif}\t_{L^{G(L/K)}_{dif}}\CD\Big(L^{G(L/K)}/ L^{G(L/K)}_{dif}\Big)$.
In particular, each differential operator in $\CD\Big(L^{G(L/K)}/ L^{G(L/K)}_{dif}\Big)$ can be uniquely lifted to a differential operator in $\CD (L/K)$.

\item If, in addition, $ L^{G(L/K)}_{dif}=K$ $(\Leftrightarrow L=L^{G(L/K)}\t L_{dif}$, by Proposition \ref{AC17Apr25}.(6)),  i.e. $L/K$ is a B-extension  then $\CD (L/K)=\CD (L^{G(L/K)}\t L_{dif}/K)=L_{dif}\t \CD (L^{G(L/K)}/K)$. 

\end{enumerate}
\end{theorem}

\begin{proof} 1. Since $L_{dif}=C_L(\CD (L/K))=L\cap C_{E(L/K)}(\CD (L/K))$, we have the equality $\CD (L/K)=\CD (L/L_{dif})$.
 By Proposition \ref{AC17Apr25}.(6), $L=L^{G(L/K)}\t_{L^{G(L/K)}_{dif}} L_{dif}$, and so, by (\ref{ELL-1}), 
 $$ E(L/L_{dif})=L_{dif}\t_{L^{G(L/K)}_{dif}}E\Big(L^{G(L/K)}/L^{G(L/K)}_{dif} \Big).$$
 For all elements $l_d\in L_{dif}$, $l^G\in L^{G(L/K)}$ and $l_d'\t \phi\in L_{dif}\t_{L^{G(L/K)}_{dif}}E\Big(L^{G(L/K)}/L^{G(L/K)}_{dif} \Big)=E(L/L_{dif})$ (where $l_d'\in L_{dif}$ and $\phi \in E\Big(L^{G(L/K)}/L^{G(L/K)}_{dif} \Big)$),
 $$
 \ad_{l_dl^G}(l_d'\t \phi)=l_dl_d'\t \ad_{e^G}(\phi).
$$
Therefore,
$$\CD (L/K)=\CD\Big(L^{G(L/K)}\t_{L^{G(L/K)}_{dif}} L_{dif}/ L^{G(L/K)}_{dif}\Big)=
L_{dif}\t_{L^{G(L/K)}_{dif}}\CD\Big(L^{G(L/K)}/ L^{G(L/K)}_{dif}\Big).$$

2. Statement 2 follows from statement 1. 
\end{proof}

 Corollary \ref{a20Apr25} is an explicit  description of the algebra $\CD (L/K)$ of    differential operators on a {\em normal} finite field extension $L/K$ in prime characteristic.

 \begin{corollary}\label{a20Apr25}%\marginpar{a20Apr25}
Let $L/K$ be a normal finite field extension of characteristic $p$. Then  $ L=L^{pi}\t L_{dif}$ (Proposition \ref{VB-C17Apr25}.(6)) and $\CD (L/K)=\CD (L/L_{dif})=\CD (L^{pi}\t L_{dif}/K)=
L_{dif}\t \CD (L^{pi}/K)$.
In particular, each differential operator in $\CD (L^{pi}/K)$ can be uniquely lifted to a differential operator in $\CD (L/K)$. 
\end{corollary}

\begin{proof}  The finite field extension $L/K$ is normal. Therefore, $L^{G(L/K)}=L^{pi}$ (Proposition \ref{A17Apr25}.(3a)), $ L=L^{pi}\t L_{dif}$ (Proposition \ref{VB-C17Apr25}.(6)) and $L^{pi}_{dif}=K$ (Proposition \ref{VB-C17Apr25}.(1)). Now, statement 1 follows from Theorem \ref{20Apr25}.(1).
\end{proof}

 \begin{lemma}\label{a26Apr25}%\marginpar{a26Apr25}
Let $L/K$ be a  finite field extension of characteristic $p$. Then:
\begin{enumerate}

\item $\Big(L/L^{G(L/K)}_{dif}\Big)_{dif}=L_{dif}$ and $\CD (L/K)=\CD \Big(L/L^{G(L/K)}_{dif}\Big)$.

\item $\Big(L/L_{dif}\Big)_{dif}=L_{dif}$.

\item $\Big(L^{G(L/K)}/L^{G(L/K)}_{dif}\Big)_{dif}=L^{G(L/K)}_{dif}$.

%\item If, in addition, $L/K$ is a normal field extension then $\Big(L^{pi}/L^{pi}_{dif}\Big)_{dif}=L^{pi}_{dif}$.

\end{enumerate}
\end{lemma}

\begin{proof} 1. By Corollary \ref{a16Apr25}.(1), $\CD (L/K)=\CD (L/L_{dif})$. It follows from $K\subseteq L^{G(L/K)}_{dif}\subseteq L_{dif}$ that 
$$\CD (L/K)\supseteq \CD \Big(L/L^{G(L/K)}_{dif}\Big)\subseteq \CD (L/L_{dif})=\CD (L/K).$$
Therefore,  $\CD (L/K)=\CD \Big(L/L^{G(L/K)}_{dif}\Big)$. 
 Now, 
$$ \Big(L/L^{G(L/K)}_{dif}\Big)_{dif}:=C_L\Big(\CD \Big(L/L^{G(L/K)}_{dif}\Big)\Big)=C_L(\CD (L/K))=L_{dif}.$$

2. By Corollary \ref{a16Apr25}.(1), $\CD (L/K)=\CD (L/L_{dif})$. Now, 
$$ \Big(L/L_{dif}\Big)_{dif}:=C_L(\CD (L/L_{dif}))=C_L(\CD (L/K))=L_{dif}.$$

3. By Theorem \ref{20Apr25}, 
$\CD (L/K)=
L_{dif}\t_{L^{G(L/K)}_{dif}}\CD\Big(L^{G(L/K)}/ L^{G(L/K)}_{dif}\Big)$. Now,
\begin{eqnarray*}
L^{G(L/K)}_{dif}&=& L^{G(L/K)}\cap L_{dif}=\{ l\in L^{G(L/K)}\, | \, l\d = \d l\;\; {\rm for\; all}\;\; \d \in \CD (L/K)\}\\
&=& \Big\{ l\in L^{G(L/K)}\, | \, l\d = \d l\;\; {\rm for\; all}\;\; \d \in \CD \Big(L^{G(L/K)}/L^{G(L/K)}_{dif}\Big)\Big\}=\Big(L^{G(L/K)}/L^{G(L/K)}_{dif}\Big)_{dif}.
\end{eqnarray*}
%3. Since the finite field extension $L/K$ is a normal     field extension,  $L^G=L^{pi}$ (Proposition \ref{A17Apr25}.(3a)). Now, statement 3 follows from statement 2.  
\end{proof}

 {\bf The algebra $L\rtimes G(L/K)$.}

 \begin{proposition}\label{A20Apr25}%\marginpar{A20Apr25}
Let $L/K$ be a  finite field extension of characteristic $p$. Then 
$$L\rtimes G(L/K)=L^{G(L/K)}\t_{L^{G(L/K)}_{dif}}\Big( L_{dif}\rtimes G \Big(L_{dif}/L^{G(L/K)}_{dif}\Big)\Big).$$
\end{proposition}

\begin{proof}  By Proposition \ref{AC17Apr25}.(6), 
$L=L^{G(L/K)}\t_{L^{G(L/K)}_{dif}}L_{dif}$. By Proposition \ref{AC17Apr25}.(2), the map $\res_{L_{dif}}: G(L/K)\ra G\Big(L_{dif}/ L^{G(L/K)}_{dif}\Big)$ is an isomorphism. Therefore, 
$L\rtimes G(L/K)=L^{G(L/K)}\t_{L^{G(L/K)}_{dif}}\Big(   L_{dif}\rtimes G \Big(L_{dif}/L^{G(L/K)}_{dif}\Big)\Big)$.
\end{proof}

 \begin{corollary}\label{aA20Apr25}%\marginpar{aA20Apr25}
Let $L/K$ be a normal finite field extension of characteristic $p$. Then 
$$L\rtimes G(L/K)=L^{pi}\t \Big(L_{dif}\rtimes G (L_{dif}/K)\Big).$$
\end{corollary}

\begin{proof} The finite field extension $L/K$ is normal. By Proposition \ref{A17Apr25}.(3a), $L^{G(L/K)}=L^{pi}$, $L^{pi}_{dif}=K$ (Proposition \ref{VB-C17Apr25}.(1)) and the corollary follows from Proposition \ref{A20Apr25}.
\end{proof}

{\bf The algebra $ \CD (L/K)\rtimes G(L/K)$ is a tensor product of two subalgebras.}  Theorem  \ref{bA20Apr25} shows that the algebra $ \CD (L/K)\rtimes G(L/K)$ is a tensor product of  the subalgebras $\CD \Big(L^{G(L/K)}/L^{G(L/K)}_{dif}\Big)$ and $ L_{dif}\rtimes  G\Big(L_{dif}/L^{G(L/K)}_{dif}\Big)$ over the field $L^{G(L/K)}_{dif}$.

 \begin{theorem}\label{bA20Apr25}%\marginpar{bA20Apr25}
Let $L/K$ be a  finite field extension of characteristic $p$. Recall that 
 $L=L^{G(L/K)}\t_{L^{G(L/K)}_{dif}} L_{dif}$  (Proposition \ref{AC17Apr25}.(6)) and  
 $$E\Big(L/L^{G(L/K)}_{dif}\Big)=E\Big(L^{G(L/K)}/ L^{G(L/K)}_{dif}\Big) \t_{L^{G(L/K)}_{dif}} E\Big( L_{dif}/ L^{G(L/K)}_{dif}\Big)
 $$
is a tensor product of algebras, see  (\ref{ELL-1}). Then $$ E(L/L^{G(L/K)}_{dif})=\CD (L/K)\rtimes G(L/K)=\CD \Big(L^{G(L/K)}/L^{G(L/K)}_{dif}\Big)\t_{L^{G(L/K)}_{dif}}L_{dif}\rtimes  G\Big(L_{dif}/L^{G(L/K)}_{dif}\Big)$$
is a tensor product of algebras such that 
$$E\Big(L^{G(L/K)}/ L^{G(L/K)}_{dif}\Big)=\CD \Big(L^{G(L/K)}/L^{G(L/K)}_{dif}\Big)\;\; {\rm and}\;\; 
  E\Big( L_{dif}/ L^{G(L/K)}_{dif}\Big)= L_{dif}\rtimes  G\Big(L_{dif}/L^{G(L/K)}_{dif}\Big).$$

\end{theorem}

\begin{proof} By  Theorem \ref{AB17Apr25}.(2), $E\Big(L^{G(L/K)}/ L^{G(L/K)}_{dif}\Big)=\CD (L/K)\rtimes G(L/K)$.  By Theorem \ref{20Apr25}.(1) and Proposition \ref{AC17Apr25}.(2), we have  the equalities 
$$\CD (L/K)= L_{dif}\t_{L^{G(L/K)}_{dif}}\CD \Big( L^{G(L/K)}/ L^{G(L/K)}_{dif}\Big)\;\; {\rm and}\;\; G(L/K)=G\Big(L_{dif}/L^{G(L/K)}/ L^{G(L/K)}_{dif} \Big),$$
respectively.  Therefore, 
$$\CD (L/K)\rtimes G(L/K)=\CD \Big(L^{G(L/K)}/L^{G(L/K)}_{dif}\Big)\t_{L^{G(L/K)}_{dif}}L_{dif}\rtimes  G\Big(L_{dif}/L^{G(L/K)}_{dif}\Big)$$
where
$$
\CD \Big(L^{G(L/K)}/L^{G(L/K)}_{dif}\Big)\subseteq E\Big(L^{G(L/K)}/ L^{G(L/K)}_{dif}\Big)\;\; {\rm and} \;\; 
L_{dif}\rtimes  G\Big(L_{dif}/L^{G(L/K)}_{dif}\Big)\subseteq  E\Big( L_{dif}/ L^{G(L/K)}_{dif}\Big).
$$
Hence, $$E\Big(L^{G(L/K)}/ L^{G(L/K)}_{dif}\Big)=\CD \Big(L^{G(L/K)}/L^{G(L/K)}_{dif}\Big)\;\; {\rm and}\;\; 
  E\Big( L_{dif}/ L^{G(L/K)}_{dif}\Big)= L_{dif}\rtimes  G\Big(L_{dif}/L^{G(L/K)}_{dif}\Big).$$
So, 
$$ E(L/L^{G(L/K)}_{dif})=\CD (L/K)\rtimes G(L/K)=\CD \Big(L^{G(L/K)}/L^{G(L/K)}_{dif}\Big)\t_{L^{G(L/K)}_{dif}}L_{dif}\rtimes  G\Big(L_{dif}/L^{G(L/K)}_{dif}\Big)$$
is a tensor product of algebras.
\end{proof}

If, in addition, the field extension $L/K$ is normal, Corollary \ref{cA20Apr25} shows that the algebra $ \CD (L/K)\rtimes G(L/K)$ is a tensor product of  the subalgebras $\CD (L^{pi}/K)$ and $ L^{sep}\rtimes  G (L^{sep}/K)$.

 \begin{corollary}\label{cA20Apr25}%\marginpar{cA20Apr25}
Let $L/K$ be a normal finite field extension of characteristic $p$. Recall that 
 $L=L^{pi}\t L^{sep}$  (Proposition  \ref{VB-C17Apr25}.(6)) and  
 $E(L/K)=E(L^{pi}/K) \t E ( L^{sep}/K)
 $
is a tensor product of algebras, see  (\ref{ELL-1}). Then 
$$ 
E(L/K)=\CD (L/K)\rtimes G(L/K)=\CD (L^{pi}/K)\t \Big(L^{sep}\rtimes  G(L^{sep}/K)\Big)
$$
is a tensor product of algebras such that 
$$E(L^{pi}/K)=\CD (L^{pi}/K)\;\; {\rm and}\;\; 
  E( L^{sep}/K)= L^{sep}\rtimes  G(L^{sep}/K).$$

\end{corollary}

\begin{proof} The finite field extension $L/K$ is normal. Therefore, $L^{G(L/K)}=L^{pi}$ (Proposition \ref{A17Apr25}.(3a)), $ L=L^{pi}\t L_{dif}$ (Proposition \ref{VB-C17Apr25}.(6)) and $L^{pi}_{dif}=K$ (Proposition \ref{VB-C17Apr25}.(1)). Now, statement 1 follows from  Theorem \ref{bA20Apr25}.
\end{proof}

{\bf $\CD (L/K)\rtimes G(L/K)\simeq \bigotimes_{i=1}^n\CD (L_i/K)\rtimes G(L_i/K)$ where  $L=\bigotimes_{i=1}^nL_i$.}

\begin{proposition}\label{A7Jun25}%\marginpar{A7Jun25}
Suppose that $L_i/K$ $(i=1, \ldots , n)$ are normal finite field extensions such that their tensor product $L=\bigotimes_{i=1}^nL_i$ is a field. Then: 
\begin{enumerate}

\item The finite field extension $L/K$ is normal, $L=L^{pi}\t L^{gal}$, $L^{pi}= \bigotimes_{i=1}^nL_i^{pi} $ and $  L^{gal}=\bigotimes_{i=1}^nL_i^{gal}$.

\item $\CD(L)=\bigotimes_{i=1}^n\CD (L_i/K)$ and $G(L/K)\simeq \prod_{i=1}^n G(L_i/K)$.

\item $\CD (L/K)\rtimes G(L/K)\simeq \bigotimes_{i=1}^n\CD (L_i/K)\rtimes G(L_i/K)$.

\end{enumerate}
\end{proposition}

\begin{proof} 1. Statement 1 is obvious.

2. By Proposition \ref{7Jun25}, $\CD(L)=\bigotimes_{i=1}^n\CD (L_i/K)$. Now,
$$G(L/K)=G(L^{pi}\t L^{gal}/K)\simeq G( L^{gal}/K)=
G\bigg( \prod_{i=1}^n L_i^{gal}/K\bigg)
\simeq \prod_{i=1}^n G(L_i/K).
$$
3. By statement 2,
$$
\CD (L/K)\rtimes G(L/K)\simeq \bigg( \bigotimes_{i=1}^n\CD (L_i/K)\bigg)\rtimes \prod_{i=1}^n G(L_i/K)\simeq \bigotimes_{i=1}^n\CD (L_i/K)\rtimes G(L_i/K).$$
\end{proof}

 %%%%%%%%%%%%%%%%%% SECTION 6 %%%%%%%%%%%%%%%%%%%%%

\section{The classes of B-extensions and  normal finite field extensions coincide} \label{B-EXT=NORMAL} %\marginpar{B-EXT=NORMAL} 

 The aim of the section is to give a proof of Theorem \ref{VVB-30Apr25} which is a characterization of B-extensions. In particular, it shows that the classes of B-extensions and  normal finite field extensions coincide. 
Corollary \ref{VVB-a4May25} is a characterization of B-extensions $L/K$ with $\CD (L/K)=E(L/K)$.
  Corollary \ref{VVB-c4May25} is a characterization of B-extensions $L/K$ with $L\rtimes G(L/K)=E(L/K)$. \\

{\bf Normality criterion for field extensions.} 
\begin{definition}
Suppose that $K$ is a field of prime characteristic $p>0$. Then each non-scalar polynomial $f(x)\in K[x]$ admits a unique presentation 

%\marginpar{f=fsepxn}
\begin{equation}\label{f=fsepxn}
 f(x)=f^{sep}(x^{p^n})\;\; {\rm where}\;\; f^{sep}(x)\in K[x]\;\; {\rm is\; a\; separable\; polynomial\; and}\;\; n\geq 0.
\end{equation}
The equality (\ref{f=fsepxn}) is called a {\bf separable presentation} of the polynomial $f(x)$. The polynomial $f^{sep}(x)$ is called the {\bf separable part} of $f$ and the natural number $n$ is called the {\bf inseparability degree} of $f(x)$ and denoted by $\deg_{ins}(f)$. 
\end{definition}

For the polynomial $f(x)=\sum_{i\geq 0} \mu_ix^i$,  ${\rm coef}(f):= \{\mu_i\, | \, i\geq 0\}$ is the 
  set of its coefficients. Clearly, 
  %\marginpar{f=fsepxn-2}
\begin{equation}\label{f=fsepxn-2}
 {\rm coef}(f)={\rm coef}(f^{sep}).
\end{equation}
Notice that 
 %\marginpar{f=fsepxn-1}
\begin{equation}\label{f=fsepxn-1}
\deg(f)=p^n\deg (f^{sep})\;\; {\rm where}\;\; n=\deg_{ins}(f).
\end{equation}
For the polynomial $f(x)$ as in (\ref{f=fsepxn}),   $f(x)=\sum_{i\geq 0} \l_ix^{ip^n}=\bigg(\sum_{i\geq 0} \l_i^\frac{1}{p^n} x^i\bigg)^{p^n}$, where $n=\deg_{ins}(f)$, and so 

%\marginpar{f=fsepxn-3}
\begin{equation}\label{f=fsepxn-3}
f(x)=\bigg( {f^{sep}}^\frac{1}{p^n}\bigg)^{p^n}\;\; {\rm where}\;\;  {f^{sep}}^\frac{1}{p^n}:=\sum_{i\geq 0} \l_i^\frac{1}{p^n} x^i\in K\Big({\rm coef}(f)^\frac{1}{p^n}\Big)[x]
\end{equation}
is a {\em separable} polynomial over the field $K\Big({\rm coef}(f)^\frac{1}{p^n}\Big)$. Clearly,
%\marginpar{f=fsepxn-4}
\begin{equation}\label{f=fsepxn-4}
{\rm roots}(f)={\rm roots}( {f^{sep}}^\frac{1}{p^n}).
\end{equation}

\begin{proposition}\label{A12Apr25}%\marginpar{A12Apr25}
Suppose that $K$ is a field of prime characteristic $p>0$, $L/K$ is a normal field extension, i.e. it is a splitting field of  a set of polynomials $\{ f_i(x)=f_i^{sep}(x^{p^{n_i}})\, | \,i\in I\}$, and $L^{pi}/K$ is the largest purely inseparable field extension in the field $L/K$. Then: 
\begin{enumerate}

\item  $L^{pi}=K\bigg(\bigcup_{i\in I} {\rm coef}(f_i)^\frac{1}{p^{n_i}} \bigg)$ where ${\rm coef}(f_i):=\{ \l^\frac{1}{p^{n_i}}\, | \, \l \in 
{\rm coef}(f_i)\}$.

\item $L/L^{pi}$ is a Galois field extension which is generated by the roots of the polynomials  $\{ f_i\, | \, i\in I\}$ or, equivalently, by the roots of separable polynomials $\{ {f_i^{sep}}^\frac{1}{p^{n_i}})\, | \, i\in I\}\subseteq L^{pi}[x]$ (this follows from (\ref{f=fsepxn-4})).

\end{enumerate}
\end{proposition}

\begin{proof}  Clearly, $L^{pi}\supseteq L^{pi}_c:=K\bigg(\bigcup_{i\in I} {\rm coef}(f_i)^\frac{1}{p^{n_i}} \bigg)$ and the field extension 
$$
L/L^{pi}_c=L^{pi}_c\bigg(\bigcup_{i\in I} {\rm roots}\Big({f_i^{sep}}^\frac{1}{p^{n_i}} \Big) \bigg)
$$ 
is a Galois field extension as it is a splitting field of the set of separable polynomials $\Big\{ {f_i^{sep}}^\frac{1}{p^{n_i}})\, | \, i\in I\Big\}\subseteq L^{pi}_c[x]$. Since $L_c^{pi}\subseteq L^{pi}\subseteq L$, $L^{pi}/L^{pi}_c$ is a purely inseparable field extension and $L/L_c^{pi}$ is a Galois field extension, we must have $L_c^{pi}=L^{pi}$. 
\end{proof}

Theorem \ref{A11May25} is a normality criterion for     a  field extension.

\begin{theorem}\label{A11May25}%\marginpar{A11May25}
Suppose that $K$ is a field of prime characteristic $p>0$, $L/K$ is a  field extension. Then the field extension $L/K$ is a normal field extension iff $L^{G(L/K)}=L^{pi}$. 
\end{theorem}

\begin{proof} Clearly, $L^{pi}\subseteq L^{G(L/K)}$.

$(\Rightarrow)$ Suppose that the field extension $L/K$ is a normal field extension. Then, by Proposition \ref{A12Apr25}.(2), the field extension $L/L^{pi}$ is a Galois field extension, and so $L^{G(L/K)}=L^{pi}$.

$(\Leftarrow)$ Suppose that  $L^{G(L/K)}=L^{pi}$. We have to show that  $\s (L)=L$ for all automorphisms $\s\in G(\bK /K)$. Clearly, $\s (L^{pi})=L^{pi}$ for all automorphisms $\s\in G(\bK /K)$. Therefore, $\s\in  G(\bK/L^{pi})$. Since the field extension $L/L^{G(L/K)}=L/L^{pi}$ is a Galois field extension, $\s (L)=L$ for all automorphisms $\s\in G(\bK /K)$. This means that the field extension $L/K$ is a normal field extension.
\end{proof}

{\bf Characterization of B-extensions.}
  Theorem \ref{VVB-30Apr25} is a characterization of  B-extensions. 
  
   \begin{theorem}\label{VVB-30Apr25}%\marginpar{VVB-30Apr25}
Let $L/K$ be a finite field extension  of prime characteristic. Then the  following statements are equivalent:
\begin{enumerate}

\item $L/K$ is a B-extension.

\item $L=L^{G(L/K)}\t L_{dif}$.

%\item $L=L^{pi}\t L_{dif}$.

\item $L^{G(L/K)}_{dif}=K$.

%\item $L^{pi}\cap  L_{dif}=\{ e\}$.

\item $L=L^{G(L/K)}\t L^{gal}$ and  $L^{G(L/K)}/K$ is a normal the field  extension.

\item $L^{G(L/K)}=L^{pi}$. 

%\item ***neverno $L=L^{pi}\t L_{dif}$.***

%\item ***neverno   $L^{pi}_{dif}=K$.  ***

\item $L=L^{pi}\t L^{gal}$.

\item  $L/K$ is normal.

\end{enumerate} 
If the equivalent conditions hold then:

\begin{enumerate}

%\item[(i)] $L^{G(L/K)}=L^{pi}$ and $L_{dif}= L^{gal}$.

\item[(a)] $L^{G(L/K)}=L^{pi}$,  $\CD (L^{G(L/K)}/K)=E(L^{G(L/K)}/K)$ and  $\Big(L^{G(L/K)}/K\Big)_{dif}=K$. 

\item[(b)]    $L_{dif}=L^{sep}= L^{gal}$,  $L_{dif}\rtimes G(L_{dif}/K)=E(L_{dif}/K)$, $(L_{dif})^{G(L_{dif}/K)}=K$ and  
\begin{eqnarray*}
E(L/K)&=&\CD (L/K)\rtimes G(L/K)=\CD (L^{G(L/K)}/K) \t \Big(L_{dif}\rtimes G(L_{dif}/K) \Big)\\
&=&\CD (L^{G(L/K)}/K) \t \Big(L^{gal}\rtimes G(L^{gal}/K) \Big)=\CD (L^{pi}/K) \t \Big(L_{dif}\rtimes G(L_{dif}/K) \Big)\\
&=&\CD (L^{pi}/K) \t \Big(L^{gal}\rtimes G(L^{gal}/K) \Big).
\end{eqnarray*}

\end{enumerate}
\end{theorem}
  
  \begin{proof}   $(1\Leftrightarrow 2\Leftrightarrow 3)$ $L/K$ is a B-extension iff $E(L/K)=\CD (L/K)\rtimes G(L/K)=E\Big(L/L^{G(L/K)}_{dif}\Big)$ (Theorem \ref{AB17Apr25}.(2)) iff $L^{G(L/K)}_{dif}=K$ iff 
$$L=L^{G(L/K)}\t L_{dif}$$
since $L=L^{G(L/K)}\t_{L^{G(L/K)}_{dif}} L_{dif}$ (Proposition \ref{AC17Apr25}).(6)).

 $(3\Rightarrow 4)$ Recall that $L_{dif}=L^{sep}$ (Theorem \ref{23May25}). Suppose that $L^{G(L/K)}_{dif}=K$. Then, by the diagram after Proposition \ref{AC17Apr25},  %(\ref{diaGLK}), 
 the condition $K=L^{G(L/K)}_{dif}=\Big( L_{dif}\Big)^{G(L/K)}$  implies that the field extension  $L_{dif}/K=L^{sep}/K$ is a Galois field extension, $L=L^{G(L/K)}\t L_{dif}$ and $L_{dif}=L^{sep}=L^{gal}$.
 
By the equality $L^{G(L/K)}_{dif}=K$  and Theorem  \ref{bA20Apr25}, we have that 
$$
E(L^{G(L/K)}/K )=\CD (L^{G(L/K)}/K).
$$
 Then,  by Theorem \ref{B11May25},  the finite field extension $L^{G(L/K)}/K$ is purely inseparable, hence normal.

 $(4\Rightarrow 5)$ Clearly, $L^{pi}\subseteq L^{G(L/K)}$. Suppose that $L=L^{G(L/K)}\t L^{gal}$ and the field  extension $L^{G(L/K)}/K$ is normal. Since the field  extension $L^{G(L/K)}/K$ is normal and $L^{pi}\subseteq L^{G(L/K)}$, we have that $\Big(L^{G(L/K)}\Big)^{pi}=L^{pi}$ and, by Corollary \ref{a25May25},
 $$
 L^{G(L/K)}= \Big(L^{G(L/K)}\Big)^{pi}\t \Big(L^{G(L/K)} \Big)^{gal} =L^{pi}\t\Big(L^{G(L/K)} \Big)^{gal}. $$
 Now, the condition $L=L^{G(L/K)}\t L^{gal}$ implies that $\Big(L^{G(L/K)} \Big)^{gal}=K $, and so $
 L^{G(L/K)}= L^{pi}$.

 $(5\Leftrightarrow 7)$ The equivalence is  Theorem \ref{A11May25}.

  $(6\Leftrightarrow 7)$ The equivalence is  Corollary  \ref{a25May25}.(1). 
 
$(7\Rightarrow 2)$ Recall that statements 5--7 are equivalent. Suppose that statement 7 holds, i.e. the finite field extension $L/K$ is normal and  
$L^{G(L/K)}=L^{pi}$ and  $L^{sep}= L^{gal}$  and 
$$
L=L^{pi}\t L^{gal}=L^{G(L/K)}\t L^{sep}.$$
Then, by Theorem \ref{23May25}, $L^{sep}=L_{dif}$, and so 
$$L=L^{G(L/K)}\t L^{sep}=L^{G(L/K)}\t L_{dif},$$
as required.

(a) Suppose that the field extension $L/K$ is a B-extension. Then $L^{G(L/K)}=L^{pi}$ (statement 6). Now,
$$ E\Big(L^{G(L/K)}/\Big(L^{G(L/K)}/K \Big)_{dif} \Big)\stackrel{{\rm Thm.}\, \ref{AC24Mar25}.(1)}{=}\CD (L^{G(L/K)}/K)=\CD (L^{pi}/K)\stackrel{{\rm Thm.}\, \ref{B11May25}}{=}E(L^{pi}/K)=
E(L^{G(L/K)}/K), 
$$
and so   $\CD (L^{G(L/K)}/K)=E(L^{G(L/K)}/K)$ and  $\Big(L^{G(L/K)}/K\Big)_{dif}=K$.

(b) Suppose that the field extension $L/K$ is a B-extension. Then the equalities $L_{dif}=L^{sep}$ and $L^{G(L/K)}=L^{pi}$, the inclusion $L^{gal}\subseteq L^{sep}$ and the fact that 
$$
L=L^{G(L/K)}\t L_{dif}=L^{pi}\t L^{sep}\supseteq L^{pi}\t L^{gal}=L
$$
imply the equality $L^{pi}\t L^{sep}= L^{pi}\t L^{gal}$. Therefore,  $L_{dif}=L^{sep}= L^{gal}$  (since $L^{sep}\supseteq  L^{gal}$). Now, 
 $L^{gal}\rtimes G(L^{gal}/K)=E(L^{gal}/K)$ (Theorem \ref{BC24Mar25}.(1,2))  and  $\Big(L^{gal}\Big)^{G(L^{gal}/K)}=K$.

\begin{eqnarray*}
E(L/K)&=&E(L^{pi}\t L^{gal}/K)=E(L^{pi}/K)\t E(L^{gal}/K)=\CD (L^{pi}/K)\t L^{gal}\rtimes G(L^{gal}/K)\\
&=& \Big(L^{gal}\t \CD (L^{pi}/K)\Big) \rtimes G(L^{gal}/K)=\CD (L/K)\rtimes G(L/K).
\end{eqnarray*}
The rest follows from the equalities $L_{dif}=L^{sep}=L^{gal}$ and $L^{G(L/K)}=L^{pi}$. 
\end{proof}

Corollary \ref{VVB-a4May25} is a characterization of B-extensions $L/K$ with $\CD (L/K)=E(L/K)$.

\begin{corollary}\label{VVB-a4May25}%\marginpar{VVB-a4May25}
Let $L/K$ be a finite field extension of prime characteristic. Then the  following statements are equivalent:
\begin{enumerate}

\item $\CD (L/K)=E(L/K)$, i.e. $L/K$ is a B-extension with  $\CD (L/K)=E(L/K)$.

\item $L_{dif}=K$.

\item $L^{G(L/K)}=L$ and $L/K$ is a B-extension.
  
\item $G(L/K)=\{ e\}$ and $L/K$ is a B-extension.

\item $L^{pi}=L$.

\end{enumerate}
\end{corollary}

\begin{proof} $(1\Leftrightarrow 5)$  Theorem \ref{B11May25}.

$(1\Leftrightarrow 4)$ For the $B$-extension $L/K$, $E(L/K)=\CD (L/K)\rtimes G(L/K)$.
Therefore $E(L/K)=\CD (L/K)$ iff $G(L/K)=\{ e\}$.

$(2\Rightarrow 3)$ Suppose that $L_{dif}=K$. Then $L^{G(L/K)}_{dif}=K$, and so the field extension $L/K$ is a B-extension, by Theorem \ref{VVB-30Apr25}. Now, by Theorem \ref{VVB-30Apr25}, $L=L^{G(L/K)}\t L_{dif}$. Therefore,  $L=L^{G(L/K)}$ (since $L^{G(L/K)}_{dif}=K$) and $L/K$ is a B-extension. 

$(4\Rightarrow 2)$ Suppose that $G(L/K)=\{ e\}$ and $L/K$ is a B-extension. Then, by Theorem \ref{VVB-30Apr25}, $L=L^{G(L/K)}\t L_{dif}$ and $G(L/K)=\{ e\}$. Therefore, $L=L\t L_{dif}$, and so the equality 
$$
[L:K]=[L:K][L_{dif}:K]
$$ 
 implies the equality $[L_{dif}:K]=1$. This means that  $L_{dif}=K$. 

$(4\Leftrightarrow 3)$ By Theorem \ref{VVB-30Apr25},  
for the $B$-extension $L/K$, $L=L^{G(L/K)}\t L^{gal}$. Therefore, $G(L/K)=\{ e\}$ iff  $L=L^{G(L/K)}$ iff
$L^{gal}=K$.
\end{proof}

%******************
%$(4\Leftrightarrow 2)$ By Theorem \ref{VVB-30Apr25},  
%for the $B$-extension $L/K$, $L=L^{G(L/K)}\t L_{dif}$. Therefore, $G(L/K)=\{ e\}$ iff  $L=L^{G(L/K)}$ iff
%$L_{dif}=K$.
%************************

%$(4\Leftrightarrow 5)$  By Theorem \ref{VVB-30Apr25},  
%for the $B$-extension $L/K$, $L^{G(L/K)}=L^{pi}$. Therefore, $G(L/K)=\{ e\}$ iff  $L=L^{G(L/K)}=L^{pi}$.

  Corollary \ref{VVB-c4May25} is a characterization of B-extensions $L/K$ with $L\rtimes G(L/K)=E(L/K)$, i.e. it is a  characterization of Galois field  extensions $L/K$. In Corollary \ref{VVB-c4May25}, the equivalences $(3\Leftrightarrow 4\Leftrightarrow 6)$ are known.

\begin{corollary}\label{VVB-c4May25}%\marginpar{VVB-c4May25}
Let $L/K$ be a finite field extension of prime characteristic. Then the  following statements are equivalent:
\begin{enumerate}

\item  $L\rtimes G(L/K)=E(L/K)$, i.e. $L/K$ is a B-extension with  $L\rtimes G(L/K)=E(L/K)$.

\item $L_{dif}=L$  and $L/K$ is a B-extension.

\item $L^{G(L/K)}=K$.

\item The field extension $L/K$ is a Galois extension.

\item $\CD (L/K)=L$ and $L/K$ is a B-extension.

\item $L/K$ is normal and $L^{pi}=K$, i.e. $L/K$ is a normal and separable.

\item $L/K$ is normal and $L_{dif}=L$.

\end{enumerate}
\end{corollary}

\begin{proof} $(1\Leftrightarrow 5)$ For the $B$-extension $L/K$, 
$$
E(L/K)=\CD (L/K)\rtimes G(L/K)\supseteq L\rtimes G(L/K)\;\; {\rm  and}\;\; \CD(L/K)\supseteq L.
$$
Therefore $E(L/K)=L\rtimes G(L/K)$ iff  $\CD(L/K)=L$.

$(2\Leftrightarrow 5)$ By the definition, $L_{dif}=C_{E(L/K)}(\CD (L/K))$. Then,  by the Double Centralizer Theorem, 
$$
\CD (L/K)   =C_{E(L/K)}(L_{dif}).
$$
By Theorem \ref{VVB-30Apr25},  
for the $B$-extension $L/K$, $L=L^{G(L/K)}\t L_{dif}$.  
Therefore, $\CD (L/K)=L$ iff 
$$
L_{dif}=C_{E(L/K)}(\CD (L/K))=C_{E(L/K)}(L)=L
$$ iff $L^{G(L/K)}=K$.

$(2\Leftrightarrow 3)$ Notice that the equality  $L^{G(L/K)}=K$ implies the equality $L^{G(L/K)}_{dif}=K$, and so the field extension $L/K$ is a $B$-extension (Theorem \ref{VVB-30Apr25}).
 By Theorem \ref{VVB-30Apr25},  
for the $B$-extension $L/K$, $L=L^{G(L/K)}\t L_{dif}$.  
Therefore, $L_{dif}=L$ iff $L^{G(L/K)}=K$. 

$(2\Leftrightarrow 4)$   Recall that statements 2 and 3 are equivalent. In particular, the field extension with   $L^{G(L/K)}=K$ is a $B$-extension. Now, by Theorem \ref{VVB-30Apr25},  
for the $B$-extension $L/K$, $L=L^{G(L/K)}\t L^{gal}$. Therefore, $L^{G(L/K)}=K$ iff 
$L=L^{gal}$.

$(1\Leftrightarrow 6)$ Recall that statements 1--5 are equivalent.  By Theorem \ref{VVB-30Apr25},  
 $L/K$ is a $B$-extension iff $L/K$ is normal and $L^{pi}\cap L_{dif}=K$, and in this case 
 $$
 L^{G(L/K)}=L^{pi}.
 $$ 
 Therefore, $L\rtimes G(L/K)=E(L/K)$ iff $L\rtimes G(L/K)=E(L/K)$  and $L/K$ is a $B$-extension iff $L^{pi}=L^{G(L/K)}=K$,  $L/K$ is normal and $L^{pi}\cap L_{dif}=K$ iff 
$L/K$ is normal and $L^{pi}=K$.

$(4\Leftrightarrow 7)$ By Theorem \ref{23May25}, $L_{dif}=L^{sep}$. So,  statement 7 is equivalent to statement 4. 
\end{proof}
%*** old proof ***
%$(6\Leftrightarrow 7)$ Notice that statements 1--6 are equivalent. By Theorem \ref{VVB-30Apr25},  
% $L/K$ is a $B$-extension iff $L/K$ is normal and $L=L^{pi}\t L_{dif}$.
%Therefore,  
%$L/K$ is normal, $L=L^{pi}\t L_{dif}$ and $L^{pi}=K$ iff $L/K$ is normal and $L=L_{dif}$ and  $L^{pi}=K$.
% Therefore,  
%$L/K$ is normal and $L^{pi}=K$ iff $L/K$ is normal and $L=L^{pi}\t L_{dif}=L_{dif}$.
%***

\begin{corollary}\label{VVB-d4May25}%\marginpar{VVB-d4May25}
The class of B-extensions that are neither G- nor D-extensions is  precisely the class of normal field extensions that are neither Galois nor purely inseparable field extensions.
\end{corollary}

\begin{proof} Recall that every B-extension is a normal finite field extension and vice versa. So,  the corollary follows from Corollary \ref{VVB-a4May25} and Corollary \ref{VVB-c4May25}.
\end{proof}

 %%%%%%%%%%%%%%%%%% SECTION 7 %%%%%%%%%%%%%%%%%%%%%
 
 \section{Analogue of the Galois Theory for normal  field extensions} \label{AN-GT-NORMAL} %\marginpar{AN-GT-NORMAL}

 The aim of the section is to prove  Theorem \ref{1Jun25} and  Theorem \ref{11Jun25}.  
 For normal field extensions, Theorem \ref{1Jun25} gives an order reversing  Galois-type correspondence for their subfields. Theorem \ref{11Jun25} describes normal subfields of normal field extensions and  establishes an analogue of the Galois correspondence  for normal subfields.\\

 {\bf Analogue of the Galois Theory for normal field extensions.}  Theorem  \ref{1Jun25} is one of the main results of the paper. For normal  field extensions, it establishes an analogue of the Galois correspondence.

\begin{definition} 
For a finite field extension $L/K$ and its intermediate subfield $K\subseteq M\subseteq L$, let
\begin{eqnarray*}
\CD (L/K)^M &:=& \{ \d \in \CD (L/K)\, | \, \d m=m\d\;\; {\rm for\; all}\;\; m\in M\}=C_{\CD (L/K)}(M),\\
 G(L/K)^M&:=& \{ g\in G(L/K)\, | \, g(m)=m\;\; {\rm for\; all}\;\; m\in M \}=C_{G(L/K)}(M).
\end{eqnarray*}
\end{definition} 
 By the definitions, $\CD (L/K)^M$ is a subalgebra of the algebra $\CD (L/K)$ that contains $L$ and $G(L/K)^M$ is a subgroup of the group $G(L/K)$.

\begin{definition}
For a finite field extension $L/K$ and an algebra $A\in  \CA (E(L/K),L)$,  let
\begin{eqnarray*}
C_{L}( A\cap \CD (L/K)_+) &:=& \{ l \in L\, | \, \d l=l\d\;\; {\rm for\; all}\;\; \d\in A\cap \CD (L/K)_+\}\\
&\stackrel{{\rm Cor.}\, \ref{a1Jun25} }{=}& L^{A\cap \CD (L/K)_+} :=
\{ l \in L\, | \, \d (l)=0\;\; {\rm for\; all}\;\; \d\in A\cap \CD (L/K)_+\},\\
 L^{A\cap G(L/K)}&:=& \{  l \in L\, | \, g(l)=l\;\; {\rm for\; all}\;\; g\in A\cap G(L/K)\}=C_{L}( A\cap G (L/K)).
\end{eqnarray*}
\end{definition} 
 By the definition, the sets $L^{A\cap \CD (L/K)_+}$ and $L^{A\cap G(L/K)}$ are subfields of $L$ that contain the field $K$ and the intersection $A\cap G(L/K)$ is a subgroup of $G(L/K)$.

\begin{lemma}\label{c1Jun25}%\marginpar{c1Jun25}
Let $L/K$ be a finite field extension and $M\in \CF (L/K)$. Then:
\begin{enumerate}

\item  $L\subseteq E(L/M)\subseteq E(L/K)$ and 
$\CD (L/K)^M=\CD (L/M) =E(L/M)\cap \CD (L/K)\subseteq  E(L/K)$.

\item If, in addition, the field extension $L/K$ is normal then  $G(L/M)=G(L/K)^M$.

\end{enumerate}

\end{lemma}

\begin{proof} 1.  It follows from the definition of the algebra of differential operators and the inclusions
$L\subseteq E(L/M)\subseteq E(L/K)$ that 
$\CD (L/M) =E(L/M)\cap \CD (L/K)\subseteq  E(L/K)$. Clearly, $E(L/M)\cap \CD (L/K)=\CD (L/K)^M$.

2. Statement 2 follows from the normality of the field extensions $L/K$ and $L/M$ and the inclusion $G(L/M)\subseteq G(L/K)$. 
\end{proof}

 \begin{theorem}\label{1Jun25}%\marginpar{1Jun25}
Let $L/K$ be a normal finite field extension of characteristic $p$. Then:

\begin{enumerate}

\item $ \CA (E(L/K),L)=\{ C_{E(L/K)}(M)=\CD (L/M)\rtimes G(L/M)=\CD (L/K)^M\rtimes G(L/K)^M\, | \, M\in \CF (L/K)\}$. 

\item  {\sc (Analogue of the  Galois correspondence for subfields of a normal  field extension)}

The map
$$
\CF (L/K)\ra \CA (E(L/K),L), \;\; M\mapsto  C_{E(L/K)}(M)=\CD (L/M)\rtimes G(L/M)=\CD (L/K)^M\rtimes G(L/K)^M
$$ is a bijection with inverse
\begin{eqnarray*}
A\mapsto C_{E(L/K)}(A)&=&C_{L}( A\cap \CD (L/K)_+)\cap L^{A\cap G(L/K)}=\Big(C_{L}( A\cap \CD (L/K)_+)) \Big)^{A\cap G(L/K)}\\
&=& L^{A\cap \CD (L/K)_+}\cap L^{A\cap G(L/K)}=\Big(L^{A\cap \CD (L/K)_+} \Big)^{A\cap G(L/K)}.\\
 \bigg( A=\CD (L/M)&\rtimes &G(L/M)\mapsto 
 L^{\CD (L/M)_+}\cap L^{G(L/M)}=\Big(L^{\CD (L/M)_+} \Big)^{G(L/M}.\bigg)
\end{eqnarray*}
%$$ A\mapsto C_{E(L/K)}(A)=L^{A\cap \CD (L/K)_+}\cap L^{A\cap G(L/K)}=\Big(L^{A\cap \CD (L/K)_+} \Big)^{A\cap G(L/K)}
%$$

\item $\mL G(L/K)=\Big\{ \Big(\CD (L/M)_+, G(L/M)\Big) \, | \, M\in \CF (L/K)\Big\}$ and for all fields $M\in \CF (L/K)$, 
$$
\Big(\CD (L/M)_+, G(L/M) \Big)=\Big(\CD (L/ML^{gal})_+, G(L/ML^{pi}) \Big)\in \mL (L/L^{gal})\times G(L/L^{pi}).
$$
 In particular, 
$$
\mL G(L/K)=\Big\{ (\CG , H)\in \mL (L/L^{gal})\times \CG (G(L/L^{pi}))\, | \, h\CG h^{-1}=\CG\;\; {\rm  for\; all\; elements}\;\; h\in H\}.
$$
%The restriction map $G(L/L^{pi})\ra G(L^{gal}/K)$, $g\mapsto g|_{L^{gal}}: L^{gal}\ra L^{gal}$ is a group  isomorphism.

\item  {\sc (Analogue of the  Galois correspondence for subfields of a normal  field extension)}

 The map
$$
\CF (L/K)\ra \mL G(L/K),  \;\; M\mapsto  \Big(\CD (L/M)_+, G(L/M) \Big)=\Big(\CD (L/ML^{gal})_+, G(L/ML^{pi}) \Big)
$$ 
 is a bijection with inverse
$ (\CG , H)\mapsto L^\CG\cap L^H$ and $h(L^\CG)=L^\CG$ for all $h\in H$.
%The restriction map $G(L/L^{pi})\ra G(L^{gal}/K)$, $g\mapsto g|_{L^{gal}}: L^{gal}\ra L^{gal}$ is a group  isomorphism.

\item For all fields $M\in \CF (L/K)$, $[M:K]|G(L/M)|\dim_K(\CD (L/M))=[L:K]^2$  or, equivalently, $|G(L/M)|\dim_M(\CD (L/M))=[L:M]^2$.

\item The field extension $L/L^{gal}$ is a finite purely inseparable field extension and 
$\mL (L/L^{gal})=\{ \CD (L/N)_+\, | \, N\in \CF (L/L^{gal})\}$.

\end{enumerate}

\end{theorem}

\begin{proof} 1. By  Theorem \ref{B24Mar25}.(3),
 $ \CA (E(L/K),L)=\{ C_{E(L/K)}(M)=E(L/M)\, | \, L\in \CF (L/K)\}$.  The finite field extension $L/K$ is  normal, hence so is the field extension  $L/M$ for each field $M\in \CF (L/K)$ and the field extension  $L/M$ is a B-extension (Theorem \ref{VVB-30Apr25}). Therefore, 
$$
E(L/M)=\CD (L/M)\rtimes G(L/M).
$$
By Lemma \ref{c1Jun25}, $\CD (L/M)\rtimes G(L/M)=\CD (L/K)^M\rtimes G(L/K)^M$, and statement 1 follows.

2.  In view of Theorem \ref{B24Mar25}.(3) and statement 1, it suffices that show that all the equalities for the inverse map of the theorem hold.
 Let $A\in \CA (E(L/K),L)$. By statement 1,  $A= \CD (L/M)\rtimes G(L/M)$ for a unique field  $M\in \CF (L/K)$. The finite field extension $L/K$ is  normal, hence $L/K$ is a B-extension (Theorem \ref{VVB-30Apr25}), and so $E(L/K)=\CD (L/K)\rtimes G(L/K)$. It follows from the inclusions $G(L/M)\subseteq G(L/K)$, $\CD (L/M)\subseteq \CD (L/K)$ (Lemma \ref{c1Jun25}) and 
$$
A= \CD (L/M)\rtimes G(L/M)\subseteq E(L/K)=\CD (L/K)\rtimes G(L/K)=\bigoplus_{g\in G(L/K)}\CD (L/K)g
$$
that $\CD (L/M)=A\cap  \CD (L/K)$ and $G(L/M)=A\cap   G(L/K)$. Notice that 
$$ 
L\oplus \CD (L/M)_+=\CD (L/M)=A\cap  \CD (L/K)=A\cap  (L\oplus \CD (L/K)_+)=L\oplus A\cap  \CD (L/K)_+
$$
Therefore, $\CD (L/M)_+=A\cap  \CD (L/K)_+$. Now,
\begin{eqnarray*}
C_L(\CD (L/M)_+) &=& C_L(L\oplus \CD (L/M)_+)=C_L(\CD (L/M))=\End_{\CD (L/M)}(L)\\
&\stackrel{{\rm Thm.}\, \ref{AC24Mar25}.(5)}{=}&(L/M)_{dif}\stackrel{{\rm Thm.}\, \ref{23May25}}{=}L^{\CD (L/M)_+}=L^{A\cap \CD (L/K)_+}
\end{eqnarray*}
and 
\begin{eqnarray*}
 C_{E(L/K)}(A)&=& C_{E(L/K)}(L)\cap C_{E(L/K)}(\CD (L/M)_+)\cap C_{E(L/K)}( G(L/M))\\
 &=& L\cap C_{E(L/K)}(\CD (L/M)_+)\cap C_{E(L/K)}( G(L/M))\\
 &=&  C_{L}(\CD (L/M)_+)\cap L^{G(L/M)}= \Big(C_{L}( \CD (L/M)_+)) \Big)^{ G(L/M)}\\
  &=& L^{\CD (L/M)_+}\cap L^{G(L/M)}
 =\Big(L^{\CD (L/M)_+} \Big)^{G(L/M)}\\ 
 &=& C_{L}( A\cap \CD (L/K)_+)\cap L^{A\cap G(L/K)}=\Big(C_{L}( A\cap \CD (L/K)_+)) \Big)^{A\cap G(L/K)}\\
 &=& L^{A\cap \CD (L/K)_+}\cap L^{A\cap G(L/K)}
 =\Big(L^{A\cap \CD (L/K)_+} \Big)^{A\cap G(L/K)}.
\end{eqnarray*}

The equality $  C_{L}(\CD (L/M)_+)\cap C_{L}(G(L/M))= \Big(C_{L}( \CD (L/M)_+)) \Big)^{ G(L/M)}$
follows from the fact that the set
 $\CD (L/M)_+$ is $G(L/M)$-stable. The last two equalities are the two equalities above the where the sets $\CD (L/M)_+$ and $G(L/M)$ are replaced by the sets $A\cap \CD (L/L)_+$ and $A\cap G(L/K)$, respectively.

6.  Since $L=L^{pi}\t L^{gal}$, the field extension $L/L^{gal}$ is a finite purely inseparable field extension. By Theorem \ref{D24Mar25}.(3), 
$\mL (L/L^{gal})=\{ \CD (L/N)_+\, | \, N\in \CF (L/L^{gal})\}$.

3. (i) {\em For all fields} $M\in \CF (L/K)$,  $
\Big(\CD (L/M)_+, G(L/M) \Big)=\Big(\CD (L/ML^{gal})_+, G(L/ML^{pi}) \Big)$: Recall that 
$G(L/M) = G(L/ML^{pi})$. The equality $\CD (L/M)_+=\CD (L/ML^{gal})_+$ follows from the equalities: 
$$ 
L\oplus \CD (L/M)_+=\CD (L/M)=\CD (L/ML^{gal})=L\oplus \CD (L/ML^{gal})_+,
$$

(ii)  $\CR :=\Big\{ \Big(\CD (L/M)_+, G(L/M)\Big) \, | \, M\in \CF (L/K)\Big\}\subseteq \mL G(L/K)$: We have to show that $\Big(\CD (L/M)_+, G(L/M)\Big)\in \mL G (L/K)$ for all fields $M\in \CF (L/K)$. By Lemma \ref{cD24Mar25}.(2,3), 
$\CD (L/M)_+\in \mL (L/K)$. Clearly, for all automorphisms $g\in G(L/M)$, 
$$
g\CD (L/M)_+g^{-1}=\CD (L/M)_+
$$  since for all $\d\in \CD (L/M)_+$, $g\d g^{-1}(1)=g\d(1)=g(0)=0$, and the statement (ii) follows.

(iii) {\em For every element} $(\CG , H)\in \mL  G(L/K)$, $(\CG , H)=(\CD (L/M_{\CG, H})_+, G(L/M_{\CG, H})$ {\rm where} $M_{\CG, H}:=L^\CG \cap L^H\in \CF (L/K)$: For the pair   $(\CG , H)\in \mL  G(L/K)$, let  $A=A(\CG , H)$ be the  subalgebra of $E(L/K)$ that is generated by the field $L$, $\CG$ and $H$. Therefore, $A\in \CA (E(L/K),L)$ and (where $E:= E(L/K)$)
\begin{eqnarray*}
C_E(A(\CG , H)) &=& C_E(L)\cap C_E(\CG)\cap C_E(H)=L\cap C_E(\CG)\cap C_E(H)\\
 &=&  C_L(\CG)\cap C_L(H) =L^\CG\cap L^H=M_{\CG, H}
\end{eqnarray*}
where we have used the facts that $C_E(L)=L$ and $C_L(\CG)=L^\CG$ (since $\CG \in \mL (L/K)$). By statement 2, 
$$
A(\CG , H)=\CD (L/M_{\CG, H})\rtimes G(L/M_{\CG, H}). 
$$
Recall that $\CD (L/K)=\CD (L/L^{gal})$. Therefore, $\CD (L/K)_+=\CD (L/L^{gal})_+$.

Recall that  $\CG \subseteq  \CD (L/K)_+=\CD (L/L^{gal})_+$ and the finite field extension $L/L^{gal}$ is purely inseparable. 
 By Lemma \ref{cD24Mar25}.(2,3), $\CG = \CD (L/\CM)_+$ for some field $\CM\in \CF (L/L^{gal})$. Since $(\CG, H)\in \mL G(L/K)$, $h\CD (L/\CM)_+h^{-1}=\CD (L/\CM)_+$ for all elements $h\in H$. Therefore,
$$ 
A(\CG , H)=\CD (L/\CM)\rtimes H.
$$
Comparing  the  above two presentations of the algebra  $A(\CG , H)$ and using the inclusion $A(\CG , H)\subseteq \CD (L/K)\rtimes G(L/K)$, we have the equalities 
$$
\CD (L/M_{\CG, H})= \CD (L/\CM)\; \; {\rm  and}\;\; G(L/M_{\CG, H})=H.
$$
The first equality yields 
 the equality $\CD (L/M_{\CG, H})_+= \CD (L/\CM)_+=\CG$, and the statement (iii) follows.

(v)  $\CR =\mL G(L/K)$: The equality  follows from  the statements (ii) and (iii).

(vi) {\em For all fields} $M\in \CF (L/K)$, 
$$
\Big(\CD (L/M)_+, G(L/M) \Big)=\Big(\CD (L/ML^{gal})_+, G(L/ML^{pi}) \Big)\in \mL (L/L^{gal})\times G(L/L^{pi}):
$$
The statement (vi) follows from the statements (i) and (v).

(vii)  $ \mL G(L/K)=\Big\{ (\CG , H)\in \mL (L/L^{gal})\times G(L/L^{pi})\, | \, h\CG h^{-1}=\CG\;\; {\rm  for\; all\; elements}\;\; h\in H\}$: The statement (vii) follows from the statements (v) and (vi). 

4. Statement 4 follows from statements 2 and 3.

5. For $M\in \CF (L/K)$,  the equality $[M:K]|G(L/M)|\dim_K(\CD (L/M))=[L:K]^2$ follows from   Theorem \ref{A24Mar25}.(3) and the equality $C_{E(L/K)}(M)=\CD (L/M)\rtimes G(L/M)$ (statement 1). Now, the equality 
$$|G(L/M)|\dim_M(\CD (L/M))=[L:M]^2$$ follows from the equalities $\dim_K (\CD (L/K))=[M:K]\dim_M (\CD (L/M))$ and $[L:K]=[L:M][M:K]$:
\begin{eqnarray*}
|G(L/M)|\dim_M(\CD (L/M))&=& [M:K]^{-2}\bigg([M:K]|G(L/M)|\dim_K(\CD (L/M))\bigg)= [M:K]^{-2}[L:K]^2\\
&=&[L:M]^2.
\end{eqnarray*}
\end{proof}

\begin{corollary}\label{a1Jun25}%\marginpar{a1Jun25}
Let $L/K$ be a normal finite field extension of characteristic $p$. Then for an algebra $A\in  \CA (E(L/K),L)$,
$$
C_{L}( A\cap \CD (L/K)) = C_{L}( A\cap \CD (L/K)_+) = L^{A\cap \CD (L/K)_+}.
$$
\end{corollary}

\begin{proof} The equalities were proven in the 
the proof of Theorem \ref{1Jun25}.
\end{proof}

Lemma \ref{d1Jun25} provides additional information on  the algebra $\CD (L/M)^K$ in Theorem \ref{1Jun25}.

\begin{lemma}\label{d1Jun25}%\marginpar{d1Jun25}
Let $L/K$ be a normal finite field extension of characteristic $p$. Then $\CD (L/M)=\CD (L/K)^M=C_{E(L/ML^{gal})}(M)=E(L/L^{gal}M)=\CD (L/L^{gal}M)$.
\end{lemma}

\begin{proof} In this proof the following facts are freely used: $L_{dif}=L^{sep}$ (Theorem \ref{23May25}), $\CD (L/K)=\CD (L/L^{sep})$ (Theorem \ref{AC24Mar25}.(1)) and  $L^{gal}=L^{sep}$ since the field extension $L/K$ is normal  (Corollary \ref{a25May25}.(1)). 

(i) $\CD (L/M)=\CD (L/K)^M$: The equality follows from Theorem \ref{1Jun25}.(2). 

(ii)  $\CD (L/M)=\CD (L/ML^{gal})$: The inclusions $M\subseteq ML^{sep}\subseteq L$ yield the  inclusions:
$$
\CD \Big(L/(L/M)^{sep}\Big)=\CD \Big(L/(L/M)_{dif}\Big)=\CD (L/M)\supseteq \CD (L/ML^{sep})\supseteq \CD (L/(L/M)^{sep}),
$$
and the statement (ii) follows

(iii) $E(L/L^{sep}M)=\CD (L/L^{sep}M)$: Since tht field extension $L/L^{gal}M$ is purely inseparable, the statement (iii) follows from (Theorem \ref{C24Mar25}.(1)).

(iv)  $C_{E(L/ML^{gal})}(M)=E(L/L^{gal}M)$: The equality is obvious. 
\end{proof}

%%%%%%%%%%%%%  direct proof of Cor. \ref{a1Jun25}   %%%
%\begin{proof} Let $M\in \CF (L/K)$ and $A=C_{E(L/K)}(M)=\CD (L/K)^M\rtimes G(L/K)^M
%$. Then 
%\begin{eqnarray*}
%\CD (L/K)^M&=& A\cap \CD (L/K),\\
%C_L(A\cap \CD(L/K))&= & C_L\Big(A\cap ( L\oplus \CD(L/K)_+)\Big)= C_L(L)\cap C_L( L\oplus A\cap \CD(L/K)_+)\\
%&=&L\cap C_L( A\cap  \CD(L/K)_+) =C_L(A\cap \CD(L/K)_+).
%\end{eqnarray*}
%So, $C_L(A\cap \CD(L/K)) =C_L(A\cap \CD(L/K)_+)$.
%By Theorem \ref{1Jun25},  
%$$
%\CD (L/K)^M=C_{E(L/L^{gal})}(M)=E(L/L^{gal}M)=\CD (L/L^{gal}M).
%$$
% These equalities are used in the arguments below:
%\begin{eqnarray*}
%A\cap \CD(L/K)_+&=& \Big( A\cap \CD(L/K)\Big) \cap  \CD(L/K)_+=\CD (L/K)^M\cap \CD(L/K)_+\\
%&=&\CD (L/L^{gal}M)\cap \CD(L/K)_+=\CD (L/L^{gal}M)_+,\\
%C_L( A\cap \CD (L/K))&=&C_L\Big(\CD (L/K)^M\Big)=C_L\Big(\CD (L/L^{gal}M)\Big)=\End_{\CD (L/L^{gal}M)}(L)\\
%&=&L^{\CD (L/L^{gal}M)_+}=L^{A\cap \CD (L/K)_+}.
%\end{eqnarray*}
%\end{proof}

{\bf Normal subfields of   normal finite field extensions.}  Theorem  \ref{11Jun25} describes normal sunfields of normal finite field extensions. Also it establishes an order reversing bijection between 
 normal subfields of a  normal finite field extension $L/K$ and the set $\CA (E(L/K),L, G(L/K))$ of $G(L/K)$-stable subalgebras of the algebra $E(L/K)$ that contain the field $L$. Recall that (Corollary \ref{a24Mar25}.(2)), 
 $$
 \CA (E(L/K),L, G(L/K))=\{L\rtimes  N \, | \, N\;\;{\rm is\; a\; normal\; subgroup\; of}\;\; G(L/K) \}.
 $$
For a finite field extension $L/K$, let $\CN (L/K)$ be the set of normal subfields $N/K$ of $L/K$ and 
 $\CN(G(L/K))$ be the set of normal subgroups of the group $ G(L/K)$.

\begin{theorem}\label{11Jun25}%\marginpar{11Jun25}
Let $L/K$ be a normal finite field extension of characteristic $p$. Then:

\begin{enumerate}

\item For each field $M\in  \CN (L/K)$, $M=M^{pi}\t M^{gal}$, where $M^{pi}=M\cap L^{pi}
$ and $M^{gal}:=M\cap L^{gal}$, and 
$$\CN (L/K)=\CF (L^{pi}/K)\t \CN (L^{gal}/K)
:=\{ N\t \G\, | \, N\in \CF (L^{pi}/K),\,  \G\in \CN (L^{gal}/K)\}.
$$
\item $\CA (E(L/K),L, G(L/K))=\Big\{\CD (L^{pi}/N)\t \Big(L^{gal}\rtimes G(L^{gal}/\G) \Big)= E(L^{pi}/N)\t  E(L^{gal}/M^{gal})\, | \,$ $ N\in \CF (L^{pi}/K),\,  \G\in \CN  (L^{gal}/K)  \Big\}$.

\item {\sc (Analogue of the  Galois correspondence for normal  subfields of a normal  field extension)} 

The map
\begin{eqnarray*}
\CN (L/K)&\ra & \CA (E(L/K),L, G(L/K)), \;\; M\mapsto  C_{E(L/K)}(M)=C_{E(L^{pi}/K)}(M^{pi})\t C_{E(L^{gal}/K)}(M^{gal})\\
&=&E(L^{pi}/M^{pi})\t E(L^{gal}/M^{gal})=\CD (L^{pi}/M^{pi})\t \Big(L^{gal}\rtimes G(L^{gal}/M^{gal}) \Big)
\end{eqnarray*}
is a bijection with inverse
\begin{eqnarray*}
 A&\mapsto& C_{E(L/K)}(A)=\Big( L^{pi}\Big)^{A\cap \CD (L^{pi}/K)_+}\t \Big( L^{gal}\Big)^{A\cap G(L^{gal}/K)}.\\
 \bigg(A&=&\CD (L^{pi}/N)\t \Big(L^{gal}\rtimes G(L^{gal}/\G) \Big) \mapsto   \Big( L^{pi}\Big)^{ \CD (L^{pi}/N)_+}\t \Big( L^{gal}\Big)^{ G(L^{gal}/\G)}.\bigg)
\end{eqnarray*}

%$$ A\mapsto C_{E(L/K)}(A)=\Big( L^{pi}\Big)^{A\cap \CD (L^{pi}/K)_+}\t \Big( L^{gal}\Big)^{A\cap G(L^{gal}/K)}.
%$$

\item {\sc (Analogue of the  Galois correspondence for normal  subfields of a normal  field extension)} 

The map
$$
\CN (L/K)\ra \mL (L^{pi}/K)\times \CN (G(L^{gal}/K)), \;\; M=M^{pi}\t M^{gal}\mapsto \Big(\CD (L^{pi}/M^{pi})_+, G(L^{gal}/M^{gal})\Big)
$$
 is a bijection with inverse $(\CG, H)\mapsto \Big(L^{pi} \Big)^\CG\t \Big( L^{gal}\Big)^H$.

\end{enumerate}
\end{theorem}

\begin{proof} 1. Let $M\in  \CN (L/K)$. By  Corollary \ref{a25May25}.(1),  $M=M^{pi}\t M^{gal}$. It follows from the equality $L=L^{pi}\t L^{gal}$
 and the inclusions $K\subseteq M\subseteq L$, that 
 $M^{pi}=M\cap L^{pi}
\in \CF (L^{pi}/K)$ and $M^{gal}:=M\cap L^{gal}\in \CN (L^{gal}/K)$. Therefore, 
$\CN (L/K)\subseteq \CF (L^{pi}/K)\t \CN (L^{gal}/K)$.
The reverse inclusion is obvious. 

2 and 3. By Theorem \ref{B24Mar25}.(3), the map 
$$\CN (L/K)\ra \CA ( E(L/K), L, G(L/K)), \;\; M\mapsto C_{E(L/K)}(M)=E(L/M)$$  is a bijection with inverse $A\mapsto C_{E(L/K)}(A)$. 
 The equality 
 $$
 E(L/K)=E(L^{pi}/K)\t E(L^{gal}/K)
 $$ 
   follows from the equality $L=L^{pi}\t L^{gal}$. For each normal field extension $M/K\in \CN (L/K)$, 
$M=M^{pi}\t M^{gal}$ where $M^{pi}\subseteq  E(M^{pi}/K)$ and $ M^{gal}\subseteq E(M^{gal}/K)$. Now,
\begin{eqnarray*}
 C_{E(L/K)}(M)&=&C_{E(L^{pi}/K)\t E(L^{gal}/K)}(M^{pi}\t M^{gal})\\
 &= &  C_{E(L^{pi}/K)}(M^{pi})\t C_{E(L^{gal}/K)}(M^{gal})\\
&=&E(L^{pi}/M^{pi})\t E(L^{gal}/M^{gal})\\
&\stackrel{{\rm Thm.} \,\ref{C24Mar25}.(1) ,\,{\rm Thm.} \,\ref{B24Mar25}.(2)}{=} &\CD (L^{pi}/M^{pi})\t \Big(L^{gal}\rtimes G(L^{gal}/M^{gal}) \Big).
\end{eqnarray*}
Therefore,
 $\CA (E(L/K),L, G(L/K))=\Big\{\CD (L^{pi}/N)\t \Big(L^{gal}\rtimes G(L^{gal}/\G) \Big)= E(L^{pi}/N)\t \Big(L^{gal}\rtimes G(L^{gal}/M) \Big)\, | \,$ $ N\in \CF (L^{pi}/K),\,  \G\in \CN  (L^{gal}/K)  \Big\}$, and statement 2 and the first part of statement 3  follow, i.e. the map 
\begin{eqnarray*}
\CN (L/K)&\ra & \CA (E(L/K),L, G(L/K)), \;\; M\mapsto  C_{E(L/K)}(M)=C_{E(L^{pi}/K)}(M^{pi})\t C_{E(L^{gal}/K)}(M^{gal})\\
&=&E(L^{pi}/M^{pi})\t E(L^{gal}/M^{gal})=\CD (L^{pi}/M^{pi})\t \Big(L^{gal}\rtimes G(L^{gal}/M^{gal}) \Big)
\end{eqnarray*}
 is a bijection. Since for each normal  field extension $M/K\in \CN (L/K)$,  $M=M^{pi}\t M^{gal}$, where $M^{pi}=M\cap L^{pi} \in E(M^{pi}/K)$ and 
 $ M^{gal}=M\cap L^{gal}\in E(M^{gal}/K)$,  and
  \begin{eqnarray*}
 M^{pi}&=&C_{E(L^{pi}/K)}(E(L^{pi}/M^{pi}))=C_{E(L^{pi}/K)}(\CD (L^{pi}/M^{pi})),\\
 M^{gal}&=&C_{E(L^{gal}/K)}(E(L^{gal}/M^{gal}))=C_{E(L^{gal}/K)}\Big(L^{gal}\rtimes G(L^{gal}/M^{gal})\Big),\\
 \CD (L^{pi}/M^{pi})&=& A\cap \CD (L^{pi}/K)  ,\\
 L^{gal}\rtimes G(L^{gal}/M^{gal})&=& A\cap \Big( L^{gal}\rtimes G(L^{gal}/K)\Big)= L^{gal}\rtimes A\cap G(L^{gal}/K),
 \end{eqnarray*}
 where $A=\CD (L^{pi}/M^{pi})\t \Big(L^{gal}\rtimes G(L^{gal}/M^{gal}) \Big)=C_{E(L/K)}(M)$, 
  the inverse map (Theorem \ref{1Jun25})
$$ \CA (E(L/K),L, G(L/K))\ra \CN (L/K), \;\; A\mapsto C_{E(L/K)}(A)$$
is given by the rule 
\begin{eqnarray*}
A&\ra& C_{E(L/K)}(A)= M=M^{pi}\t M^{gal}=C_{E(L^{pi}/K)}(\CD (L^{pi}/M^{pi}))\t C_{E(L^{gal}/K)}\Big(L^{gal}\rtimes G(L^{gal}/M^{gal})\Big)\\
&=&C_{E(L^{pi}/K)}(A\cap \CD (L^{pi}/K))\t C_{E(L^{gal}/K)}\Big(  L^{gal}\rtimes A\cap G(L^{gal}/K)\Big)\\
%&=&C_{E(L^{pi}/K)\t E(L^{gal}/K) }(\CD (L^{pi}/M^{pi})) \t \Big(L^{gal}\rtimes G(L^{gal}/M^{gal})\Big)=C_{E(L/K)}(A)\\
&=& C_{E(L^{pi}/K)}\Big(L^{pi}\oplus A\cap \CD (L^{pi}/K)_+\Big)\t \Big( L^{gal}\Big)^{A\cap G(L^{gal}/K)}\\
&=& \Big( L^{pi}\Big)^{A\cap \CD (L^{pi}/K)_+}\t \Big( L^{gal}\Big)^{A\cap G(L^{gal}/K)}.
\end{eqnarray*}

4. Statement 4 follows from statement 3 (see the last line of statement 3) and Theorem \ref{D24Mar25}.(3.4). 
\end{proof}

 %%%%%%%%%%%%%%%%%% SECTION 8   %%%%%%%%%%%%%%%%%%%%%
 
 \section{Field extensions $L/L^{G(L/K)}_{dif}$ and their properties} \label{LLGKL} %\marginpar{LLGKL}
 
In this section, we apply the results of the previous sections to  field extensions  $L/L^{G(L/K)}_{dif}$ where $L/K$ are finite field extension of characteristic $p$. We show that the  field extensions  $L/L^{G(L/K)}_{dif}$ are B-extensions, hence normal (Theorem \ref{9Jun25}). Corollary \ref{a9Jun25} is about  the structure of the field extension $L/L^{G(L/K)}_{dif}$. In particular, it shows that   
$$\Big( L/L^{G(L/K)}_{dif}\Big)^{pi}= L^{G(L/K)}]
\;\; {\rm  and }\;\;\Big( L/L^{G(L/K)}_{dif}\Big)^{gal}=\Big( L/L^{G(L/K)}_{dif}\Big)^{sep}=\Big( L/L^{G(L/K)}_{dif}\Big)_{dif}=L_{dif}.
$$
In prime characteristic, we prove that  $(L/K)_{nor}=L^{G(L/K)}_{dif}$ (Theorem \ref{A9Jun25}) and in characteristic zero, $(L/K)_{nor}=L^{G(L/K)}$  (Proposition \ref{B9Jun25}). \\

{\bf The field extension $L/L^{G(L/K)}_{dif}$  is a B-extension/a normal field extension.}  Theorem  \ref{9Jun25} shows that the  field extension $L/L^{G(L/K)}_{dif}$  is a B-extension.

 \begin{theorem}\label{9Jun25}%\marginpar{9Jun25}
Let $L/K$ be a  finite field extension of characteristic $p$. Then:

\begin{enumerate}

\item The field extension $L/L^{G(L/K)}_{dif}$   is a B-extension, hence normal (by Theorem \ref{VVB-30Apr25}). 

\item $\CD (L/K)\rtimes G(L/K)=\CD \Big(L/L^{G(L/K)}_{dif}\Big)\rtimes G\Big(L/L^{G(L/K)}_{dif}\Big)=E\Big(L/L^{G(L/K)}_{dif}\Big)$. 

\item $\CD (L/K)=\CD \Big(L/L^{G(L/K)}_{dif}\Big)$.

\item $ G(L/K) =G\Big(L/L^{G(L/K)}_{dif}\Big)$.

%\item $\Big( L/L^{G(L/K)}_{dif}\Big)_{dif}=L^{G(L/K)}_{dif}$.
\end{enumerate}

\end{theorem}

\begin{proof} 2--4.  Statement 3 follows from  Lemma \ref{a26Apr25}.(1) and statement 4 is obvious. 
 Now, by Theorem \ref{AB17Apr25}.(2),
 $$  
 \CD \Big(L/L^{G(L/K)}_{dif}\Big)\rtimes G\Big(L/L^{G(L/K)}_{dif}\Big)=\CD (L/K)\rtimes G(L/K)=E\Big(L/L^{G(L/K)}_{dif}\Big).
 $$ 

1. By  statement 2, the field extension $L/L^{G(L/K)}_{dif}$   is a B-extension, hence normal, by Theorem \ref{VVB-30Apr25}. 
\end{proof}

%5. Since $K\subseteq L^{G(L/K)}_{dif}$, we have the inclusion $E\Big(L/L^{G(L/K)}_{dif}\Big)\subseteq E(L/K)$. Now, 
%\begin{eqnarray*} 
% L^{G(L/K)}_{dif} &:=& C_{E(L/K)} ( \CD (L/K)\rtimes G(L/K))\stackrel{ {\rm st.}\, 2,3}{=}C_{E(L/K)} \Big( \CD \Big(L/L^{G(L/K)}_{dif}\Big)\rtimes G\Big(L/L^{G(L/K)}_{dif}\Big)\Big)\\
% &\supseteq & C_{E\Big(L/L^{G(L/K)}_{dif}\Big)} \Big( \CD \Big(L/L^{G(L/K)}_{dif}\Big)\rtimes G\Big(L/L^{G(L/K)}_{dif}\Big)\Big)\\
% &=:& \Big( L/L^{G(L/K)}_{dif}\Big)_{dif}\supseteq L^{G(L/K)}_{dif}.
%\end{eqnarray*}

%Therefore, $\Big( L/L^{G(L/K)}_{dif}\Big)_{dif}=L^{G(L/K)}_{dif}$.

Corollary \ref{a9Jun25} clarifies the structure of the field extension $L/L^{G(L/K)}_{dif}$.

 \begin{corollary}\label{a9Jun25}%\marginpar{a9Jun25}
Let $L/K$ be a  finite field extension of characteristic $p$. Then:

\begin{enumerate}

\item $L=\Big( L/L^{G(L/K)}_{dif}\Big)^{pi}\t_{L^{G(L/K)}_{dif}}\Big( L/L^{G(L/K)}_{dif}\Big)^{gal}=L^{G(L/K)}\t_{L^{G(L/K)}_{dif}}   L_{dif}$.

\item $\Big( L/L^{G(L/K)}_{dif}\Big)^{pi}=L^{G\Big(L/L^{G(L/K)}_{dif}\Big)}= L^{G(L/K)}$.

\item  $\Big( L/L^{G(L/K)}_{dif}\Big)^{gal}=\Big( L/L^{G(L/K)}_{dif}\Big)^{sep}=\Big( L/L^{G(L/K)}_{dif}\Big)_{dif}=L_{dif}$.

\end{enumerate}

\end{corollary}

\begin{proof}  By Proposition \ref{AC17Apr25}.(6),  
 $L=L^{G(L/K)}\t_{L^{G(L/K)}_{dif}}   L_{dif}$. 
 By Theorem \ref{9Jun25}.(1), the field extension $L/L^{G(L/K)}_{dif}$   is a B-extension, hence normal (by Theorem \ref{VVB-30Apr25}). Therefore, $L=\Big( L/L^{G(L/K)}_{dif}\Big)^{pi}\t_{L^{G(L/K)}_{dif}}\Big( L/L^{G(L/K)}_{dif}\Big)^{gal}$, by Theorem \ref{VVB-30Apr25}.(6). Now, by statements 2 and 3,  
$$\Big( L/L^{G(L/K)}_{dif}\Big)^{pi}\t_{L^{G(L/K)}_{dif}}\Big( L/L^{G(L/K)}_{dif}\Big)^{gal}=L^{pi}\t_{L^{G(L/K)}_{dif}}   L^{gal}.$$
So, it remains to show that statements 2 and 3 hold: 
\begin{eqnarray*}
\Big( L/L^{G(L/K)}_{dif}\Big)^{pi}& \stackrel{ {\rm Thm.}\, \ref{VVB-30Apr25}.(a)}{=} &L^{G\Big(L/L^{G(L/K)}_{dif}\Big)}\stackrel{ {\rm Thm.}\, \ref{9Jun25}.(4)}{=}L^{G(L/K)},\\
 \Big( L/L^{G(L/K)}_{dif}\Big)^{gal}
 &\stackrel{ {\rm Thm.}\, \ref{VVB-30Apr25}.(b)}{=}&
 \Big( L/L^{G(L/K)}_{dif}\Big)^{sep}\stackrel{ {\rm Thm.}\, \ref{VVB-30Apr25}.(b)}{=}\Big( L/L^{G(L/K)}_{dif}\Big)_{dif}\stackrel{ {\rm Lem.}\, \ref{a26Apr25}.(1)}{=} L_{dif}.
\end{eqnarray*}
\end{proof}

{\bf The least co-normal field extension $(L/K)_{nor}$.}
\begin{definition}
For a finite field extension $L/K$ of arbitrary characteristic, a field extension $M/K$ in $L/K$is called a {\bf co-normal subfield} if the field extension $L/M$ is normal. The {\bf least co-normal subfield} of $L/K$ (w.r.t. $\subseteq$) is denoted by $(L/K)_{nor}=L_{nor}$. 
\end{definition}

 Theorem  \ref{A9Jun25} shows that the  field extension $L^{G(L/K)}_{dif}/K$  is the least co-normal field extension of $L/K$ when $\char (K)=p$.

 \begin{theorem}\label{A9Jun25}%\marginpar{A9Jun25}
Let $L/K$ be a  finite field extension of characteristic $p$. Then  $(L/K)_{nor}=L^{G(L/K)}_{dif}/K$. 

\end{theorem}

\begin{proof} By Theorem \ref{9Jun25}, the field extension $L/L^{G(L/K)}_{dif}$ is  normal. Suppose that $N/K$ is a  field extension such that $N\subseteq L$ and  the field extension $L/N$ is normal. We have to show that $L^{G(L/K)}_{dif}\subseteq N$. Notice that $\CD (L/N)\subseteq  \CD (L/K)$ and  $G(L/N)\subseteq G(L/K)$. Now, by Theorem \ref{VVB-30Apr25}.(b),
$$
E(L/N)=\CD (L/N)\rtimes G(L/N)\subseteq \CD (L/K)\rtimes G(L/K)=E\Big( L/L^{G(L/K)}_{dif}\Big).
$$
Therefore, $L^{G(L/K)}_{dif}\subseteq N$, as required. 
\end{proof}

In characteristic zero, Proposition  \ref{B9Jun25} shows that the  field extension $L^{G(L/K)}/K$  is the least co-normal  field extension of $L/K$.

 \begin{proposition}\label{B9Jun25}%\marginpar{B9Jun25}
Let $L/K$ be a  finite field extension of characteristic zero. Then $(L/K)_{nor}=L^{G(L/K)}/K$. 

\end{proposition}

\begin{proof} Notice that all fields of characteristic zero are separable. So, the concepts of `normal' and `Galois' coincide for field extensions in characteristic zero.  By the Galois Theory, the field extension $L/L^{G(L/K)}$ is Galois, hence normal. Suppose that $N/K$ is a  field extension such that $N\subseteq L$ and  the field extension $L/N$ is normal, hence Galois. We have to show that $L^{G(L/K)}\subseteq N$. 
 It follows from the inclusion $G(L/N)\subseteq G(L/K)=G\Big( L/L^{G(L/K)}\Big)$ the inclusion
$$
N=L^{G(L/N)}\supseteq L^{G(L/K)},
$$ 
by the Galois correspondence, as required. 
\end{proof}

 {\bf Licence.} For the purpose of open access, the author has applied a Creative Commons Attribution (CC BY) licence to any Author Accepted Manuscript version arising from this submission.

{\bf Declaration of interests.} The authors declare that they have no known competing financial interests or personal relationships that could have appeared to influence the work reported in this paper.

%{\bf Disclosure statement.} No potential conflict of interest was reported by the author.

{\bf Data availability statement.} Data sharing not applicable – no new data generated.

{\bf Funding.} This research received no specific grant from any funding agency in the public, commercial, or not-for-profit sectors.

\small{

School of Mathematical  and Physical Sciences

Division of Mathematics

University of Sheffield

Hicks Building

Sheffield S3 7RH

UK

email: v.bavula@sheffield.ac.uk}

\end{document}